\documentclass[12pt]{article}

\usepackage{amssymb,amsmath,amsthm,graphicx,geometry,multicol,comment,xcolor,mathtools,wrapfig,tikz-cd,pdfpages}
\usepackage[margin=2cm,font=scriptsize]{caption}

\usepackage{manfnt}

\usepackage[pagebackref,breaklinks=true]{hyperref}\hypersetup{colorlinks,
  linkcolor={green!50!black},
  citecolor={green!50!black},
  urlcolor=blue
}

\graphicspath{{Pictures/}}

\newtheorem{definition}{Definition}
\newtheorem{theorem}[definition]{Theorem}

\newtheorem{conjecture}[definition]{Conjecture}
\newtheorem{lemma}[definition]{Lemma}
\newtheorem{corollary}[definition]{Corollary}

\newcommand{\Z}{\mathbb{Z}}
\newcommand{\Q}{\mathbb{Q}}
\newcommand{\U}{\mathbb{U}}
\newcommand{\Oo}{\mathbb{O}}
\newcommand{\A}{\mathbb{A}}
\newcommand{\D}{\mathbb{D}}
\newcommand{\B}{\mathbb{B}}
\newcommand{\K}{\mathbb{K}}

\newcommand{\pp}{/\hspace{-3pt}/}
\newcommand{\p}{\partial}
\newcommand{\llb}{[\hspace{-1.2pt}[}
\newcommand{\rrb}{]\hspace{-1.2pt}]}
\newcommand{\la}{\langle}
\newcommand{\ra}{\rangle}

\newcommand{\tbb}{\tilde{\mathbf{b}}}
\newcommand{\tby}{\tilde{\mathbf{y}}}

\newcommand{\tZ}{\tilde{Z}}
\newcommand{\btZ}{\tilde{\mathbf{Z}}}

\newcommand{\bZ}{\mathbf{Z}}
\newcommand{\bt}{\mathbf{t}}
\newcommand{\by}{\mathbf{y}}
\newcommand{\bI}{\mathbf{1}}
\newcommand{\bb}{\mathbf{b}}
\newcommand{\ba}{\mathbf{a}}
\newcommand{\bx}{\mathbf{x}}
\newcommand{\bp}{\mathbf{p}}
\newcommand{\bR}{\mathbf{R}}
\newcommand{\bm}{\mathbf{m}}
\newcommand{\bS}{\mathbf{S}}
\newcommand{\bA}{\mathbf{A}}
\newcommand{\bB}{\mathbf{B}}
\newcommand{\bv}{\mathbf{v}}
\newcommand{\bu}{\mathbf{u}}
\newcommand{\bC}{\mathbf{C}}
\newcommand{\bpi}{\boldsymbol{\pi}}

\newcommand{\eps}{\epsilon}

\newcommand{\G}{\mathcal{G}}
\newcommand{\Aa}{\mathcal{A}}

\newcommand{\OO}{\mathcal{O}}
\newcommand{\PG}{\mathcal{PG}}

\newcommand{\tr}{\mathrm{tr}}
\newcommand{\rot}{\mathrm{rot}}
\newcommand{\id}{\mathrm{id}}
\newcommand{\Hom}{\mathrm{Hom}}
\newcommand{\wt}{\mathrm{wt}}
\newcommand{\wh}{\mathrm{wh}}
\newcommand{\Uh}{\overline{U(\mathfrak{h})}}

\newcommand{\gm}{\raisebox{5pt}{$\scriptscriptstyle \G$}\hspace{-3pt}m}
\newcommand{\gS}{\raisebox{5pt}{$\scriptscriptstyle \G$}\hspace{-3pt}S}
\newcommand{\gSbar}{\raisebox{5pt}{$\scriptscriptstyle \G$}\hspace{-3pt}\bar{S}}
\newcommand{\gD}{\raisebox{5pt}{$\scriptscriptstyle \G$}\hspace{-3pt}\Delta}
\newcommand{\gep}{\raisebox{5pt}{$\scriptscriptstyle \G$}\hspace{-3pt}\varepsilon}
\newcommand{\gpi}{\raisebox{5pt}{$\scriptscriptstyle \G$}\hspace{-3pt}\pi}

\newcommand{\qbin}[2]{
\left[\begin{array}{c} #1 \\ #2 \end{array}\right]}

\title{Perturbed Gaussian generating functions for universal knot invariants}
\author{Dror~Bar-Natan and Roland~van~der~Veen}

\begin{document}

\maketitle

\abstract{
We introduce a new approach to universal quantum knot invariants that emphasizes generating functions instead of generators and relations.
All the relevant generating functions are shown to be perturbed Gaussians of the form $Pe^G$, where $G$ is quadratic and $P$ is a suitably restricted "perturbation".  
After developing a calculus for such Gaussians in general we focus on the rank one invariant $\bZ_\D$ in detail. 
We discuss how it dominates the $\mathfrak{sl}_2$-colored Jones polynomials and relates to knot genus and Whitehead doubling.
In addition to being a strong knot invariant that behaves well under natural operations on tangles $\bZ_\D$ is also computable in polynomial time in the crossing number of the knot.
We provide a full implementation of the invariant and provide a table in an appendix. 
}

\tableofcontents

\section{Introduction}

The key idea of this paper is to replace algebra generators by generating functions. This approach is especially effective in dealing with quantum groups and the universal
knot invariants that come with them. With some simple modifications the generating functions for all important operations and elements take the form of
perturbed Gaussians $Pe^G$. We will see that manipulating Gaussians is conceptually and practically simpler than working in terms of generators.

To illustrate what we mean by replacing generators by generating functions consider the universal enveloping algebra $H$ of the Heisenberg algebra.
It is an algebra generated by elements $\bp,\bx$ subject to the relation $\bp\bx-\bx\bp=1$. 
Instead of computing in this algebra directly we precompute all the multiplications of all the monomials $\bp^{k_1}\bx^{n_1}\cdot\bp^{k_2}\bx^{n_2}$ once and for all
and record the results in the following generating function: 
\[
m = \sum_{k_1,k_2,n_1,n_2 = 0}^\infty \bp^{k_1}\bx^{n_1}\bp^{k_2}\bx^{n_2}\frac{\pi^{k_1}\xi^{n_1}\pi^{k_2}\xi^{n_2}}{k_1!k_2!n_1!n_2!}
\]
As so often, the generating function is less complicated than the sequential data it encodes. In this case $m$ has the shape of a Gaussian:
\[m = e^{(\pi_1+\pi_2)\bp}e^{(\xi_1+\xi_2)\bx-\xi_1\pi_2}\] 
The generating function technique works especially well for computing universal knot invariants because there one has to multiply many terms of a similar exponential form
$\bR=e^{\bp\otimes \bx- \bI\otimes \bp \bx}$. Roughly speaking the rule is to assign a copy of $\bR$ to each crossing in the knot diagram with the first factor on the overpass
and the second on the underpass. The universal knot invariant $\bZ$ is then the product of all terms taken in order of appearance as one walks along the knot.
Using a quantum field theory inspired calculus of perturbed Gaussian generating functions we are able to compute the invariant $\bZ$ without writing down a single generator.

Our Gaussian generating function techniques should apply to computations in many algebras including the universal invariants corresponding to any of the Drinfeld-Jimbo quantum groups $\mathcal{U}_\hbar(\mathfrak{g})$ \cite{Oh01}.
For concreteness however we focus on the rank one case in this paper. We show how to build a version $\D$ of quantum $\mathfrak{sl}_2$ and express all structural operations and elements in terms of perturbed Gaussian generating functions.

Expanding our invariant $\bZ_\D$ as a series in an auxiliary parameter $\eps$ we obtain the Alexander polynomial at the $0$-th order and increasingly strong invariants as one computes higher orders in $\eps$. 
Unlike usual quantum knot invariants computations of $\bZ_\D(K)$ can be done in polynomial time in the complexity of the knot $K$, see Theorem \ref{thm.compZD}.
Furthermore, we relate $\bZ_\D$ to the knot genus, the Alexander polynomial and Whitehead doubling, see Theorems \ref{thm.structZ}, \ref{thm.genus} and \ref{thm.WDouble}.  In Theorem \ref{thm.UQ} we show $\bZ_\D$ determines all the $\mathfrak{sl}_2$ colored Jones polynomials. 
We also expect $\bZ_\D$ is equivalent to Rozansky's loop expansion of the colored Jones polynomial \cite{Ro98}. The present work is a continuation of \cite{BV19} in which a simplified version of $\bZ_\D$ appears to first order in $\eps$. 

In summary $\bZ_\D$ is a strong yet computable knot invariant whose topological interpretation is tractable due to its good behaviour under tangle operations, strand doubling and reversal. As such we hope it may play a role
in addressing questions of topological interest.

The plan of the paper is as follows. First we will introduce Gaussian generating functions in the context of linear maps between polynomial rings.
In this setting we develop the key contraction theorem that allows us to compose linear maps by contracting their Gaussian generating functions.
Next we illustrate the main techniques in a section on the Heisenberg algebra and the universal knot invariants related to it. As a preparation for
later sections we introduce a notion of tangles appropriate for our invariants. The calculation of the Heisenberg tangle invariants is fully implemented in Mathematica.
 
Only in Section \ref{sec.MainExample} we set out to construct our main example of a ribbon Hopf algebra $\D$ using the Drinfeld double construction.
We show in detail how to compute in $\D$ using Gaussian generating functions in the next section. After setting up the algebra we extend our notion of
tangles to account for rotation numbers and introduce the universal tangle invariant $\bZ_\D$. 
In Section \ref{sec.structure} we prove our main theorems using a combination of the Gaussian calculus and a Hopf algebraic interpretation of Seifert surfaces.
We end the work with a discussion of the computer demonstration. In the Appendix \ref{sec.Implementation} we list the full implementation of our algorithms and 
our invariant is tabulated for all prime knots up to ten crossings in Appendix \ref{sec.KnotTable}.\\

\noindent{\bf Acknowledgments. } The authors wish to thank Jorge Becerra for careful reading of the manuscript and Thomas Fiedler for hosting a workshop on this subject. D. B. was supported by NSERC grant
RGPIN-2018-04350.

\section{Generating functions for maps between polynomial rings}
\label{sec.GenFunc}

In this section we describe linear maps between polynomial algebras using generating functions. The polynomials we will work 
with\footnote{For concreteness we prefer to work over the rationals although much of what say applies to more general fields of characteristic $0$.}  
are in $\Q[z_J]$ where $J$ is a finite set and the multi-index $z_J$ means $(z_j)_{j\in J}$. Our goal is to investigate
the linear maps between two such vector spaces. Denote the space of linear maps by $\Hom(\Q[z_J],\Q[z_K])$.

In later sections many such linear maps come from choosing a basis in a non-commutative algebra and describing multiplication in terms of
the ordered monomials. Some more elementary examples are the following:
\begin{enumerate}
\item The identity $\id_{\Q[z_J]}\in \Hom(\Q[z_J],\Q[z_J])$
\item The shift map $\mathrm{sh}:\Q[z]\to \Q[z]$ defined by $\mathrm{sh}(p)(z) = p(z+1)$
\item Partial differentiation $\frac{\p}{\p z}:\Q[z,w]\to \Q[z,w]$.
\item Definite integration $\int:\Q[z]\to \Q[z]$ defined by $\int p = \int_0^zp(t)dt$.
\end{enumerate}

The monomials form a basis for $\Q[z_J]$ and so any linear map $\varphi\in\Hom(\Q[z_J],\Q[z_K])$ is determined by its values on all monomials.
For example the partial derivative can be characterized by just knowing $\frac{\p}{\p z} z^kw^\ell = kz^{k-1}w^\ell$.
In general then our linear map $\varphi$ is determined by a sequence of polynomials $\varphi(z_J^n)\in \Q[z_K]$ for every multi-index $n\in \mathbb{N}^J$.

This is where the generating functions come into play. Whenever there is a sequence of numbers or other algebraic objects one can try to
understand its collective behaviour by placing them into a generating function. To encode each of the basis vectors we use an auxiliary variable $\zeta_j$ corresponding to each variable $z_j$.

\begin{definition} {\bf(Generating function of a linear map)}\\
\label{def.G}
Define $\G:\Hom(\Q[z_J],\Q[z_K])\to \Q[z_K]\llb \zeta_J\rrb$ by
\[
\G(\varphi) = \sum_{n\in\mathbb{N}^J} \frac{\zeta_J^n}{n!}\varphi(z_J^n)
\]
Here $n$ is a multi-index $n=(n_j)_{j\in J}$ and $n! = \prod_{j\in J}n_j!$ and $z_J^n = \prod_{j\in J}z_j^{n_j}$.
\end{definition}

Notice that in some sense we are just inserting an exponential into $\varphi$ in the above definition. More formally, extending $\varphi$ to a map $\Q[z_J]\llb \zeta_J\rrb\to \Q[z_K]\llb \zeta_K\rrb$ by treating the $\zeta_J$ as scalars, the exponential generating function may also be expressed as 
\begin{equation} 
\label{eq.Gexp}
\G(\varphi) = \varphi(e^{z_J \zeta_J}) = \varphi(\G(\id_{\Q[z_J]}))
\end{equation} This explains why we are using exponential generating functions.

In our partial differentiation example we take $\omega$ to be the variable corresponding to $w$ and $\zeta$ corresponding to $z$ and find: 
\[
\G(\frac{\p}{\p z}) = \sum_{k,\ell} \frac{\p}{\p z}(z^kw^\ell)\frac{\zeta^k\omega^\ell}{k!\ell!} = \sum_{k,\ell} kz^{k-1}w^\ell\frac{\zeta^k\omega^\ell}{k!\ell!}  =
e^{\omega w}\sum_{k} z^{k-1}w^\ell\frac{\zeta^k}{(k-1)!} = \zeta e^{\omega w+\zeta z}  
\]
The generating functions for the other examples can be found similarly using properties of the exponential function and perhaps Equation \eqref{eq.Gexp}:
\begin{enumerate}
\item $\G(\id_{\Q[z_J]}) = e^{z_J\zeta_J}$, where $z_J\zeta_J = \sum_{j\in J}z_j\zeta_j$.
\item $\G(\mathrm{sh}) = e^{(z+1)\zeta}$
\item $\G(\frac{\p}{\p z}) = \zeta e^{\omega w+\zeta z}$.
\item $\G(\int) = \frac{e^{\zeta z}-1}{\zeta}$.
\end{enumerate}

The polynomial ring is a Hopf algebra and this provides us with a few more examples of linear maps. Many of the generating functions we will use later to describe more general Hopf algebras
have features similar to the simple ones presented here. The Hopf algebra maps and their generating functions are listed below: 
\begin{enumerate}
\item Multiplication $m:\Q[z_1,z_2]\to \Q[z],\quad m(z_1^kz_2^\ell) = z_1^kz_2^\ell$,\quad $\G(m) = e^{(\zeta_1+\zeta_2)z}$.
\item Co-Multiplication $\Delta:\Q[z]\to \Q[z_1,z_2]$,\quad $\Delta(z^k) = (z_1+z_2)^k$,\quad $\G(\Delta) = e^{\zeta(z_1+z_2)}$
\item Antipode $S:\Q[z]\to \Q[z]$,\quad $S(z^k) = (-1)^kz^k$,\quad $\G(S) = e^{-\zeta z}$
\item Unit $\eta:\Q\to \Q[z]$,\quad $\eta(1) = 1$,\quad $\G(\eta) = 1$
\item co-unit $\varepsilon: \Q[z]\to \Q$,\quad $\varepsilon(z^k) = \delta_{k,0}$,\quad $\G(\varepsilon) = 1$
\end{enumerate}

The generating function contains all information about the linear map and so we can move back and forth between the two descriptions.

\begin{lemma} {\bf(The generating function determines the linear map)}\\
\label{lem.GeneratingFunction}
$\G:\Hom(\Q[z_J],\Q[z_K])\to\Q[z_K]\llb \zeta_J\rrb$ is a bijection. The inverse is given by 
\[\G^{-1}(f)(p) = p(\p_{\zeta_J})f(\zeta_J,z_K)|_{\zeta_J=0} = f(\p_{z_J},z_K)p(z_J)|_{z_J=0}\]
where $p\in \Q[z_J]$ and $f\in \Q[z_K]\llb \zeta_J\rrb$.
\end{lemma}
\begin{proof}
Notice that $p(z_J) = p(\p_{\zeta_J})e^{\zeta_J\cdot z_J}|_{\zeta_J = 0}$. Therefore 
\[(\G^{-1}\G(\varphi))(p) = p(\p_{\zeta_J})\varphi(e^{\zeta_J\cdot z_J})|_{\zeta_J=0}=
\varphi( p(\p_{\zeta_J})e^{\zeta_J\cdot z_J}|_{\zeta_J=0})=\varphi(p)\]
Likewise for $f(\zeta_J,z_K)$ we find
\[
\G(\G^{-1}(f)) = (\G^{-1}f)(e^{\zeta_J\cdot z_J}) = f(\p_{z_J},z_K)e^{\zeta_J\cdot z_J}|_{z_J=0} = f(\zeta_J,z_K)
\]
\end{proof}

The advantage of the generating function approach is in describing composition of linear maps. Composition\footnote{In what follows we will often write composition as $f\circ g = g \pp f$.} translates into the following operation on the corresponding generating functions.

\begin{lemma}{\bf(Composition of generating functions)}\\
\label{lem.CompGeneratingFunctions}
Suppose $J,K,L$ are finite sets and $\varphi\in \Hom(\Q[z_J],\Q[z_K])$ and $\psi \in \Hom(\Q[z_K],\Q[z_L])$. We have
\[
\G(\varphi\pp \psi) = \Big(\G(\varphi)|_{z_K\mapsto \p_{\zeta_K}}\G(\psi)\Big)|_{\zeta_K=0}
\]
\end{lemma}
\begin{proof}
Set $\alpha=\G(\varphi)$ and $\beta=\G(\psi)$ then 
\[
\G(\varphi\pp \psi) = e^{\zeta_K\cdot z_K}\pp\varphi\pp \psi = \alpha\pp\psi = \alpha\pp\G^{-1}(\beta) =  \Big(\alpha|_{z_K\mapsto \p_{\zeta_K}}\beta\Big)|_{\zeta_K=0}
\]
using Lemma \ref{lem.GeneratingFunction} in the final equality.
\end{proof}

In categorical terms we can summarize what we found so far by introducing two categories $\tilde{\mathcal{P}}$ and $\tilde{\mathcal{C}}$ whose objects are finite sets 
and whose morphisms are $\Hom_{\tilde{\mathcal{P}}}(J,K) = \Hom(\Q[z_J],\Q[z_K])$ and $\Hom_{\tilde{\mathcal{C}}}(J,K) = \Q[z_K]\llb z_J \rrb$.
Composition in $\tilde{\mathcal{C}}$ is defined as in the previous Lemma: 
\[f\pp g =  \Big(f|_{z_K\mapsto \p_{\zeta_K}}g\Big)|_{\zeta_K=0}\qquad f\in \Hom_{\tilde{\mathcal{C}}}(J,K),\ g\in\Hom_{\tilde{\mathcal{C}}}(K,L)\]
Then the functor $\G:\tilde{\mathcal{P}}\to \tilde{\mathcal{C}}$ is an isomorphism of categories. In fact $\G$ is a monoidal functor if we take $J\otimes K$ to be the disjoint union and
$f\otimes g = fg$ in $\tilde{\mathcal{C}}$.

The composition formula involves differentiating with respect to many variables at the same time, as many as the size of $K$. We view each of the $|K|$ differentiations as a contraction, just like one contracts indices in a tensor.
The benefit is that a single contraction operation (defined precisely below) is simpler to deal with and serves as the fundamental building block for our computations.

\begin{definition}{\bf(Contraction)}\\
We say $f$ is the contraction of $g(r,s) = \sum_{k,\ell}c_{k,\ell}r^ks^\ell$ along the pair of variables $(r,s)$ if 
\[f = \sum_{k}c_{k,k}k! = \sum_{k,\ell}c_{k,\ell}\p_{s}^k s^\ell|_{s=0}\] 
Our notation will be $f=\la g(r,s) \ra_{r,s}$ or simply $f= \la g\ra_r$ when it is clear which variables $s$ correspond to the $r$. 
We also allow $r$ and $s$ to be vectors of variables of equal size
in which case $r_i$ is understood to be contracted with $s_i$, so $\la g \ra_{((r_1,r_2,r_3),(s_1,s_2,s_3))} = \la \la \la g \ra_{(r_1,s_1)} \ra_{(r_2,s_2)} \ra_{(r_3,s_3)}$.
\end{definition} 

So far the variables we used for contractions were $r=\zeta,s=z$. For example $f=42$ is the contraction of $g(r,s) = 7\zeta_1\zeta_2^3z_1z_2^3+3\zeta_1\zeta_2^2z_1^3z_2^2$ with
 $r=(\zeta_1,\zeta_2), s = (z_1,z_2)$: 
\[
\la g(r,s) \ra_{r,s} = \la 7\zeta_1\zeta_2^3z_1z_2^3+3\zeta_1\zeta_2^2z_1^3z_2^2\ra_{(\zeta_1,\zeta_2),(z_1,z_2)} = 42
\]
A more subtle example that is central to our theory is about the Gaussian exponential $g(r,s) = e^{\hbar r s} \in \Q[r,s]\llb \hbar \rrb$. Summing the geometric series shows $f = \frac{1}{1-\hbar}\in \Q\llb \hbar \rrb$ is the contraction of $g$ along the pair $(r,s)$:
\[
\la g \ra_{r,s} = \sum_{n=0}^\infty \p_s^n s^n\frac{\hbar^n}{n!} = \sum_{n=0}^\infty \hbar^n = f
\]
This example also suggests contractions are not always well defined. For example if we instead take $g = e^{rs}\in \Q[r]\llb s \rrb$ then the same calculation yields
 $\la g(r,s) \ra_{r,s} = \sum_{n=0}^\infty 1$.

Coming back to our discussion of composition of linear maps by contraction of their generating functions we arrive at the following. 
If $\varphi\in \Q[z_K]\llb \zeta_J \rrb$ and $\psi\in \Q[z_L]\llb \zeta_K \rrb$ then
\[
\varphi \pp \psi  = \la \varphi \psi\ra_{\zeta_K,z_K}
\]
where $\varphi \psi$ means ordinary multipliciation of the power series in $\Q[z_L,z_K]\llb \zeta_K,\zeta_J \rrb$.

Looking at the above examples we notice they all involve exponentials of quadratic forms, also known as \emph{Gaussians}, possibly with some kind of perturbation. 
Perturbed Gaussians not only appear in all our examples, there is also a concrete formula for composing and contracting them! 

\begin{lemma}{\bf (Contraction Lemma)}\\ 
\label{lem.zip}
For any $n\in \mathbb{N}$ consider the ring $R_n = \Q[r_j, g_j]\llb s_j,\,\ W_{ij},\ f_j|1\leq i,j\leq n \rrb$. We have the following equality in $R_n$:
\[
\la e^{gs+rf +r W s} \ra_{r,s} = \det(\tilde{W})e^{g\tilde{W}f}\qquad \tilde{W} = (1-W)^{-1} 
\]
here $r=(r_i),s=(s_i),g=(g_i),f=(f_i)$ are thought of as vectors of size $n$ and $W=(W_{ij})$ is an $n\times n$ matrix.
\end{lemma}

\begin{proof}
Since we are in a power series ring $\tilde{W}$ is to be taken as an infinite series $\tilde{W}_{ab} = \sum_{k=0}^\infty W^k_{ab}$,
where $W^k_{ab} = \sum_{i_1,\dots i_{k-1}}W_{ai_1}W_{i_1i_2} \dots W_{i_{k-1}b}$.

We start by proving the special case $f=g=0$. Set $Z(\lambda) = \la e^{r\lambda Ws} \ra_{(r,s)}$ and $A(\lambda) = \frac{1}{\det 1-\lambda W}$. 
The result $Z(\lambda)=A(\lambda)$ follows once we show both sides are solutions to the initial value problem $\p_\lambda F(\lambda) =(\tr W(1-\lambda W)^{-1})F$ and $F(0) = 1$. 

For $A(\lambda)$ this is verified using 
$\det e^X = e^{\tr X}$ so
\[\p_\lambda A(\lambda)  =\p_\lambda\frac{1}{\det 1-\lambda W} = \p_\lambda e^{-\tr\log(1-\lambda W)} = (\tr W(1-\lambda W)^{-1})\frac{1}{\det 1-\lambda W}\]
For $Z(\lambda)$ we use the shorthand $\la \varphi \ra = \la \varphi \ra_{(r,s)}$ and claim that for any matrix $M$ of size $n$ with entries in $R_n$ we have
\begin{equation}
\label{eq.zipstep}
\la  r M s e^{r\lambda W s}\ra = \tr(M) Z(\lambda)+ \la \lambda rMWs e^{r\lambda Ws} \ra
\end{equation}
Using Formula \eqref{eq.zipstep} repeatedly, we find a geometric series:
\[\p_\lambda Z(\lambda) = \la rWs e^{r\lambda Ws} \ra = (\tr W)Z(\lambda) + \la r\lambda W^2s e^{r\lambda Ws} \ra=\]
\[
(\tr W)Z(\lambda) + (\tr \lambda W^2)Z(\lambda) + \la r\lambda^2 W^3s e^{r\lambda Ws} \ra = \dots = (\tr W(1-\lambda W)^{-1})Z(\lambda)\]
the error term vanishes because the powers of $W$ converge to $0$ in the topology of $R^n$.
Finally, formula \eqref{eq.zipstep} is verified explicitly by expanding both sides as power series in $r,s$ and carrying out the contraction
in terms of monomials. This proves
$A(\lambda) = Z(\lambda)$ and hence the special case $f=g=0$.

For the general case we follow a similar strategy. Introduce 
\[X(\lambda) = \la e^{rWs+\lambda(gs+rf) } \ra,\quad B(\lambda) = \det(1-W)^{-1} e^{\lambda^2 g(1-W)^{-1}f}\]
and show that $\p_\lambda B(\lambda) = \p_\lambda X(\lambda)$. This is enough since equality at $\lambda = 0$ is precisely the special case $f=g=0$ above.
The key is the following equation for any vectors $F,G$ of length $n$ in $R_n$ (proven in the same way as Equation \eqref{eq.zipstep}):
\begin{equation}
\label{eq.zipstep2}
\la (Gs+rF)e^{rWs+\lambda(gs+rf)} \ra = 2\lambda GF X(\lambda)+\la (GWs+rWF)e^{rWs+\lambda(gs+rf)} \ra
\end{equation}
Applying this equation repeatedly and using the geometric series finishes the proof:
\[\p_\lambda X(\lambda) = \la (gs+rf)e^{rWs+\lambda(gs+rf)} \ra = 2\lambda gfX(\lambda)+\la (gWs+rWf)e^{rWs+\lambda(gs+rf)} \ra =\]\[
2\lambda (gf+gWf)X(\lambda)+\la (gW^2s+rW^2f)e^{rWs+\lambda(gs+rf)} \ra = \dots = 2\lambda g(1-W)^{-1}f X(\lambda) = \p_\lambda B(\lambda)\] 
\end{proof}

For later use we note that we may allow more general perturbations:

\begin{theorem}{\bf (Contraction Theorem)}\\
\label{thm.zip}
With the same notation as in Lemma \ref{lem.zip} and $P\in R_n$ that only depends on $r,s$. We have the following equality:
\[
\la P(r,s)e^{g s + r f  +r Ws} \ra_{r,s}
= \det(\tilde{W})e^{g\tilde{W}f}\la P(r+g\tilde{W},\tilde{W}(s+f))\ra_{r,s} \qquad \tilde{W} = (1-W)^{-1} 
\]
\end{theorem}

\begin{proof}
To derive the theorem from Lemma \ref{lem.zip} above we introduce auxiliary variables $m,\mu$ and write
\[\la P(r,s)e^{g s+  r f +r W s} \ra = P(\p_m,\p_\mu)\la e^{(g+\mu)s + r(f+m) +r W s}\ra|_{m=\mu=0}\]
since these differentiations commute with contraction. Replacing $f$ by $f+m$ and $g$ by $f+\mu$ Lemma \ref{lem.zip} says
\[
\la P(r,s)e^{g s + r f  +r W s} \ra = 
\det(\tilde{W})P(\p_m,\p_\mu)e^{(g+\mu)\tilde{W}(f+m)}|_{m=\mu=0} = 
\]
\[
\det(\tilde{W})P(\p_m,\tilde{W}(f+m))e^{g\tilde{W}(f+m)}|_{m=0} = \det(\tilde{W})e^{g\tilde{W}f}\la P(r+g \tilde{W},\tilde{W}(f+s)) \ra
\]
\end{proof}

Often one can just contract one variable at the time and for that the following simplified version of Theorem \ref{thm.zip} is useful:
\begin{equation}
\label{eq.1zip}
\la P(r,s)e^{c+rf+gs+Wrs}\ra_{r,s} = \la P\big(r+\frac{g}{1-W},\frac{(s+f)}{1-W}\big)\ra_{r,s}\frac{e^{c+\frac{gf}{1-W}}}{1-W}
\end{equation}
However if one needs to bound the degree of the denominators appearing after contraction it is better to use the full power of the theorem instead of repeatedly using the one variable case.

Finally we remark that contraction may also be phrased more symmetrically as 
\[\la g(r,s) \ra_{r,s} = e^{\p_r\p_s}g(r,s)|_{r=s=0}\]
This helps seeing the close relationship with formal Gaussian integration, especially as it is used in physics \cite{Po05,Ab03}. 
In perturbation theory, without paying attention to convergence issues, there is another form for the contraction:
\[
\la g(r,s) \ra_{r,s} \propto \int e^{-rs}g(r,s)drds
\]
Many of the computations in this paper may be recast in terms of the above formula and Gaussian integration, yet we prefer to use this fact only for inspiration.  
There is simply nothing to gain: everything one can do with integration we can also do directly with \eqref{eq.1zip}. The perspective of integration does sometimes
make it clearer perturbed Gaussians are closed under contraction.

\section{Turning algebra into linear algebra}
\label{sec.a2la}

In this section we lay the foundation for our application of generating function techniques to computations in algebra.
For the sake of argument we will set up our construction for algebras $A$ over $\Q$ but later the same ideas will be applied slightly more generally.

In analogy with the common notation for polynomial rings $\Q[z_J]$ for a set $J$ we notation we will use something similar for tensor products of more general algebras. In this notation the factors in the tensor product are indexed by the elements of $J$. The subscript indicates the tensor factor an element is in. 

\begin{definition}{\bf (Labelled tensor products)}\\
\label{def.labten}
For any finite set $J$ and any algebra $A$ with unit we set $A^{\otimes J}$ to be the free algebra on elements $\{a_j|a\in A, j\in J\}$ quotiented out by the relations $a_ja'_j = (aa')_j$ for $a,a'\in A$ and $a_ia'_j = a'_ia_j$ for all $i\neq j$. Any factors $1_j$ will be ommitted. 
We identify $A^{\otimes \{1,2,\dots n\}}$ with $A^{\otimes n}$ sending $a_i$ to $1\otimes 1\otimes \dots a \otimes 1 \dots \otimes 1$ with $a$ in the $i$-th position.
\end{definition}

For example we would write $1\otimes x+x\otimes 1+x\otimes y = x_2+x_1+x_1y_2\in A^{\otimes \{1,2\}}$.

We say $A$ is a PBW\footnote{The name refers to the Poincare-Birkhoff-Witt theorem in Lie theory that says that the universal enveloping algebra is PBW in the above sense.} algebra if there exists a vector space isomorphism $\Oo:\Q[z]\to A$. The variable $z$ is allowed to be a vector of variables. For example the associative algebra $A=\mathcal{U}(\mathfrak{h})$ generated $\mathbf{p},\mathbf{x}$, subject to the relation $[\mathbf{p},\mathbf{x}]=1$ is PBW. We will use the isomorphism $\Oo:\Q[p,x] \to A$ that sends the commutative monomials to lexicographically ordered monomial in the generators $z = (p,x)$. More precisely we will use the vector space isomorphism $\Oo:\Q[p,x]\to \mathcal{U}(\mathfrak{h})$ sending $p^kx^\ell$ to $\mathbf{p}^k\mathbf{x}^\ell$. 

In a PBW algebra $A$ we will use the same notation $\Oo$ also for its extensions to tensor powers, so $\Oo:\Q[z_J]\to A^{\otimes J}$. With these definitions in place we can use $\Oo$ to transfer any structures on $A$ and its tensor powers to a linear map between polynomial rings. Recall we introduced a category $\tilde{\mathcal{P}}$ for such linear maps. If we define a similar category $\tilde{\mathcal{H}}$ for morphisms between tensor powers of $A$ then $\Oo$ gives us an equivalence of categories. 

\begin{lemma}
Define a category $\tilde{\mathcal{H}}$ whose objects are finite sets and whose morphisms are linear maps $\Hom(A^{\otimes J},A^{\otimes K})$.
The isomorphism $\Oo$ gives rise to an isomorphism of
monoidal categories $\OO:\tilde{\mathcal{H}} \to \tilde{\mathcal{P}}$ defined by $\OO(f) = \Oo^{-1}\pp f\pp \Oo$. 
\end{lemma}

Combining with the generating function functor $\G$ gives us a way to describe many structures on $A$ using generating functions:
\[
\tilde{\mathcal{H}}\xrightarrow{\OO}\tilde{\mathcal{P}}\xrightarrow{\G} \tilde{\mathcal{C}}\]
where we recall all three categories have finite sets as objects and $\Hom_{\tilde{\mathcal{C}}}(J,K) = \Q[z_K]\llb \zeta_J\rrb$ and $\Hom_{\tilde{\mathcal{P}}}(J,K) = \Hom(\Q[z_J],\Q[z_J])$.

The prime example of the type of map we would like to describe from the point of view of the polynomial ring is the multiplication itself.
In our notation it makes sense to expand the algebra multiplication map $m:A\otimes A\to A$ to a family of maps 
\[
\mathbf{m}^{ij}_k:A^{\otimes\{i,j\}}\to A^{\otimes \{k\}}
\]
where $\mathbf{m}^{ij}_k$ multiplies from the $i$-th tensor factor with the $j$-th tensor factor and places the result in tensor factor indexed $k$.

\[ \begin{tikzcd}
A^{\otimes\{i,j\}} \arrow{r}{\mathbf{m}^{ij}_k} & A^{\otimes\{k\}}  \\
\Q[z_i,z_j] \arrow{r}{m^{ij}_k}\arrow{u}{\Oo} & \Q[z_k]\arrow{u}{\Oo}
\end{tikzcd}
\]

As a notational convention we write elements and maps from the original algebra in boldface, polynomial ring objects in regular script and generating functions either as $\G(f)$ or as
a smaller pre-superscript. For example $\mathbf{m}^{ij}_k$, $m^{ij}_k$ and generating function $\gm^{ij}_k = \G(m^{ij}_k)$. In the next subsection we will see some concrete instances of this process.

\subsection{Heisenberg algebra case}
\label{sec.Weyl}

To further illustrate our technique of turning algebra into linear algebra and then into generating functions we focus on
the Heisenberg algebra. Recall this is $A=\mathcal{U}(\mathfrak{h})$ with generators $\mathbf{p,x}$ and relation $[\mathbf{p,x}] = 1$. We ordered the generators in each monomial
alphabetically to obtain a vector space isomorphism $\Oo:\Q[p,x]\to A$. In this example we take $z = (p,x)$. 

The way we will apply our generating function techniques to carry out computations in algebras is by choosing a basis. Throughout this section we will assume 
we have an algebra $A$ together with a vector space isomorphism $\Oo:\Q[z]\to A$. For example we can consider the associative algebra $A=\mathcal{\mathfrak{h}}$ generated $\mathbf{p},\mathbf{x}$, subject to the relation $[\mathbf{p},\mathbf{x}]=1$. Like in all universal enveloping algebras of Lie algebras the PBW theorem tells us there is a vector space isomorphism
$\Oo:\Q[p,x] \to A$ sending a monomial to an ordered one. So in this case $z = (p,x)$ and we choose to use alphabetic ordering of the monomials. 
More precisely we will use the vector space isomorphism $\Oo:\Q[p,x]\to \mathcal{U}(\mathfrak{h})$ sending $p^kx^\ell$ to $\mathbf{p}^k\mathbf{x}^\ell$.

\begin{lemma}{\bf(Generating function for multiplication)}\\
\label{lem.mHeis}
\[\gm^{ij}_k=\G(m^{ij}_k) = e^{(\pi_i+\pi_j)p_k+(\xi_i+\xi_j)x_k-\xi_i\pi_j}\in \Q[p_k,x_k]\llb \pi_i,\pi_j,\xi_i,\xi_j\rrb =\tilde{\mathcal{C}}(\{i,j\},\{k\})\]
\end{lemma}
\begin{proof}
The generating function $\G(m^{ij}_k)$ is found by using Weyl's canonical commutation relation $e^{\xi\mathbf{x}}e^{\pi\mathbf{p}} = e^{-\pi\xi}e^{\pi\mathbf{p}}e^{\xi\mathbf{x}}$ as follows.
\[
\gm^{ij}_k=m^{ij}_k(e^{\pi p+\xi x})= e^{\pi p+\xi x}\pp\Oo^{\otimes\{i,j\}}\pp \mathbf{m}^{ij}_k=\]
\[e^{\pi_i \mathbf{p}}e^{\xi_i \mathbf{x}}e^{\pi_j \mathbf{p}}e^{\xi_j \mathbf{x}}\pp(\Oo^{\otimes \{k\}})^{-1} =e^{(\pi_i+\pi_j) \mathbf{p}}e^{-\xi_i\pi_j}e^{(\xi_i+\xi_j) \mathbf{x}}\pp(\Oo^{\otimes \{k\}})^{-1}\]\[ =
e^{(\pi_i+\pi_j)p_k+(\xi_i+\xi_j)x_k-\xi_i\pi_j}
\]
\end{proof}

Of course we already know that multiplication in $\mathcal{U}(\mathfrak{h})$ is associative but it is instructive to also check it using generating functions.
So instead of checking associativity directly in the form
\[\mathbf{m}^{12}_k\pp \mathbf{m}^{k3}_\ell  = \mathbf{m}^{23}_k\pp \mathbf{m}^{1k}_\ell\]
we check the same equation after applying the functors $\OO,\G$:
\[\gm^{12}_k\pp \gm^{k3}_\ell  = \gm^{23}_k\pp \gm^{1k}_\ell\]
The left hand side is computed using $\la f(s)e^{r\lambda} \ra_{r,s} = f(\lambda)$ twice:
\[
\gm^{12}_k\pp \gm^{k3}_\ell  = \la e^{(\pi_1+\pi_2)p_k+(\pi_k+\pi_3)p_\ell+(\xi_1+\xi_2)x_k+(\xi_k+\xi_3)x_\ell-\xi_1\pi_2-\xi_k\pi_3} \ra_{(\pi_k,p_k),(\xi_k,x_k)} =
\]
\[
e^{(\pi_1+\pi_2+\pi_3)p_\ell+(\xi_1+\xi_2+\xi_3)x_\ell-\xi_1\pi_2-(\xi_1+\xi_2)\pi_3}
\]
Notice that the generating function for multiplication is Gaussian, it is the exponential of a quadratic. The composition of the generating functions was rather easy because of
the absence of terms $p_ix_j$ in the exponent. In fact such terms $e^{px}$ cannot even occur in the ring $\Q[p,x]\llb \pi,\xi \rrb$. On the other hand, in the next section
we will be interested in multiplying precisely such expressions. In knot theory they are known as $R$-matrices and they are the elementary building blocks of the computation, representing the crossings
in the knot diagram. In this case the form of the $R$-matrix and its inverse is 
\begin{equation}\label{eq.RHeis}\mathbf{R}_{ij} = e^{t(\mathbf{p}_i-\mathbf{p}_j)\mathbf{x}_j} \qquad \mathbf{R}^{-1}_{ij} = e^{-t(\mathbf{p}_i-\mathbf{p}_j)\mathbf{x}_j}\end{equation}
Later we will see how these formulas naturally come out of the Drinfeld double construction, see Section \ref{sub.Double}.
The point of these elements is that they provide solutions to the Yang-Baxter equation
\[
\bR_{12}\bR_{13}\bR_{23} = \bR_{23}\bR_{13}\bR_{12}
\]
that is central to both knot theory and integrable systems \cite{CP94}. We will verify this equation in the next section.
 
To accommodate the $R$-matrices we extend our algebra to an algebra over $\Q\llb t \rrb$. This can be done by simply tensoring all the constructions we carried out so far with $\Q\llb t \rrb$ and completing with respect to the $t$-adic topology. We call this completion $\overline{U(\mathfrak{h})}$ and its elements are power series in $t$ with coefficients non-commutative polynomials in $\mathbf{p}$ and $\mathbf{x}$. As such we have an isomorphism of topological $\Q\llb t \rrb$ modules 
$\Q[p,x]\llb t\rrb \xrightarrow{\Oo}\overline{U(\mathfrak{h})}$ that we call still $\Oo$.
Provided we take $\Hom$ to refer to  $\Q\llb t \rrb$ module maps there still is an isomorphism 
\[\Hom_{\Q\llb t \rrb}(\Q[p_J,x_J]\llb t\rrb, \Q[p_K,x_K]\llb t\rrb)\cong \Q[p_K,x_K]\llb t,\pi_J,\xi_J\rrb\]
Everything we said so far about generating functions continues to be true in this setting. For a more thorough account of the infinite dimensional intricacies involving topological  $\Q\llb t \rrb$ modules
we refer to Section \ref{sec.TwoStep}. Here we will mostly gloss over the details of topological algebra and move on to algebraic topology in the next section instead.

To compute with the $R$-matrices using the generating function $\gm^{ij}_k$ we need to map it $\mathbf{R}_{ij}$ into an element $R_{ij} = \Oo^{-1}(\mathbf{R}_{ij})\in \Q[p_i,p_j,x_i,x_j]\llb t\rrb$.

\begin{lemma}
\label{lem.RHeis}
$R_{ij}^{\pm 1}=(\Oo^{-1})^{\otimes\{i,j\}}(\mathbf{R}^{\pm 1}_{ij}) = e^{(T^{\pm 1}-1)(p_i-p_j)x_j}$ where $T=e^{-t}$.
\end{lemma}
\begin{proof}
We will only prove the formula for $\mathbf{R}_{ij}$, leaving the similar proof for $\mathbf{R}^{-1}_{ij}$ to the reader.
Set $\Phi_1=\mathbf{R}_{ij}$ and $\Phi_2 = \Oo(e^{(T-1)(p_i-p_j)x_j})$ 
We will show $\Phi_1=\Phi_2$ by proving both satisfy the same ODE in $\overline{\U(\mathfrak{h})}^{\otimes\{i,j\}}$ given by
$\Phi|_{t=0}=1$ and $\p_t \Phi = (\mathbf{p}_i-\mathbf{p}_j)\mathbf{x}_j\Phi$. 

The only non-trivial part to check is that $\Phi_2$ satisfies this ODE. Set $\Psi = \Oo^{-1}(\Phi_2)=e^{(T-1)(p_i-p_j)x_j}$ and compute

\[
\p_t \Phi_2 = \Oo(\p_t e^{(T-1)(p_i-p_j)x_j})=\Oo(T(p_i-p_j)x_j\Psi) =(\mathbf{p}_i-\mathbf{p}_j)\Oo(x_j\Psi+(T-1)x_j\Psi)
\]
\[
(\mathbf{p}_i-\mathbf{p}_j)\Oo(x_j\Psi-\p_{p_j}\Psi) =  (\mathbf{p}_i-\mathbf{p}_j)\mathbf{x}_j\Oo(\Psi)=(\mathbf{p}_i-\mathbf{p}_j)\mathbf{x}_j\Phi_2
\]
The first equality sign comes from the formula $\mathbf{x}f(\mathbf{p}) = f(\mathbf{p})\mathbf{x}-\p_{\mathbf{p}}f(\mathbf{p})$ for any power series $f$ in $\mathbf{p}$.
\end{proof}

As a first test let us multiply two $R$-matrices: $\mathbf{F} = \mathbf{R}_{a1}\mathbf{R}_{b2}\pp \mathbf{m}^{ab}_i\in \overline{U(\mathfrak{h})}^{\otimes \{i,1,2\}}$.
For the knot theoretical interpretation of this computation see the next section. In computations we prefer to work with $F = \Oo^{-1}(\mathbf{F})$ and use generating functions:
\begin{equation}
\label{eq.F}
F = R_{a1}R_{b2}\pp \gm^{ab}_i
\end{equation}

First, Lemmas \ref{lem.RHeis} and \ref{lem.mHeis} $\gm^{ij}_k$ tell us that 
\[ F = \la e^{(T-1)\big((p_a-p_1)x_1+(p_b-p_2)x_2\big)-\xi_a\pi_b+(\xi_a+\xi_b)x_0+(\pi_a+\pi_b)p_0}\ra_{a,b} = \la e^{c+rWs+gs+rf}\ra_{a,b}\]
Here we abbreviated the contraction notation to just state that all the pairs $p_u,\pi_u$ and $x_u,\xi_u$ should be contracted for $u\in\{a,b\}$. Also
in the final formula we have $r = (p_a,p_b,\xi_a,\xi_b), s=(\pi_a,\pi_b,x_a,x_b)$, $c=(1-T)(p_1x_1+p_2x_2)$ and $W = -E^3_2$, where $E^i_j$ denotes the elementary matrix with
$1$ at the $(i,j)$-th entry and zero elsewhere. Finally $f=((T-1)x_1,(T-1)x_2,x_0,x_0)$ and $g = (p_0,p_0,0,0)$. Since $\tilde{W} = I-E^3_2$ the Contraction lemma implies
that 
\[ F = \det(\tilde{W})e^{c+g\tilde{W}f} = e^{(T-1)\big((p_i-p_1)x_1+(p_i-p_2)x_2\big)}\]
Calculations such as the above may be simplified by dealing with uncomplicated contractions using an easy case of formula \eqref{eq.1zip}:
\begin{lemma} 
\label{lem.easym}
If $G = ux_i+vp_j$ and $u,v$ do not depend on either of $p_j$ or $x_i$ then
\begin{equation}
\label{eq.easym}
e^{G} \pp m^{ij}_k = e^{G-uv}|_{i,j\mapsto k}
\end{equation}
\end{lemma}
\begin{proof}
By definition 
\[e^{ux_i+vp_j} \pp m^{ij}_k = \la e^{ux_i+vp_j+(\pi_i+\pi_j)p_k+(\xi_i+\xi_j)x_k-\xi_i\pi_j} \ra_{i,j}\]
To contract this formula we repeatedly use $\la f(r)e^{\lambda s} \ra_{r,s} = f(\lambda)$.
Contracting the pairs $\pi_i,p_i$ and $\xi_j,x_j$ yields
\[\la e^{ux_i+vp_j+\pi_jp_k+\xi_ix_k-\xi_i\pi_j} \ra_{x_i,p_j}|_{x_j\mapsto x_k,p_i\mapsto p_k}\]
Next, contracting $\xi_i,x_i$ replaces $x_i\mapsto x_k-\pi_j$ and deletes the $\xi_i$ terms:
\[\la e^{u(x_k-\pi_j)+vp_j+\pi_jp_k} \ra_{p_j}|_{x_j\mapsto x_k,p_i\mapsto p_k}\]
Finally we replace $p_j$ by $p_k-u$ to obtain
\[ e^{ux_k+vp_k-uv}|_{x_j\mapsto x_k,p_i\mapsto p_k}\]
\end{proof}
Applying formula \eqref{eq.easym} to $e^{G}$ with $G = (T-1)\big((p_1-p_i)x_i+(p_2-p_j)x_j\big)$ we see that 
\[e^{G}\pp m^{12}_0 = e^{G}|_{1,2\mapsto 0} = e^{(T-1)\big((p_0-p_i)x_i+(p_0-p_j)x_j\big)}\] since our $G$
does not even depend on $x_1$.

We close this section with an example where a denominator does arise:
\[
R_{12} \pp m^{21}_0 = \la e^{(T-1)(p_1-p_2)x_2+(\pi_1+\pi_2)p_0+(\xi_1+\xi_2)x_0-\xi_2\pi_1} \ra_{1,2}\]
Choosing $r=(p_1,p_2,\xi_1,
\xi_2)$ and $s = (\pi_1,\pi_2,x_1,x_2)$ we may write the exponent as $r W s+ rf+gs$ with 
\[
1-W = \left(\begin{array}{cccc} 1 & 0 & 0& 1-T\\
                                0 & 1 & 0& T-1\\
                                0 & 0 & 1& 0\\
                                1 & 0 & 0 & 1
 \end{array} \right) \quad f = (0,0,x_0,x_0) \quad g = (p_0,p_0,0,0)
\]
By the contraction Lemma \ref{lem.zip} we find $R_{12} \pp m^{21}_0 =\frac{1}{T}$. In the next section we will see what this has to do with the Reidemeister I move in knot theory.

\section{From algebras to tangle invariants 1}
\label{sec.AlgTang1}

\subsection{Tangle diagrams 1}

This section is intended to showcase the application of generating functions to knot and tangle invariants in a simplified and less technical case. We simplified matters in two ways: first the algebra used is just the completed Heisenberg algebra $\overline{U(\mathfrak{h})}$ introduced in the previous section. Second, the notion of tangle diagram used here is more or less standard. It does not take into account rotation numbers as Morse diagrams would. This section may be skipped without loss of continuity in the rest of the paper. We nevertheless hope that studying this simplified case will help understanding our way of thinking. 

\begin{figure}[htp!]
\begin{center}
\includegraphics[width = 10cm]{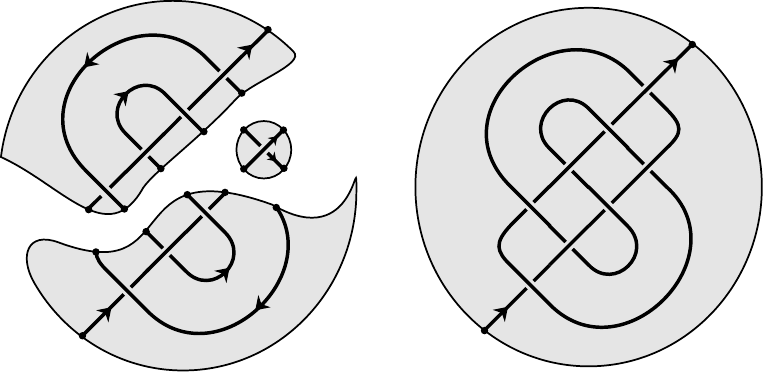}
\end{center}

\caption{Left: A tangle diagram with three underlying disks (in grey) and eight strands. Right: a tangle diagram with one strand and one underlying disk that respresents the knot $8_{17}$.}
\label{fig.Tangle817}
\end{figure}

\begin{definition}{\bf(Tangle diagrams)}\\
\label{def.TangleDiagram}
A {\bf tangle diagram} $D$ consists of a finite union of disjoint closed topological disks in the plane, called the {\bf underlying disks}, notation: $|D|$.
In each disk at least one closed interval is properly immersed. The immersed intervals (also known as {\bf strands}) intersect transversally in what we refer to as crossings. Each crossing comes with a sign $\pm$ that is indicated as in Figure \ref{fig.Xings} (left). Finally our strands are labelled with distinct elements of some set, oriented and have distinct endpoints on the boundary of the disk. 
\end{definition}

In Figure \ref{fig.Tangle817} a typical tangle diagram $D$ is shown on the left with $|D|$ consisting of three disks and $8$ strands. We often think about tangles as pieces of a knot and in this example the relevant knot is $8_{17}$ shown in the same figure as a long knot (one strand tangle). The reader is warned that our tangle diagrams are slightly non-standard in that like string links they do not have closed components however unlike string links they are not embedded in a single disk and their endpoints can be anywhere on the boundary.  As we will see this type of tangles is well suited for the universal knot invariants we are about to introduce.

The simplest tangle diagrams are the crossingless diagrams and the two crossings, see Figure \ref{fig.Xings}. A single strand in a single disk without crossings is called $1_i$. The positive and negative crossings where the two strands are labeled $i$ and $j$ and we always list the over-passing strand label first are called $X_{ij}$ and $X_{ij}^{-1}$. When convenient we also use the alternative notation $X_{ij}^{-1}=\bar{X}_{ij}$.

\begin{figure}[htp!]
\begin{center}
\includegraphics[width = \textwidth]{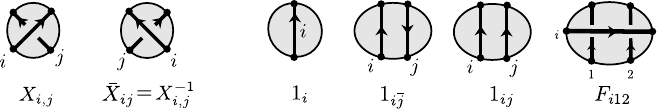}
\end{center}

\caption{The simplest oriented tangles. Left: the positive and negative crossing. Middle: some crossingless tangles. Right: the three strand tangle $F$.}
\label{fig.Xings}
\end{figure}

We often think of tangle diagrams as being assembled from crossings using two operations called merging and disjoint union.

\begin{definition}{\bf(Disjoint union and merging of tangle diagrams)}\\
\label{def.dum}
Given two tangle diagrams $D,E$ we define their disjoint union (notation: $DE$) to be their union after applying a planar isotopy to make sure their underlying disks $|D|$ and $|E|$ are disjoint.

The merge $E=D \pp m^{ij}_k$ of diagram $D$ connects the endpoint of strand $i$ with the beginning of strand $j\neq i$ by an embedded interval $c$ in the plane disjoint from the disks of $D$.
$|E|$ is obtained from $|D|$ by taking the union with a tubular neighborhood $B$ of $c$ that does not intersect more of $|E|$ and its strands than necessary\footnote{this means $B$ is a band two of whose opposite sides are attached to sufficiently small intervals in the boundary of the disk(s) containing the endpoint of $i$ and the start point of $j$ that does not meet any other strands.}
as shown in Figure \ref{fig.Merge} below. In case an annulus is created by attaching $B$ the operation is only allowed if one of the boundary components of the annulus does not contain any ends of strands. In that case a disk will be attached to turn $|E|$ into a union of disks once more. 
\end{definition}

\begin{figure}[htp!]
\begin{center}
\includegraphics[width = \linewidth]{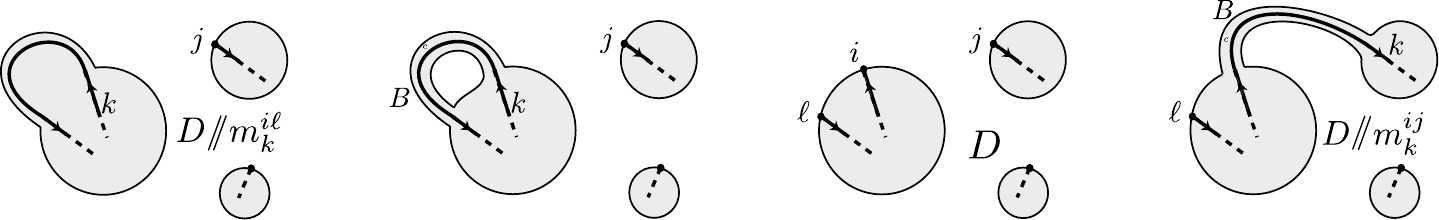}
\end{center}

\caption{Merging a pair of strands in tangle diagram $D$ (3rd picture). Either two underlying disks become one (Right) or an annulus is created (2nd picture) that is then capped off (Left).}
\label{fig.Merge}
\end{figure}

Even though the merging operation $m^{ij}_k$ may not always be defined, it does allow us to construct any tangle diagram from a disjoint union of crossings\footnote{Technically the crossingless diagrams cannot be constructed this way.}. 
Just start with the desired diagram, cut the crossings loose and then merge them back. A concrete example of this process is shown in \ref{fig.Tref}. Alternatively we could write this trefoil diagram as
\[\mathcal{T} = X_{12}X_{34}X_{56}\pp m^{14}_a \pp m^{23}_b\pp m^{a5}_i\pp m^{b6}_j\pp m^{ij}_0\]

Tangle diagrams are meant to represent tangles. Concretely, a diagram $D$ with underlying disks $|D|$ is to be interpreted as intervals properly embedded into $|D|\times [-1,1]$ with $|D|$ the underlying disks of our diagram and the endpoints all distinct on $|D|\times \{0\}$. Such embeddings are to be taken up to isotopy fixing the endpoints. By a straightforward extension of the classical Reidemeister theorem this gives rise to the following notion of equivalence of tangle diagrams.

\begin{definition} {\bf(Equivalence of diagrams)}\\
Generate an equivalence relation $"="$ on tangle diagrams by the following rules, where $D,E,F$ are tangle diagrams:
\begin{enumerate}
\item $D=E$ if $D,E$ are planar isotopic respecting the orientation and labels on the strands.  
\item If $D=E$ then $DF =EF$.
\item If $D=E$ then $D\pp m^{ij}_k = E\pp m^{ij}_k$, provided both make sense.
\item $D=E$ if $D$ and $E$ appear in one of the Reidemeister equalities shown in Figure \ref{fig.OReidemeister}.
\end{enumerate}
\end{definition}

\begin{figure}[htp!]
\begin{center}
\includegraphics[width = \textwidth/2]{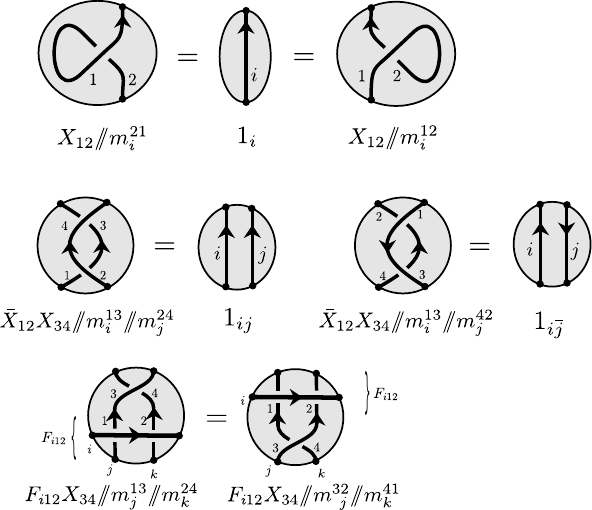}
\end{center}

\caption{The oriented Reidemeister moves, together with their algebraic description.}
\label{fig.OReidemeister}
\end{figure}

In using the notation $=$ for the equivalence relation between tangle diagrams we abuse our notation slightly in that strict equality is not expressible anymore but this will not cause any problems in the sequel.

As hinted at above we assert that isotopy classes of tangles in a disjoint union of cylinders are in bijection with equivalence classes of diagrams. Also tangles with a single strand coincide with long knots which are well known to be equivalent to round knots\footnote{embeddings of the circle}, see for example a diagram of the knot $8_{17}$ on the right of Figure \ref{fig.Tangle817} and a long trefoil in Figure \ref{fig.Tref} below. 

As shown in Figure \ref{fig.OReidemeister} the Reidemeister equivalences can be written algebraically as follows.
\begin{align}
\label{eq.OReid}
X_{12}\pp m^{21}_i = 1_i =  X_{12}\pp m^{12}_i \\
X^{-1}_{12}X_{34}\pp m^{13}_i \pp m^{24}_j = 1_{ij} \qquad X^{-1}_{12}X_{34}\pp m^{13}_i \pp m^{42}_j = 1_{i\bar{j}} \\
F_{i12}X_{34}\pp m^{13}_j \pp m^{24}_k = F_{i12}X_{34}\pp m^{32}_j \pp m^{41}_k \qquad F_{i12} =X_{a1}X_{b2}\pp m^{ab}_i
\end{align}

\subsection{Universal tangle invariants 1}

The tangle invariants we are about to introduce are defined by specifying their value on the crossings and giving rules for how they behave under disjoint union and merging.
In this section we use the algebra developed so far to provide invariants $\btZ_A$ of knots and tangles. This invariant is a simplified version of the so called universal knot invariant first introduced by Lawrence \cite{La89}, see also \cite{Oh01}. This is a simplified version of the full-fledged universal invariant $Z$ treated in Section \ref{sec.AlgTang2}. We also use a simplified model of tangle diagrams here as described in the previous subsection.

\begin{definition}\label{def.uiZt}
Suppose $A$ is an associative algebra $A$ with unit $\bI$ and multiplication $\mathbf{m}^{ij}_k:A^{\otimes \{i,j\}}\to A^{\otimes \{k\}}$ and we have chosen elements $\mathbf{R}^{\pm 1}_{ij}\in A^{\otimes \{i,j\}}$.
For a tangle diagram $D$ whose strands are labeled by set $L$ define $\btZ_A(D)\in A^{\otimes L}$ by the following rules.
\begin{enumerate}
\item If $D$ does not have crossings, $\btZ_{A}(D) = 1^{\otimes L}$. (crossingless diagrams don't count).
\item $\btZ_{A}(X_{ij}^{\pm 1}) = \mathbf{R}^{\pm 1}_{ij}$ (value of the crossings).
\item If diagram $E$ is labelled by set $M$ then $\btZ_{A}(DE) = \btZ_A(D)\otimes \btZ_{A}(E)\in A^{\otimes L\sqcup M}$\\ (disjoint union is tensor product).
\item $\btZ_{A}(D\pp m^{ij}_k) = \btZ_{A}(D)\pp \mathbf{m}^{ij}_k$ (merging is multiplication).
\end{enumerate}
We will often use the notation $\tZ = \Oo^{-1}(\btZ)$.
\end{definition}

In the literature the universal invariants are usually presented more informally as follows. Place a copy of $\mathbf{R}$ on each positive crossing of the diagram and a copy of $\mathbf{R}^{-1}$ on each negative crossing so that the first tensor factor is assigned to the over-strand and the second tensor factor to the under strand. For every strand we multiply the elements of $A$ in order of appearance and finally we tensor over all strand labels. This description matches ours as can be seen by induction on the number of crossings and merges necessary to construct the tangle diagram.

Associativity of the algebra will make sure the above rules do in fact define a value $\btZ_A(D)$ on each of our tangle diagrams. However without further assumptions 
equivalent diagrams may be assigned completely different values. To ensure $\btZ_A(D)$ is a true tangle invariant it suffices to ensure that $\btZ_A$ takes the same value on each of the Reidemeister tangle equalities \eqref{eq.OReid}. Applying the defining rules for $\btZ_A$ to each side transforms these equations into equations for the algebra $A$ and the chosen $\mathbf{R}^{\pm 1}_{ij}$ (traditionally known as $R$-matrices).
\begin{align}
\label{eq.OReidZA}
\mathbf{R}_{12}\pp \mathbf{m}^{21}_i = 1_i =  \mathbf{R}_{12}\pp \mathbf{m}^{12}_i\\
\mathbf{R}^{-1}_{12}\mathbf{R}_{34}\pp \mathbf{m}^{13}_1 \pp \mathbf{m}^{24}_2 = \mathbf{1}_{12} \qquad \mathbf{R}^{-1}_{12}\mathbf{R}_{34}\pp \mathbf{m}^{13}_1 \pp \mathbf{m}^{42}_2 = \mathbf{1}_{1\bar{2}} \\
\label{eq.ZR3}
\btZ_A(F_{i12})\mathbf{R}_{34}\pp \mathbf{m}^{13}_j \pp \mathbf{m}^{24}_k = \btZ_A(F_{i12})\mathbf{R}_{34}\pp \mathbf{m}^{32}_j \pp \mathbf{m}^{41}_k 
\qquad \btZ_A(F_{i12}) =\mathbf{R}_{a1}\mathbf{R}_{b2}\pp \mathbf{m}^{ab}_i
\end{align} 

In what follows we will illustrate the invariant concretely by working with the Heisenberg algebra. So throughtout the section we will assume $A=\Uh$ and will abbreviate $\btZ=\btZ_{\Uh}$ and
$\tZ = \Oo^{-1}\btZ$. In this case we can illustrate the power of the Contraction Lemma \ref{thm.zip}.
For convenience we work with generating functions throughout and recall (Lemma \ref{lem.RHeis}) the $R$-matrices of $\Uh$ are 
\[R_{ij}^{\pm 1} = e^{(T^{\pm 1}-1)(p_i-p_j)x_j} = \Oo^{-1}(\bR^{\pm}_{ij})\]

We are now in a position to verify the Reidemeister moves explicitly in the form \eqref{eq.OReidZA}. This will show that $\btZ$ is well-defined up to multiplication by a power of $\pm T$ where $T=e^{t}$. The ambiguity in powers of $T$ comes from the fact that the Reidemeister I move is only satisfied up to such powers. Indeed at the end of the previous section we already checked that $R_{12}\pp m^{21}_0 = T^{-1}$. We leave Reidemeister II to the reader and spend the remainder of this section verifying Reidemeister III and computing the value of the trefoil knot in Figure \ref{fig.Tref}.

We already computed the invariant of the three strand tangle $F_{i12}$ shown to the right in Figure \ref{fig.Xings}. We read Reidemeister III as the act of
sliding a crossing through $F$ and so it is useful to recall that we computed in the previous section (see Equation \eqref{eq.F}) that 
\[
F_{i12} = e^{(T-1)\big((p_i-p_1)x_1+(p_i-p_2)x_2\big)}
\]
The left hand side of Reidemeister III (Equation \eqref{eq.ZR3}) after applying our functors $\OO$ and $\G$ becomes
\[
LHS =F_{i12}R_{34}\pp \gm^{13}_j\pp \gm^{24}_k
\]
We compute this step by step using $e^{G}\pp m^{ij}_k = e^{G-uv}|_{i,j\mapsto k}$ where $G=ux_i+vp_j$ is a quadratic with no factor $x_ip_j$, see Lemma \ref{lem.easym}:
\[
LHS =e^G\pp \gm^{13}_j\pp \gm^{24}_k = e^{G'}\pp \gm^{24}_k = e^{(T-1)\big((p_i-p_j)x_j+(p_j-p_k)x_k-(T-1)(p_i-p_j)x_k+T(p_i-p_k)x_k\big)}
\]
Here we set $G = (T-1)\big((p_i-p_1)x_1+(p_i-p_2)x_2+(p_3-p_4)x_4\big)$ and in the first step we subtract $(T-1)^2(p_i-p_1)x_4$ and
substitute $1,3\mapsto j$ to get $G' = (T-1)\big((p_i-p_j)x_j+(p_i-p_2)x_2+(p_j-p_4)x_4-(T-1)(p_i-p_j)x_4\big)$. In the final step
we subtract $-(T-1)^2(p_i-p_2)x_4$ and substitute $2,4\mapsto k$ to obtain the formula for LHS shown above.

Similarly the right hand side $RHS = F_{i12}R_{34}\pp \gm^{32}_j\pp \gm^{41}_k$ is computed in two steps as
\[
RHS = e^{G}\pp \gm^{32}_j\pp \gm^{41}_k = e^{G''} \pp \gm^{41}_k  = e^{(T-1)\big((p_i-p_k)x_k+(p_i-p_j)x_j+(p_j-p_k)x_k+(T-1)(p_j-p_k)x_k\big)}
\]
where $G$ is as above and not depend on $x_3$ so that we get $G''$ by simply renaming $3,2\mapsto j$:
$G'' = e^{(T-1)\big((p_i-p_1)x_1+(p_i-p_j)x_j+(p_j-p_4)x_4\big)}$. In the final step we 
subtract $-(T-1)^2(p_j-p_4)x_1$ and substitute $4,1\mapsto k$ to get the formula shown. This verifies Reidemeister III.

\begin{figure}[htp!]
\begin{center}
\includegraphics[width = 10cm]{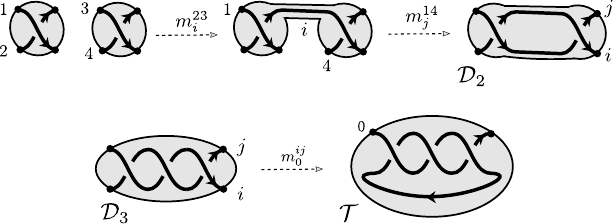}
\end{center}

\caption{Left to right: Building up a diagram of a long trefoil knot by merging the disjoint union of three crossings.}
\label{fig.Tref}
\end{figure}

As another example let us compute the invariant on the (long, positive) trefoil knot shown in Figure \ref{fig.Tref} (right).
Denote by $\mathcal{D}_n$ then $2$-strand braid diagram with $n$ positive crossings and strands named $i,j$, where $i$ is the final overpass.
Notice that $\mathcal{D}_1 = X_{ij}$ and $\mathcal{D}_2 = X_{12}X_{34}\pp m^{23}_i\pp m^{14}_j$ and
$\mathcal{D}_3 = \mathcal{D}_2 X_{12}\pp m^{1i}_i\pp m^{2j}_j$ as shown in Figure \ref{fig.Tref}.
After computing the invariant on these diagrams we compute the invariant of the trefoil using $\mathcal{T} = \mathcal{D}_3 \pp m^{ij}_0$.

Using Lemma \ref{lem.easym} as in the Reidemeister III example above we compute
\[
\btZ(\mathcal{D}_2) = R_{12}R_{34}\pp m^{14}_j \pp m^{23}_i = e^{ (T-1)(p_i-p_j)(x_i-Tx_j)}
\]
Next, to find $\mathcal{D}_3$ we merge another crossing: $\btZ(\mathcal{D}_3) = \btZ(\mathcal{D}_2)R_{12}\pp m^{1i}_i\pp m^{2j}_j = $
\[
e^{ (T-1)\big((p_i-p_j)(x_i-Tx_j)+(p_1-p_2)x_2\big)}\pp m^{1i}_i\pp m^{2j}_j= 
e^{ (T-1)(p_i-p_j)\big(-Tx_i+(1+T^2)x_j\big)} 
\]
To complete the trefoil computation we need the full power of the contraction Lemma \ref{lem.zip}
$
\btZ(\mathcal{T}) = \btZ(\mathcal{D}_3)\pp m^{ij}_0 = $
\[
e^{ (T-1)(p_i-p_j)\big(-Tx_i+(1+T^2)x_j\big)+(\pi_i+\pi_j)p_0+(\xi_i+\xi_j)x_0-\xi_i\pi_j}\pp m^{ij}_0 =
\la e^{rf+gs+rWs} \ra_{i,j} = \frac{1}{1-T+T^2}
\]
using $r = (p_i,p_j,\xi_i,\xi_j)$, $s = (\pi_i,\pi_j,x_i,x_j)$ and with $\tilde{W} = (1-W)^{-1}$:
{\small
\[W = \left(
\begin{array}{cccc}
 0 & 0 & -(T-1) T & -(1-T) \left(T^2+1\right) \\
 0 & 0 & -(1-T) T & -(T-1) \left(T^2+1\right) \\
 0 & -1 & 0 & 0 \\
 0 & 0 & 0 & 0 \\
\end{array}
\right)
\]
} and $f = (0,0,x_0,x_0)$ and $g = (p_0,p_0,0,0)$. Recall $\tilde{W} = (1-W)^{-1}$ and that according to the contraction lemma the contraction should be $e^{g\tilde{W}f}\det\tilde{W}= \frac{1}{1-T+T^2}$ because the exponent $g\tilde{W}f$ is zero. 

Our conclusion then is that \[\btZ(\mathcal{T}) = \frac{1}{1-T+T^2}\]
This concludes our warmup on the Heisenberg algebra invariant. We will leave it to the reader to prove one always finds the
reciprocal of the Alexander polynomial of a knot. A similar result will be proven as the $\eps=0$ part of Theorem \ref{thm.structZ}. \\

\subsection{Computer practicum 1}
The reader has noticed that although not particularly hard, the computation of $\btZ$ takes some effort to do by hand even for simple tangles.
Using the more general Mathematica program explained in Appendix \ref{sec.Implementation}, the invariant $\btZ$ can be computed quickly as follows.
First we set the parameters $\$k$ and $\hbar$ that do not play a role now will be important in later sections.\\

The notation used in the program is similar to the one in the main text except for the following details. We use $y$ and $\eta$ instead of $p$ and $\pi$. Also
A morphism in $\tilde{\mathcal{C}}(J,K)$ of the form $Pe^G$ is denoted by $\mathbb{E}_{J\to K}[G,P]$, where $G$ is supposed to be a quadratic in $x,y,\xi,\eta$ with coefficients in $\Q(T)$
and $P$ is also a rational function of $T =e^{-t}$. When $P=1$ it is omitted from the $\mathbb{E}$ notation. The program also uses $hm_{i,j\to k}$ for our $m^{ij}_k$ and $hR_{ij}$ for $R_{ij}$. 
As a final touch there is a command for replacing $e^{-t}$ by $T$ and this is called by appending \texttt{/.} \texttt{l2U}. With this notation in place we can easily recompute the results we just did by hand and many more. Our convention is to write the input in boldface and the output right below it. Sometimes we string together several input lines in a single list. This has the advantage that the output will also be a single list of outputs.

First we investigate invariance under the Reidemeister moves. We see that Reidemeister 1 fails as expected and Reidemeister 2 holds. Reidemeister 3 is checked as in the main text by first computing
the value of the tangle F. We then checked equality (using $\equiv$) between the two sides of the Reidemeister move obtaining the reassuring output \texttt{True}\\

\noindent\includegraphics[width=10cm]{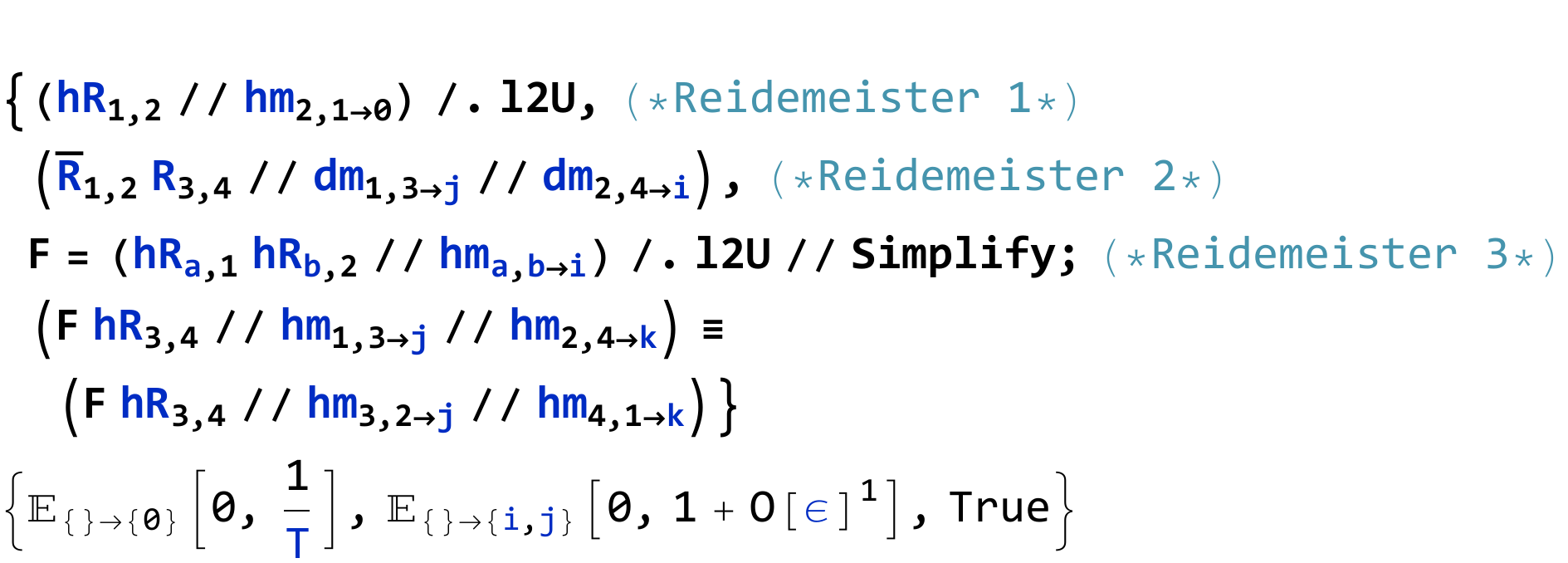}\\

Notice how after applying the \texttt{Simplify} and \texttt{l2U} command we do indeed find the same expression for the three strand tangle F as we did before.

Next we check again the computation for $D_2,D_3$ leading up to the value of the trefoil:\\

\noindent\includegraphics[width=14cm]{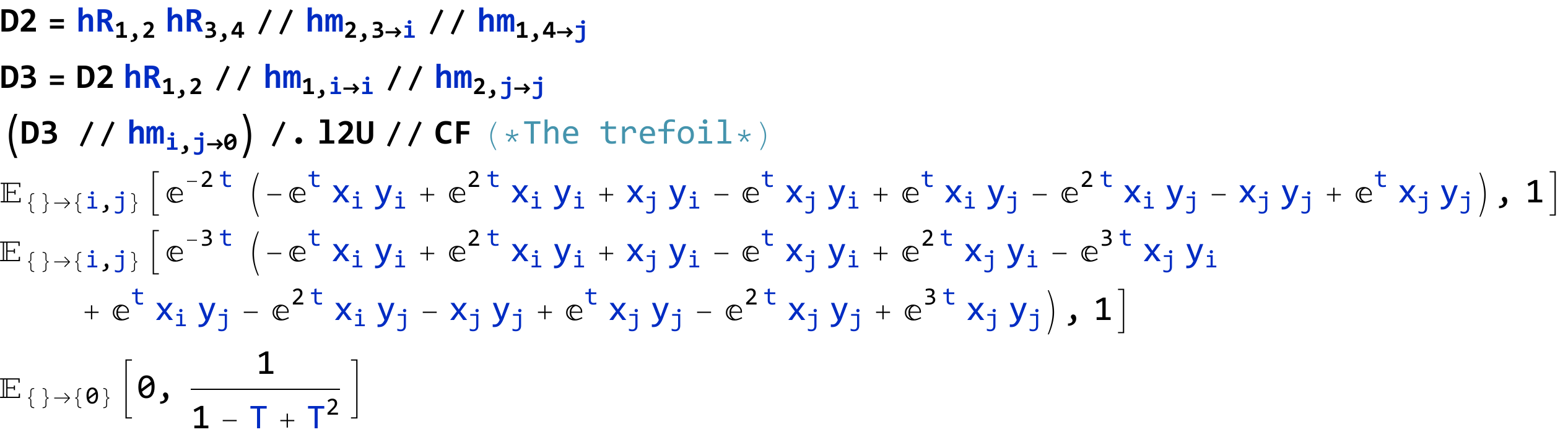}\\

As a bonus we compute the universal invariant for the knot $8_{17}$ using a \texttt{Do} loop:\\

\noindent\includegraphics[width=10cm]{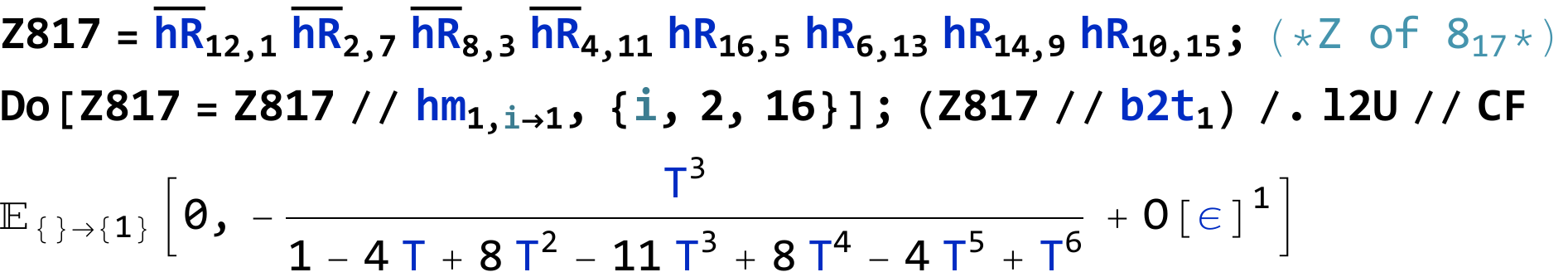}\\

The above description of the knot as a merging of eight crossings can be read off directly from Figure \ref{fig.Tangle817}.
One just enumerates the over- and underpasses as one walks along the strand and uses these labels for the eight crossings involved.

\section{Main example}
\label{sec.MainExample}

We introduce a Hopf algebra $\B$ that forms the basis for our main example $\D$. The algebra $\B$ is based on the two-dimensional non-commutative Lie algebra so in some sense the simplest possible
algebra of its sort. Next we construct the Drinfeld double $\D$ of our algebra $\B$. The point of the double construction is that $\D$ automatically comes with a solution to the Yang-Baxter equation called the (universal) $R$-matrix. Roughly speaking $\D = \B\otimes \B^*$ as a co-algebra but with a twisted product explained below. 
The product on $\D$ is designed so that $\D$ becomes a quasi-triangular Hopf algebra. Even better, the $R$-matrix is just the identity $\id_\B$! (viewed as an element in  $\B\otimes \B^*$).
In our case we are even more lucky and $\D$ contains a ribbon element. All these notions are crucial to constructing the universal knot and tangle invariants in Section \ref{sec.AlgTang2}
refining the invariants $\btZ$ from Section \ref{sec.AlgTang1}.

\subsection{Two-dimensional algebra and two-step Gaussians}
\label{sec.TwoStep}

The foundation for our main example is the algebra $\B$ for beginning or Borel. The latter refers to the close connection to the Borel part of $\mathcal{U}_\hbar(\mathfrak{sl}_2)$. 
The reader is warned in advance that our variable $\eps$ is not playing the role of the usual $\hbar$ in quantized enveloping algebras of Drinfeld-Jimbo type \cite{ES98}. One key difference is that $\eps$ does not appear in the co-product of Definition \ref{def.B} below. 

As the dimension of our algebras is infinite we need to take some care to define a topology to make sense of infinite series such as $e^{b\hbar}$.
We generally use formal power series in $\hbar$ and the $\hbar$-adic topology as recalled below.  

\begin{definition}{\bf ($\hbar$-adic topology)}
\label{def.hadic}
\begin{enumerate}
\item Define $\Q_\hbar[z] = \Q[z]\llb \hbar \rrb$ and set $\K=\Q_\hbar[\eps]$. 
\item Recall the $\hbar$-adic norm on $\Q_\hbar[z]$ is defined by $|f| = 2^{-k}$ if $\hbar^k$ is the highest power of $\hbar$ dividing $f$ (and $|0|=0$ by convention). 
\item We say a topological $\K$-module $M$ is {\bf topologically generated} by generators $\mathbf{z}$ if all $m\in M$ can be written as a (convergent) series in $\hbar$ whose coefficients are non-commutative polynomials in $\mathbf{z}$ with coefficients in $\Q[\eps]$.
\item Finally $M$ is of {\bf PBW-type} if there exists an isomorphism of topological $\K$-modules $\Oo:\Q_\hbar[\eps,z]\to M$. 
\end{enumerate}
\end{definition}

In what follows, tensor products will always be taken over $\K$ and it is understood that the tensor product will be completed. For the PBW-type modules we work with 
the completion process is relatively straightforward in that if the modules are isomorphic to $\Q_\hbar[a]$ and $\Q_\hbar[b]$ then the completed tensor product will be isomorphic to $\Q_\hbar[a,b]$. A relevant example of why this is necessary is to deal with expressions like $e^{a_1a_2\hbar }\in\Q_\hbar[a]^{\otimes 2}$. 

Notice that $\K$-module maps between topologically generated $\K$ modules are automatically continuous with respect to the chosen topologies. As such they are completely determined by the images of the finite monomials. For a more complete exposition on topological algebras in the context of quantum groups we refer to \cite{Ka95} chapter XVI. The casual reader can mostly ignore the topological subtleties.

\begin{definition} {\bf(The Hopf algebra $\B$)}\\
\label{def.B}
The topological $\K$-module $\B$ is topologically generated by $\by$ and $\bb$ subject to the relation
\[
[\by,\bb] = \epsilon\by
\]
Also define $\Oo:\Q_\hbar[\eps,y,b]\to \B$ by 
\[
\Oo \big(\sum_\ell\sum_{i,j} f_{ij}y^jb^i \hbar^\ell\big) =  \sum_\ell\sum_{i,j} f_{ij} \by^j\bb^i\hbar^\ell
\] 

Setting $\mathbf{B}=e^{-\hbar \bb}$ and $q = e^{\epsilon \hbar}$ we define a (topological) Hopf algebra structure by
\[\mathbf{\Delta}(\by) =\by_2+\by_1\mathbf{B}_2 \quad \mathbf{\Delta}(\bb) =\bb_1+\bb_2 \qquad \mathbf{S}(\by) = -\by\mathbf{B}^{-1} \quad \mathbf{S}(\bb) = -\bb\]
Finally the co-unit\footnote{Not to be confused with our deformation parameter $\epsilon$.} $\varepsilon$ sends both $\by,\bb$ to $0$. 
\end{definition}

It should be clear that $\Oo$ makes $\B$ of PBW type where we ordered the monomials anti-alphabetically $y,b$ in honor of Yang-Baxter.
The reader is invited to verify that extending the co-unit and co-product multiplicatively and the antipode anti-multiplicatively this does indeed define a Hopf algebra.
We remark that $\B$ is the simplest instance of at least two constructions of Hopf algebras. For example $\B$ can be understood as the bosonisation of the braided line of Majid (Thm 16.4 \cite{Ma02}). $\B$ may also be viewed as a quantization of the two-dimensional Lie bi-algebra with generators $y,b$ and bracket as shown and "standard" co-bracket $\delta(b) = 0$, $\delta(y) = y\wedge b$ (see sec 6.4 of \cite{CP94}). The form of the co-product is in fact dual to that of the product making, $\B$ self-dual in some sense. In the next section we will give a precise meaning to this statement.

Before moving on to the main example in the next section we briefly explore what generating functions for the Hopf algebra operations of $\B$ look like.
Many of the complications we will meet later find their origin here.

Each of the operations will be viewed as a $\K$ module map 
$\mathbf{f}:\B^{\otimes J}\to\B^{\otimes K}$. Using the notation introduced in Section \ref{sec.GenFunc} but working over $\K$ instead of $\Q$, 
the corresponding $\K$-module map is:
\[\OO(\mathbf{f}) =\Oo\pp \mathbf{f}\pp \Oo^{-1}= f\in \Hom(\Q_\hbar[\eps,y_J,b_J],\Q_\hbar[\eps,y_K,b_K])\]
We then compute the generating function $\G(f) \in \Q[\eps,y_K,b_K]\llb \hbar,\eta_J,\beta_J \rrb$ by the formula $\G(f) = f(e^{\beta_J b_J+\eta_J y_J})$. Recall that 
$\G(\id_J) = e^{\beta_J b_J+\eta_J y_J}$ and this is just a fast way to write the usual exponential generating function.

\begin{lemma} {\bf(Generating multiplication in $\B$)}
\label{lem.genmulB}
\[\G(m^{ij}_k) = e^{(\beta_i+\beta_j)b_k+(e^{-\eps\beta_i}\eta_i+\eta_j)y_k}\]
\end{lemma}
\begin{proof}
Since $e^{\beta\bb}\by = \by e^{\beta(\bb-\eps)}$ we have
\[e^{\beta\bb}e^{\eta\by} = e^{\beta\bb}\sum_{k}\frac{\eta^k}{k!} \by^k = \sum_{k}\frac{\eta^k}{k!} \by^ke^{-k\epsilon\beta}e^{\beta\bb} = e^{e^{-\eps\beta}\eta\by}e^{\beta\bb}\]
It follows that $\G(m^{ij}_k) =  e^{\beta_ib_i+\beta_jb_j+\eta_iy_i+\eta_jy_j} \pp \Oo \pp \mathbf{m}^{ij}_k \pp \Oo^{-1} =$
\[
 e^{\eta_i\by_k}e^{\beta_i\bb_k}e^{\eta_j\by_k}e^{\beta_j\bb_k}\pp \Oo^{-1} =
e^{\eta_i\by_k+e^{-\eps\beta_i}\eta_j\by_k}e^{\beta_i\bb_k+\beta_j\bb_k}\pp \Oo^{-1} = e^{(\beta_i+\beta_j)b_k+(e^{-\eps\beta_i}\eta_i+\eta_j)y_k}
\]
\end{proof}

When $\eps = 0$ this generating function is of the same Gaussian form as the one we found in Section \ref{sec.GenFunc} for the commutative multiplication of polynomials $\Q[y,b]$.
This is correct because the commutation relation is $[\by,\bb] = \epsilon\by$. However, once $\eps$ is non-zero the generating function $\G(m^{ij}_k)$ is no longer of the Gaussian type that we can compose directly using the Contraction Theorem \ref{thm.zip}. 

Fortunately it is a perturbed Gaussian in \emph{two} different senses: 1) as seen from the point of view of $y$ or 2) from the perspective of $b$.
Indeed if we regard $y,\eta$ as constants then $\G(m^{ij}_k)$ is a perturbed Gaussian in $\beta,b$ with perturbation $e^{(e^{-\epsilon\beta_j}\eta_i+\eta_j)y_k}$. Conversely, if we fix $\beta,b$ then 
$\G(m^{ij}_k)$ is a Gaussian in $\eta,y$ with constant perturbation. 
The upshot is that we can still multiply in $\B$ using the contraction theorem but we have to do carry out the multiplication in two steps:
first contract $\beta,b$ and then $y,\eta$ (or the other way around). We call perturbed Gaussian expressions of this type {\bf two-step Gaussians}.

The generating function for the co-product is more challenging and this is one reason why we introduced $\eps$. The generating function is 
\[
\G(\Delta^i_{jk}) = \Oo^{-1}\mathbf{\Delta}^i_{jk}(e^{\eta_i\by_i}e^{\beta_i\bb_i}) = \Oo^{-1}\big(\mathbf{\Delta}(e^{\eta_i\by_i})\mathbf{\Delta}(e^{\beta_i\bb_i})\big)\]

and the problem is that the factor depending on $\by$ does not seem to have a Gaussian expression. 
Such co-products are often expanded in terms of $q$-binomial coefficients $\qbin{n}{k} = \frac{[n]!}{[n-k]![k]!}$, where $[k] = \frac{1-q^k}{1-q}$ and $[k]! = [1][2]\dots [k]$.
The $q$-numbers appear because $\mathbf{B}$ and $\by$ may not commute but they $q$-commute in the sense that $\mathbf{B}\by = q\by\mathbf{B}$. Using the $q$-binomial theorem in the form  
$(\by+\mathbf{B})^n =\sum_{k=0}^n \qbin{n}{k} \by^k\mathbf{B}^{n-k}$ we find
\[
\mathbf{\Delta}(e^{\eta \by}) = \sum_{n=0}^\infty\frac{\eta^n}{n!}(\by_2+\by_1\mathbf{B}_2)^n=
\sum_{n=0}^\infty\frac{\eta^n}{n!}\sum_{k=0}^n \qbin{n}{k} \by_2^k\mathbf{B}_2^{n-k}\by_1^{n-k}
\]
Except for the special case $\eps = 0$ we do not know how to sum this series to make a Gaussian expression. Instead we will expand in $\eps$ around $\eps = 0$.
That way $\G(\Delta^i_{jk})$ will be a perturbed two-step Gaussian where now the perturbation is a series in $\eps$. For example to the first order in $\eps$
we find 
\[\G(\Delta^i_{jk}) = e^{\beta_i(b_j+b_k)+\eta_i(B_k y_j+y_k)}(1+\frac{1}{2} B_k \eta_i^2y_jy_k\eps+\OO(\eps^2))\]
In Section \ref{sec.GenD} we will show how to compute higher order terms in $\eps$, for now we will end this section by listing a similar perturbed two-step Gaussian
expression for the generating function of the antipode:
\[
\G(S_i) = e^{-\beta_i b_i-B_i^{-1}\eta_iy_i}(1-( B_i^{-1}\beta_i\eta_iy_i+ \frac{1}{2}B_i^{-2}\eta_i^2y_i^2)\eps+\OO(\eps^2))
\]
The above discussion of the co-product illustrates an important point that we will return to often. Instead of computing with $q$-special functions such as
$q$-factorials, $q$-binomials and $q$-exponentials and so on we instead expand $q = e^{\eps} = 1+\eps + \dots$ as series in $\eps$.
For simple formulas expanding $q$ like this seems wasteful in that we sacrifice nice closed form expressions for infinite series in $\eps$. However as the expressions one computes with
get more complicated the decisive advantage of expanding in $\eps$ is that this way all the formulas are in a well understood space of perturbed two-step Gaussians that we will formalize 
in Section \ref{sec.GenD}.

\subsection{Generalities on quasi-triangular Hopf algebras}
\label{sub.qtriang}

Before getting to our main example we give some details on quasi-triangularity and the Drinfeld double construction. This material is standard \cite{ES98} but for convenience we write it in our notation.
Recall the multiplication $\bm^{ij}_k$ acts on tensor factors $i$ and $j$ and placing the result in (previously unused) factor $k$. In the same way we denote by $\mathbf{\Delta}^i_{jk}$ the co-product applied to factor $i$, placing the result in factors $j$ and $k$.  

\begin{definition} {\bf(quasi-triangular Hopf algebra)}\\
A Hopf algebra $H$ is quasi-triangular if it contains an invertible element $\bR\in H^{\otimes 2}$ that satisfies the following axioms: 
\begin{enumerate}
\item $\bR_{13}\pp\mathbf{\Delta}^1_{12} = \mathbf{R}_{13}\mathbf{R}_{24}\pp \mathbf{m}^{34}_3$
\item $\mathbf{R}_{13}\pp\mathbf{\Delta}^3_{23} = \mathbf{R}_{13}\mathbf{R}_{42}\pp \mathbf{m}^{14}_1$
\item $\mathbf{\Delta}^i_{kj}\mathbf{R}_{12}\pp \mathbf{m}^{j1}_1 \mathbf{m}^{k2}_2 = \mathbf{\Delta}^i_{jk}\mathbf{R}_{12}\pp \mathbf{m}^{1j}_1 \mathbf{m}^{2k}_2$
\end{enumerate}
\end{definition}

As hinted at above an important consequence of quasi-triangularity is that $\mathbf{R}$ solves the Yang-Baxter equation:
\[
\mathbf{F}\pp \bm^{aj}_1 \pp \bm^{bi}_2 = \mathbf{F}\pp \bm^{ia}_1 \pp \bm^{jb}_2, \qquad \mathbf{F} =\bR_{ab}(\bR_{1i}\bR_{2j}\pp \bm^{12}_0)
\]
Another useful consequence is that the antipode inverts the $R$-matrix and applying it on both sides has no effect: 
\[\bR_{ij}^{-1} = \bR_{ij}\pp \mathbf{S}_i, \qquad \bR_{ij}\pp \mathbf{S}_i\pp \mathbf{S}_j=\bR_{ij}\]
Applying the antipode to the 'wrong' side of the $R$-matrix gives rise to the Drinfeld element $\mathbf{u}$ defined by 
\[\mathbf{u}_i = \bR_{12}\pp \mathbf{S}_2 \pp \mathbf{m}^{21}_i\]

\begin{definition} {\bf(ribbon Hopf algebra)}\\
\label{def.ribbonHopfAlg}
A ribbon element in a quasi-triangular Hopf algebra is a central element $\bv$  satisfying:
\[
\bv^2 = \bu\bS(\bu)\qquad \mathbf{\Delta}^{i}_{r,l}(\bv_i) = \bv_5\bv_6 \bR_{21}^{-1}\bR_{43}^{-1}\pp \bm^{145}_l\pp \bm^{236}_r \qquad \bS(\bv) = \bv \qquad \boldsymbol{\varepsilon}(\bv) = 1
\]
where $\tau(a\otimes b) = b\otimes a$. 
A quasi-triangular Hopf algebra together with a choice of a ribbon element is called a ribbon Hopf algebra.
\end{definition}

Ribbon elements do not always exist and if they do then they may not be unique. Fortunately in the case we are interested in there does exist a canonical choice of a ribbon element.

\begin{lemma} {\bf(Spinner)}\\
\label{lem.spinner}
A spinner in a quasi-triangular Hopf algebra $H$ is an element $\bC\in H$  satisfying:
\[
\bC^{-1}\bu = \bS(\bu)\bC, \quad \boldsymbol{\varepsilon}(\bC) = 1, \quad \mathbf{\Delta}(\bC) = \bC\otimes \bC, \quad \bS(\bC) = \bC^{-1}, \quad \bC x \bC^{-1} = \bS^{2}(x)
\]
A ribbon Hopf algebra always comes with a choice of spinner given by $\bC = \bu\bv^{-1}$. Conversely the existence of a spinner in $H$ implies $H$ is ribbon with
ribbon element $\bv = \bC^{-1}\bu$.
\end{lemma}

In the literature the spinner $\bC$ is known as the distinguished group-like element but for reasons becoming clear in Section \ref{sec.AlgTang2} prefer the name spinner. 
In that section we also learn how to use the universal invariant as a graphical calculus to manipulate the quasi-triangular Hopf algebra axioms in terms of tangle diagrams together
with a sense of planar rotation, a spinner. In this context all of the formulas given here make intuitive sense.

\subsection{Generalities on the Drinfeld double}
\label{sub.Double}

In this section we briefly review the Drinfeld double construction. It is a procedure for turning any Hopf algebra $H$ into a quasi-triangular Hopf algebra $D(H)$. 
For the sake of argument we will assume throughout this subsection that $H$ is finite dimensional over $\Q$. Analogues of this construction also work for infinite dimensional topological algebras
such as $\B$ from Section \ref{sec.TwoStep}, see \cite{ES98} Chapter 12.2, but here we will not go into this technical aspect in full generality.

Our form of the double construction will involve $H^{cop}$ the co-opposite of Hopf algebra $H$ which is just $H$ with $\Delta$ replaced by $\mathbf{\Delta}^{op} = \mathbf{\Delta} \pp \tau$ 
where $\tau(a\otimes b) = b\otimes a$ and $S$ replaced by $S^{-1}$.

\begin{theorem} {\bf(Drinfeld double construction)}\\
For any Hopf algebra $H$ there exists a unique Hopf algebra structure $D(H)$ on the vector space $H\otimes H^*$ such that
$H\otimes 1$ and $1\otimes (H^*)^{cop}$ are sub-Hopf algebras and $\bR = \id_H \in H \otimes H^*$ makes $D(H)$ quasi-triangular.
\end{theorem}

To write down the Hopf algebra structure of $D(H)$ in terms of that of $H$ explicitly we pause to introduce some additional notation.
First define $\boldsymbol{\pi}:H^*\otimes H \to \Q$ to be the evaluation pairing, so $\boldsymbol{\pi}(f\otimes x)  =f(x)$. In working with tensor products we will also use
the notation $\boldsymbol{\pi}^{ij}$ to mean we are pairing the $i$-th with the $j$-th tensor factors.

Since $D(H) = H\otimes H^*$ is a tensor product itself so we refine our index notation for tensor factors in $D(H)^{\otimes J}$ a little. 
If $j\in J$ then we denote by $j_1$ and $j_2$ the $H$ and the $H^*$ parts of the tensor factor indexed by $j$. For example $h_{1_1}k_{2_1}f_{1_2} = h\otimes f \otimes k\otimes 1 \in D(H)\otimes D(H)$.
Also denote by $\mathbf{\Delta}^{(2)} = \mathbf{\Delta} \pp (\id\otimes \mathbf{\Delta})$.

\begin{lemma}  {\bf(Explicit Hopf algebra structure in the double)}\\
As a co-algebra $D(H) = H\otimes (H^*)^{cop}$ meaning $\mathbf{\Delta}_{D(H)} = \mathbf{\Delta}_H\otimes \mathbf{\Delta}^{op}_{H^*}$. The antipode is given by $\mathbf{S}_{D(H)}(r\otimes s) = \mathbf{S}^{-1}_{H^*}(s)\mathbf{S}_H(r)$ and the co-unit is $\boldsymbol{\varepsilon}_{D(H)} = \boldsymbol{\varepsilon}_H\otimes \boldsymbol{\varepsilon}_{H^*}$.

Finally the product is given by
\begin{equation}
\label{eq.doublemult}
(\mathbf{m}_{D(H)})^{i j}_k = (\mathbf{\Delta}_H^{(2)})^{i_2}_{\bar{1}\bar{2}\bar{3}} (\mathbf{\Delta}_{H^*}^{(2)})^{j_1}_{123}\pp(\mathbf{S}_{H^*}^{-1})_3\pp\boldsymbol{\pi}^{3,\bar{1}}\pp\boldsymbol{\pi}^{1,\bar{3}}\pp (\mathbf{m}_H)^{i_1, \bar{2}}_{k_1}\pp (\mathbf{m}_{H^*})^{2, j_2}_{k_2}
\end{equation}
\end{lemma}

\subsection{The construction of the ribbon Hopf algebra $\D$}
\label{sub.D}

In this section we explore an appropriate dual to $\B$ and combine it with $\B$ to construct our main example $\D$. In our context of topological $\K$-modules the dual of $\K$-module $M$ is
$\Hom(M,\K)$ where $\Hom$ means $\K$-module maps. 

\begin{definition}{\bf(The dual algebra $\A$)}\\
Define $\tilde{\B}\subset \B$ to be the sub algebra topologically generated by $\tby = \hbar \by$ and $\tbb = \hbar \bb$.
Also define $\A = \tilde{\B}^*$ with the strong topology coming from the norm $|f| = \sup_{z\neq 0}\frac{|f(z)|}{|z|}$.
\end{definition}

On $\A$ there is a multiplication $\mathbf{m}_\A$ dual to the coproduct $\Delta_{\B}$. In terms of the evaluation pairing 
\[\bpi: \A\otimes \B\to \K\]
this means $\bpi(uu',v) = \bpi(u\otimes u',\Delta(v))$ for all $u,u'\in \A$ and all $v\in \B$. With respect to this multiplication we would like to find topological algebra generators for $\A$.
Elements of $\A = \Hom(\B,\K)$ are continuous so they are uniquely defined by their values on the anti-alphabetically ordered monomials 
$\by^m\bb^n$. Introduce $\mathbf{a},\mathbf{x}\in \A$ by $\bpi(\mathbf{a},\by^m\bb^n) = \delta_{m,0}\delta_{n,1}$, where $\delta$ is the Kronecker delta. 
So $\mathbf{a}$ vanishes of on all anti-alphabetically ordered monomials except for $\bb$. Likewise define $\mathbf{x}$ by 
$\bpi(\mathbf{x},\by^m\bb^n) = \delta_{m,1}\delta_{n,0}$.

\begin{lemma} {\bf(The dual algebra $\A$)}\\
\label{lem.AlgA}
If $\bpi:\A\otimes \tilde{\B} \to \K$ is the evaluation pairing, introduce 
 $\ba,\bx\in \A$ by $\bpi(\ba,\tby^m\tbb^n) = \delta_{m,0}\delta_{n,1}$ and $\bpi(\mathbf{x},\tby^m\tbb^n) = \delta_{m,1}\delta_{n,0}$. 
$\A$ is topologically generated over $\K$ by $\ba,\bx$. The evaluation pairing $\bpi$ satisfies
\[\bpi(\mathbf{a}^k\mathbf{x}^\ell,\tby^m\tbb^n ) = \delta_{k,n}\delta_{m,n}[m]!n!\] 
Moreover the generators satisfy the commutation relation \[[\mathbf{a},\mathbf{x}] = \mathbf{x}\]
$\A$ is of PBW-type using $\Oo:\Q_\hbar[\eps,a,x]\to \A$ sending monomials to alphabetically ordered monomials.
\end{lemma}
\begin{proof} (Sketch, see \cite{CP94} section 8.3B for the full account).\\
As an illustration of the general argument, let us find out what linear map is represented by $\mathbf{x}^2$. 
We have to pair it with all ordered monomials $\by^n\bb^m$.
By definition 
\[
\bpi(\mathbf{x}^2,\by^n\bb^m) = \bpi(\bx_1\bx_2, \Delta(\by^n\bb^m)) = \bpi(\mathbf{x}_1\mathbf{x}_2, \Delta(\by)^n\Delta(\bb)^m)
\]
In the previous section we learnt that $\Delta(\by)^n = \sum_{k=0}^n\qbin{n}{k}\by_2^k\mathbf{B}_2^{n-k}\by_1^{n-k}$ and similarly $\Delta(\bb)^m = (\bb_1+\bb_2)^m$
. In the pairing $\bpi(\mathbf{x}^2,\by^n\bb^m)$ we are supposed to pair the variables with common indices and so the pairing can only be non-zero when $m=0$ because each term in $\Delta(\bb)^m$ contains at least one of $\bb_1$ and $\bb_2$. By definition $\bx_i$ only has a non-zero pairing with $\by_i$ so the only contribution comes
from $k=1=n-k$ giving $\bpi(\mathbf{x}^2,\by^n\bb^m) = [2]$. 

The appearance of the $q$-factorial can be understood by viewing $[n]!$ as the generating function of the numbers permutations with a fixed number of inversions. By an inversion of $\sigma$ we mean a pair $i<j$ such that $\sigma(i)>\sigma(j)$ and denote the number of inversions of $\sigma$ by $\mathrm{inv}(\sigma)$. The result is $[n]! = \sum_{\sigma \in S_n}q^{\mathrm{inv}(\sigma)}$.

In $\bpi(\bx^n,\by^n) = \bpi(\bx_1\bx_2 \dots \bx_n,\Delta^{(n-1)}(\by)^n)$ this will come up because (writing $\Delta^{(s)}$ for the repeated application of the coproduct):
\[
\Delta^{(n-1)}(\by)^n = \big(\sum_{k=1} \by_k\prod_{j=1+k}^{n}\mathbf{B}_{j}\big)^n
\] 
Pairing this with $\bx_1\bx_2 \dots \bx_n$ we get precisely one term for each permutation $\sigma\in S_n$ because we need to select each of the terms in $\Delta^{(n-1)}(\by)$ precisely once.
Each term needs to be brought into anti-alphabetic order before pairing and the number of commutations necessary is precisely $q^{\mathrm{inv}(\sigma)}$.

Similarly one checks the commutation relation between $\ba,\bx$ by computing the pairing of $\bx\ba$ with all monomials $\by^n\bb^m$ to find that the only non-zero pairings are
$\bpi(\bx\ba,\by\bb) = 1$ and 
\[\bpi(\bx\ba,\by) = \bpi(\bx_1\ba_2,\Delta(\by)) = \bpi(\bx_1\ba_2,\by_2+\by_1\mathbf{B}_2)) = \bpi(\bx_1\ba_2,-\by_1\bb_2)) = -1\]
\end{proof}

\begin{lemma} {\bf(The Hopf algebra $\A$)}\\
\label{lem.HopfA}
The following determines a Hopf algebra structure on $\A$ that makes $\bpi$ a pairing of Hopf algebras.
$\varepsilon(\mathbf{a}) = \varepsilon(\mathbf{x}) = 0$ and 
$[\mathbf{a},\mathbf{x}] = \mathbf{x}$ and setting $\mathbf{A}=e^{-\epsilon \hbar \mathbf{a}}$ we have
\[\mathbf{\Delta}(\mathbf{a}) = \mathbf{a}_1+\mathbf{a}_2\quad \mathbf{S}(\mathbf{a}) = -\mathbf{a} \quad \mathbf{\Delta}(\mathbf{x}) = \mathbf{x}_1+\mathbf{A}_1\mathbf{x}_2 \quad \mathbf{S}(\mathbf{x}) = -\mathbf{A}^{-1}\mathbf{x}\]
\end{lemma}
\begin{proof} All these relations can be established by patiently pairing with the monomial basis, see \cite{CP94} section 8.3B for details.
\end{proof}

With the Drinfeld double construction in mind we define our main example $\D$:
 
\begin{definition} {\bf(The $\K$-module $\D$)}\\
\label{def.D}
As $\K$-modules define $\D = \B\otimes \A$ and $\tilde{\D} = \tilde{\B}\otimes \A$. On $\tilde{\D}$ define a Hopf algebra structure using
the formulas for $D(\tilde{\B})$ from Subsection \ref{sub.Double}.
\end{definition}

\begin{lemma} {\bf(The Hopf algebra $\D$)}\\
\label{lem.D}
The Hopf algebra structure of $\tilde{\D}$ extends to $\D$ making it $\hbar$-adically generated by $\by,\bb,\ba,\bx$ with relations
\[
\mathbf{x}\mathbf{y} = q\mathbf{y}\mathbf{x}+\frac{\mathbf{1}-\mathbf{AB}}{\hbar} \quad [\mathbf{a},\mathbf{x}] = \mathbf{x} \quad [\mathbf{b},\mathbf{x}]=\eps \mathbf{x} \quad [\mathbf{a},\mathbf{y}]=-\mathbf{y} \quad [\mathbf{b},\mathbf{y}]=-\eps \mathbf{y} \quad [\mathbf{a},\mathbf{b}] = 0
\]
The coproduct is determined by $\mathbf{\Delta}_\D = \mathbf{\Delta}_\B\otimes \mathbf{\Delta}^{op}_\A$, where $\mathbf{\Delta}^{op}= \mathbf{\Delta}\pp \tau$ and $\tau(r\otimes s) = s\otimes r$. 
The antipode is given by $\mathbf{S}_\D(r\otimes s) = \mathbf{S}^{-1}_\A(s)\mathbf{S}_\B(r)$ and the co-unit $\varepsilon_\D$ just sends all four generators to $0$.
\end{lemma}
\begin{proof} (sketch, see \cite{CP94} Section 8.3B for details). 

The commutation relations between the generators can be derived from formula \eqref{eq.doublemult}. The most interesting one is
\[
\bx\tby = \bx_i\tby_j\pp (\mathbf{m}_{\tilde{\D}})^{ij}_k =q\tby\bx+1-\mathbf{AB}\]
To derive this we first compute 
\begin{align}
\label{eq.DeltaAB}
\mathbf{\Delta}_\A^{(2)}(\mathbf{x})\pp(\mathbf{S}^{-1}_\A)_3 &= \mathbf{x}_1+\mathbf{A}_1\mathbf{x}_2-q^{-1}\mathbf{A}_1\mathbf{A}_2\mathbf{A}_3^{-1}\mathbf{x}_3\\
\mathbf{\Delta}^{(2)}(\tby) &= \tby_3+\mathbf{B}_3\tby_2+\mathbf{B}_3\mathbf{B}_2\tby_1
\end{align}
The formula says that we need to pair the 1st and 3rd factor of the top three monomials with the 3rd and the 1st tensor factor of the bottom three monomials.
When we get a non-zero answer we should multiply it with whatever is present in the 2nd tensor factors. Since $\bx$ only pairs non-zero with $\tby$
each monomial can only get a non-zero pairing with the one right below it in Equation \eqref{eq.DeltaAB}. Using $\pi(\mathbf{A},\mathbf{B}) = q$ the
three terms contribute $q\tby\mathbf{x}$, $1$ and $-\mathbf{AB}$ respectively. Adding these up gives the promised commutation relation in $\tilde{\D}$.
Dividing through by $\hbar$ we get the corresponding relation in $\D$ because $1-\mathbf{AB}$ is divisible by $\hbar$.
In a similar way we find the remaining commutation relations and so the algebra structure of $\D$.
\end{proof}

The main point of the Drinfeld double construction is to produce quasi-triangular Hopf algebras. We will show below that the
above defined relations on $\D$ are in fact obtained by the double construction. Hence we gain an $R$-matrix and in this case even a ribbon element. 

\begin{theorem}  {\bf($\D$ is a ribbon Hopf algebra)}\\
\label{thm.Drh}
$\D$ is a ribbon Hopf algebra with respect to $\bR,\bC$ defined below
\[
\mathbf{R} = \sum_{m,n=0}^{\infty} \frac{\hbar^{m+n}}{[m]!n!}\mathbf{y}^m\mathbf{b}^n\otimes \mathbf{a}^n\mathbf{x}^m
\]
and
\[
\bC =(\mathbf{AB})^{\frac{1}{2}}
\]
where we take the positive square root in the sense of power series in $\hbar$.
\end{theorem}
\begin{proof}
The first part follows from the topological version of the double construction, see \cite{ES98} Section 12.2 and \cite{CP94} Section 8.3B. 
We verified above that the algebra structure for $\D$ comes from the double construction.
That way the formula for the universal $R$-matrix is automatic as we already decided $\tby^m\tbb^n$ is a basis for $\tilde{\B}$ with dual basis $\frac{1}{[m]!n!}\ba^n\bx^m$ for $\A$.
Finally the formula for $\bC$ follows from the fact that $\bu = \bA\bB\ \bS(\bu)$ which is proven in Appendix \ref{sec.uABSu}. 
\end{proof}

We end this section with a comment on the relation to one of the most common quantum groups:  $\mathcal{U}_\hbar(\mathfrak{sl}_2)$. Recall this is a quantization of the universal enveloping algebra of $\mathfrak{sl}_2$, \cite{CP94}. It has (topological) generators $H,E,F$ over $\Q\llb\hbar \rrb$ with relations
\[[H,E] = 2E\quad [H,F]=-2F\quad [E,F] = \frac{q^{\frac{H}{2}}-q^{-\frac{H}{2}}}{q^{\frac{1}{2}}-q^{-\frac{1}{2}}} \qquad q = e^{\hbar}\]
In fact we have the following isomorphism:

\begin{lemma}
\label{lem.isoUh}
Set $\eps=1$ and define $\phi:\D\to \mathcal{U}_\hbar(\mathfrak{sl}_2)$ 
by $\phi(\ba) = \frac{1}{2}H$ and $\phi(\bx) = \frac{q^{\frac{1}{2}}-q^{-\frac{1}{2}}}{q\hbar}Eq^{-H/2}$ and $\phi(\by)=F$ and $\phi(\bt)=0$ 
provides a ribbon Hopf algebra isomorphism between $\mathcal{U}_\hbar(\mathfrak{sl}_2)$ and the quotient of $\D$ by the two sided ideal generated by the central element
$\bt= \bb-\epsilon \ba$.
\end{lemma}
\begin{proof} 
Instead of proving the isomorphism formally we remark that the construction of  $\mathcal{U}_\hbar(\mathfrak{sl}_2)$ as carried out in e.g. Chapter 8 of \cite{CP94} takes precisely the same steps we took to construct $\D$. The only difference is that to obtain $\mathcal{U}_\hbar(\mathfrak{sl}_2)$ one assumes $\eps=1$ and, quotients by $\bt$ and makes a slight change of variables at the very end. The lemma is thus true by construction.
\end{proof}

\section{Gaussian generating functions for $\D$}
\label{sec.GenD}

In this section we apply the generating function ideas developed in Section \ref{sec.GenFunc} to our main example $\D$. Ordering the generators $\mathbf{y,b,a,x}$ we get an isomorphism of topological $\K$-modules $\Oo:\Q_\hbar[\epsilon,y,b,a,x]\to \D$.  More precisely $\Oo$ is defined on monomials as follows and is extended $\K$-linearly.
\[\Oo(y^k b^\ell a^m x^n) = \by^k\bb^\ell \ba^m\bx^n\]
In the terminology of Definition \ref{def.hadic} the algebra $\D$ is of PBW type.
We use the same notation also for the extension of this map to tensor products indexed by a finite set $J$ to get an isomorphism 
\[\Oo:\Q_\hbar[\epsilon,y_J,b_J,a_J,x_J]\to \D^{\otimes J}\] 
Extending the categories introduced in Section \ref{sec.GenFunc} to $\K$ we get categories $\mathcal{D},\mathcal{H},\mathcal{C}$ 
to describe $\K$-linear maps between tensor powers of $\D$. These categories all have the same objects, i.e. finite sets and their morphisms are defined as follows:
\begin{enumerate}
\item
$\Hom_{\mathcal{D}}(J,K) = \Hom(\D^{\otimes J},\D^{\otimes K})$,
\item $\Hom_{\mathcal{H}}(J,K) = \Hom(\Q_\hbar[\epsilon,y_J,b_J,a_J,x_J],\Q_\hbar[\epsilon,y_K,b_K,a_K,x_K])$,
\item $\Hom_\mathcal{C}(J,K) = \Q_\hbar[\epsilon,y_K,b_K,a_K,x_K]\llb \eta_J,\beta_J,\alpha_J,\xi_J\rrb$.
\end{enumerate}
We then have two functors both of which are equivalences of categories. The first functor $\OO:\mathcal{D}\to \mathcal{H}$ expresses the given map in terms of the basis and is the identity on objects. So it maps $\mathbf{f}:\D^{\otimes J}\to \D^{\otimes K}$ to $f$ where $f(z) = (\Oo^{\otimes K})^{-1}(\mathbf{f}(\Oo^{\otimes J}(z)))$. The second functor $\G:\mathcal{H}\to \mathcal{C}$ is takes the generating function as in Section \ref{sec.GenFunc}. It is the exponential generating function of all values of the morphism $f\in \Hom_{\mathcal{H}}(J,K)$ on the monomials. More precisely 
\[\G(f) = f(e^{\eta_J y_K+\tau_J b_K+\alpha_J a_K+\xi_J x_K})\]

Composition in the categories $\mathcal{H},\mathcal{C}$ is designed so as to make sure $\OO,\G$ are really functors. So by definition $f\pp g = \Oo^{-1}(\Oo(f)\pp \Oo(g))$ describes composition in $\mathcal{H}$. More importantly
composition in $\mathcal{C}$ comes out as 
\[\psi \pp \phi = \psi \phi|_{\zeta_K \mapsto \p_{z_K}}|_{z_K =  0} = \la \psi\phi \ra_{z_K}\] Here $z$ is short-hand for all generators $z_K = (y_K,t_K,a_K,x_K)$ and $\zeta_K=(\eta_K,\tau_K,\alpha_K,\xi_K)$. 

The generating functions for the tangle invariants and Hopf operations in $\D$ that we will be interested in turn out to belong to a very special and small subcategory of $\mathcal{C}$. This category is the category $\PG$ of perturbed Gaussian generating functions. The definition of $\PG$ may at first seem rather technical but the fact that its morphisms are given by concrete formulas makes them much more tractable in practice. When one tries to do similar computations in a regular quantum group the $q$-hypergeometric expressions often become hard to manage rather quickly. Unlike in $\PG$ there seems to be no notion of closure, no useful specific form that all expressions must take in usual the quantum group setting.

\begin{definition}{\bf(Perturbed Gaussians $\PG$)}\\
\label{def.PG}
\noindent 
Define two weights $\wt,\wh$ on the monomials in $\mathcal{C}$ by $\wt(uv)=\wt(u)+\wt(v)$ and $\wh(uv)=\wh(u)+\wh(v)$ and the following values on the generators: 
\[
\begin{array}{c|c|c|c|c|c|c|c|c|c|c}
 {}& y & \eta & b &\beta & a & \alpha & x & \xi & \epsilon & \hbar\\
\hline
\wh& 1 & -1   & 1 & -1 &  0 &    0   & 0 &  0  &    1     & -1   \\
\wt& 1 &  1   & 0 & 2  &  2 &    0   & 1 &  1 &     -4    &  0 
 \end{array}
\]
Also define $\Aa = e^{\alpha}$. 
For finite sets $J,K$ define $\PG(J,K)$ to be the elements $Pe^G \in \mathcal{C}(J,K)$ satisfying:
\begin{enumerate}
\item $\wh$ is zero on all monomials in $P$ and $G$.
\item $\wt(G)=2$ and for all monomials in $P$ we have $\wt \leq 0$.
\item $G = (\beta_J,a_K)G^{(1)}{b_K\choose \alpha_J}+ (y_K,\xi_J)G^{(2)}{\eta_J\choose x_K}$ with $G^{(1)}_{jk}\in\Q[\hbar]$ and $G^{(2)}_{jk}\in \Q(\Aa_J,B_K^{\frac{1}{2}})[\hbar]$.
\item $P = \sum_{k=0}^\infty P_k\epsilon^k$ with $P_k \in \Q(\Aa_J,B_K^{\frac{1}{2}})[z_K,\zeta_J,\hbar]$.
\end{enumerate}
Here we abbreviated $z = (y,b,a,x)$ and $\zeta=(\eta,\beta,\alpha,\xi)$.
\end{definition}

The formula for $G$ above should be read as matrix and vector multiplication, e.g. $(\beta_J,a_K)$ represents a row vector and $G^{(1)} = (G^{(1)}_{jk})$ is a matrix of size $|J|+|K|$.
More precisely we have
\[
(\beta_J,a_K)G^{(1)}{b_K\choose \alpha_J} = \sum_{i,j\in J}\beta_i G^{(1)}_{i,j}\alpha_j+\sum_{j\in J,k\in K}\beta_j G^{(1)}_{j,k}b_k+\sum_{j\in J,k\in K}a_k G^{(1)}_{k,j}\alpha_j
+\sum_{j\in J,k\in K}a_k G^{(1)}_{k,\ell}b_\ell
\]
The above definition is motivated by looking closely at the expansion of the R-matrix and the multiplication in $\D$. Parts of it can already be observed in the two-step Gaussians from Section \ref{sec.TwoStep}. Notice that the $\hbar$-dependence of $P$ and $G$ is completely determined by the condition $\wh = 0$. For example the Gaussian exponent $G^{(1)}$ is at most linear in $\hbar$.

For technical reasons we introduce two more flavours of the $\PG$. These will be
 important when dealing with the double construction and controlling the denominators of the knot invariant in later sections.

\begin{definition}{\bf(Perturbed Gaussians $\PG^{\pm},\PG^{+}$)}\\
\label{def.PGm}
For finite sets $J,K$ define $\mathcal{C}^\pm(J,K)=\Q[\epsilon,z_K, \hbar^{\pm 1}]\llb\zeta_J\rrb$
and extend $\wt$ by $\wt(\hbar^{-1})=1$. Also define 
 $\PG^\pm(J,K)$ to be the set of elements $Pe^G \in \mathcal{C}^\pm(J,K)$ satisfying points 1,2,3,4 of Definition \ref{def.PG}
with the modifications that $G^{(2)}_{jk}\in \Q[\Aa^{\pm 1}_J,B_K^{\pm \frac{1}{2}},\hbar^{\pm 1}]$ and 
$P_k \in \Q[\Aa^{\pm 1}_J,B_K^{\pm\frac{1}{2}},z_K,\zeta_J,\hbar^{\pm 1}]$.\\
Finally $\PG^+(J,K)$ is the set of elements of $\PG^\pm$ that do not have negative powers of $\hbar$.
\end{definition}

For example $e^{\alpha_1\beta_2\hbar^{-1}}\in \PG^\pm(\{1,2\},\emptyset)$ but not in $\PG^+$. 
Also note that $e^{b_ia_j\hbar}$ is not in $\PG^\pm$ since it is not a power series in the Greek variables, it is not in $\mathcal{C}^{\pm}$. On the other hand we do know that
$\PG^+(J,K) \subset \PG(J,K)$.

The reader is warned that it is not always possible to compose a morphism of $\PG$ with a morphism from $\PG^\pm$. For example
$e^{a_ib_j\hbar}\pp e^{\alpha_i \beta_j\hbar^{-1}}$ would be infinite. The contraction lemma does show that $\PG^\pm,\PG^+$ and $\PG$ are categories:

\begin{lemma}{\bf(Closure under composition in the $\PG$s)}
\label{lem.compPG}
\begin{enumerate}
\item Taking the same composition law for $\PG^\pm,\PG^+$ as the one in $\PG$, all three are closed under composition. 
\item Suppose $f\in \PG(J,K)$ and $g\in \PG^\pm(K,L)$. If $f|_{\eps=0}\pp g|_{\eps=0}\in \PG^\pm(J,L)$ is well-defined then $f\pp g\in \PG^\pm(J,L)$ too. 
\end{enumerate}
\end{lemma}
\begin{proof}
We will show that for $I,J,K$ disjoint and $\phi\in \PG(I,J)$ and $\psi\in \PG(J,K)$ we have $\phi \pp \psi = \la\phi\psi\ra_{z_J,\zeta_J}\in \PG(I,K)$.
By definition $\phi$ and $\psi$ are morphisms in the category $\mathcal{C}$ and so their composition is a morphism of $\mathcal{C}(I,K)$.

Since $\phi\psi \in \mathcal{PG}(I\cup J,J\cup K)$ it suffices to show the following simpler result. If $v\in \PG(I\cup \{j\},K\cup\{j\})$ 
is such that there exists $u\in \mathcal{C}(I,K)$ with $u=\la v \ra_{z_j,\zeta_j}$ then we must have
$u\in \PG(I,K)$. The four contractions to get from $v$ to $u$ may be carried one at a time and for each of the contractions it is easy to see that
the conditions on $P$ and $G$ are preserved using the one variable version of the contraction theorem, see Equation \eqref{eq.1zip}. 

Moving on to $\PG^\pm$ the proof is the same except that we have to argue that no denominators will arise because the $W$ in Equation \eqref{eq.1zip}
is always $0$. Finally the case $\PG^+$ follows from the observation that no inverse powers of $\hbar$ can be created using this formula.

Part 2) Is proven by again looking at the contraction formula and noting that the only way the contraction can be undefined is for the determinant to vanish.
If that happens it will already have happened at $\eps=0$.
\end{proof}

With the definitions in place we now set out to prove that all structural elements and maps of $\D$ are in $\mathcal{PG}$. Recall we use
bold notation for elements in $\D$ and regular script for the corresponding element of $\mathcal{C}$. For example $R_{ij} = \bR_{ij}\pp \Oo^{-1}$.

\begin{theorem} {\bf(Main $\PG$ theorem)}\\
\label{thm.MainPG}
 All structural elements and maps of $\D$ are in $\mathcal{PG}$. More precisely:
\begin{enumerate}
\item $R_{ij}^{\pm 1} \in \PG(\emptyset,\{i,j\})$ and $v_i^{\pm},C_{i}^{\pm 1} \in \PG(\emptyset,\{i\})$.
\item The generating functions $\gD_\D,\gm_\D,\gS_\D,\gep_\D$ for the Hopf algebra operations of $\D$ are morphisms of $\PG^+$.
\item $\gpi^{ij}\in \PG^\pm(\{i,j\},\emptyset)$.
\end{enumerate}
\end{theorem}

The remainder of this section is dedicated to the proof of the main $\PG$ theorem. We break it down in a lemma for each of the elements and operations, starting with the $R$-matrix.

\begin{lemma}{\bf (Faddeev \cite{Za07})}\\
\label{lem.Faddeev}
Recall $q = e^{\hbar \epsilon}$ and $e_q^z = \sum_{n=0}^\infty \frac{z^n}{[n]!}$. We have
\[
e_q^z = e^z\exp \sum_{n=2}^\infty \frac{(1-q)^nz^n}{(1-q^n)n}
\]
and $R_{ij} = \Oo^{-1}(\bR_{ij})= e^{b_ia_j\hbar}e_q^{y_i x_j\hbar}\in\PG(\emptyset,\{{i,j\}})$.
\end{lemma}
\begin{proof} Following Zagier \cite{Za07}, we recall the $q$-derivative of a function $f(z)$ is $\frac{f(qz)-f(z)}{qz-z}$. From the series formula of the $q$-exponential is easily checked that it equals its $q$-derivative so 
\[e_q^z = \frac{e_q^{qz}-e_q^z}{qz-z} \quad \text{or rearranging} \quad e_q^{qz}=e_q^{z}(1+(q-1)z)\]
Taking the logarithm on both sides yields $\log e_q^{qz}=\log e_q^{z}+\log(1-(1-q)z)$. If we set $\log e_q^{z} = \sum_n c_n z^n$ and compare powers of $z^n$ in the previous equation
then we find $q^n c_n = c_n-\frac{(1-q)^n}{n}$. It follows that $c_n = \frac{(1-q)^nz^n}{(1-q^n)n}$ and exponentiating proves the first assertion of the lemma.

The formula for the R-matrix from Theorem \ref{thm.Drh} implies that $R_{ij} = \Oo^{-1}(\bR_{ij})= e^{b_ia_j\hbar}e_q^{y_i x_j\hbar}$.
To show that it is in $\PG(\emptyset,\{i,j\})$ we write $R_{ij}=Pe^G$ with $G = \hbar(b_ia_j+y_ix_j)$.
Clearly $\wh(G)=0$ and $\wt(G)=2$. The expression for $e_q^z$ proven in the first part of the lemma shows that the perturbation part of $R_{ij}$ is 
\[P = \exp \sum_{n=2}^\infty \frac{(1-q)^n(y_ix_j\hbar)^n}{(1-q^n)n}
=1+\sum_{k=1}^\infty P_k \eps^k\] from which it is already clear that $\wh(P_k)=0$. We should also check that $\wt(P_k)\leq 0$ and that its coefficients do not depend on $B,\Aa$ and are polynomial in the other variables. This follows by studying the expansion of $\frac{(1-q)^nz^n}{n(1-q^n)}$ in $\eps$, recalling that $q=e^{\eps \hbar}$.
For $n\geq 2$ this rational function has an $(n-1)$-fold zero at $\epsilon=0$ and so the highest power of $z$ in the coefficient of $\eps^k$
of $e_q^z$ is $2k$, coming from the taking the exp of the $n=2$ terms.
\end{proof}

Next is the inverse of the $R$-matrix, for which we also need the generating functions of multiplication in $\B$ and $\A$.

\begin{lemma}{\bf(Multiplication in $\A,\B$ and $R^{-1}$)}
\begin{enumerate}
\item The generating functions for multiplication in $\A$ and $\B$ are in $\PG^+(\{i,j\},\{k\})$:
\[(\gm_\A)^{ij}_k = e^{(\alpha_i+\alpha_j)a_k+(\Aa_j^{-1}\xi_i+\xi_j)x_k} \qquad (\gm_\B)^{ij}_k = e^{(\beta_i+\beta_j)b_k+(e^{-\eps\beta_i}\eta_i+\eta_j)y_k}
\]
\item There is a series $P = 1+\sum_{k=1}^\infty P_k\eps^k$ such that
$R_{ij}^{-1} = Pe^G \in \PG(\emptyset,\{i,j\})$ with $G=-b_ia_j-B_i^{-1}y_ix_j$.
\end{enumerate}
\end{lemma}
\begin{proof}
Part 1) already appeared as Lemma \ref{lem.genmulB} in the case of $\B$. The proof for the $\A$ case is analogous.

For part 2) we set $G=-b_ia_j-B_i^{-1}y_ix_j$ and find $P=1+P_1\eps+P_2\eps^2+\dots$ order by order in $\eps$ such that $R_{ij}^{-1} = Pe^G$ satisfies the defining equation 
\[R_{12}R_{ij}^{-1}\pp (\gm_\B)^{1i}_i(\gm_\A)^{2j}_j = 1_i1_j\]
First we check that when $\eps=0$ the defining equation is true with $R_{ij}^{-1} = e^{G}$. Note the multiplication in $\B$ becomes commutative at $\eps=0$ so we compute 
(using $\la e^{sz_j}f(\zeta_j) \ra_j = f(s)$):
\[
R_{12}e^{G_{ij}}\pp (\gm_\B)^{1i}_i(\gm_\A)^{2j}_j=\la e^{b_ia_2-b_ia_j+y_ix_2-B_i^{-1}y_ix_j+(\alpha_2+\alpha_j)a_k+(\Aa_j^{-1}\xi_2+\xi_j)x_k}\ra_{2,j}=
\]
\[
\la e^{-B_i^{-1}y_ix_j+(B_i^{-1}y_i+\xi_j)x_k}\ra_{j} = 1
\]
Next we compute the perturbation $P_k$ order by order recursively. Let us assume that we found $P_k$ for all $k<n$ satisfying 
\[R_{12}(\sum_{k=0}^{n-1}P_k\eps^ke^G)_{ij}\pp (\gm_\B)^{1i}_i(\gm_\A)^{2j}_j = 1_i1_j+\eps^n E_{ij}\]
for some error $E_{ij} = E[0]_{ij}+\OO(\eps)$. We are looking for $P_n$ that satisfies
\[R_{12}(Pe^G)_{ij}\pp (\gm_\B)^{1i}_i(\gm_\A)^{2j}_j = R_{12}(P_n\eps^ne^G)_{IJ}\pp (\gm_\B)^{1I}_i(\gm_\A)^{2J}_j+1_i1_j+\eps^n E_{ij}=0\mod \eps^{n+1}\]
Taking the coefficient of $\eps^n$ on both sides and using $R_{ij} = R[0]_{ij}+\OO(\eps)$ yields:
\begin{equation}
\label{eq.REn}
R[0]_{12}(P_ne^G)_{IJ}\pp (\gm_\B)^{1I}_i(\gm_\A)^{2J}_j+E[0]_{ij}=0
\end{equation}
Therefore left-multiplying with $e^G$ yields $P_ne^G$ and multiplying by $e^{-G}$ in the commutative sense shows $(P_n)_{ij}=$
\[e^{-G_{ij}}(e^G)_{34}R[0]_{12}(P_ne^G)_{IJ}\pp (\gm_\B)^{31I}_i(\gm_\A)^{42J}_j=-e^{-G_{ij}}(e^G)_{34}E[0]_{IJ}\pp (\gm_\B)^{3I}_i(\gm_\A)^{4J}_j
\]
The final equality follows from associativity and solving for $E[0]$ in Equation \eqref{eq.REn}.
By induction on $n$ we then see that $P_n$ satisfies the criteria for $R_{ij}^{-1}$ to be in $\PG$.
\end{proof}

The pairing $\bpi:\A\otimes \tilde{\B}\to \K$ from Lemma \ref{lem.AlgA} can be extended to a $\K$-module map $\bpi:\A\otimes \B \to \Q[\eps]((\hbar))$ by the formula $\bpi(\ba^n\bx^m,\by^{m'}\bb^{n'})= \delta_{n,n'}\delta_{m,m'}\frac{n![m]!}{\hbar^{m+n}}$. Writing $\pi = \OO(\bpi)$ its generating function $\gpi^{ij} = \pi(e^{\alpha_i a_i+\xi_i x_i+\eta_j y_j+\beta_j y_j})$ can be computed order by order in $\eps$ and this shows it is in $\PG^\pm$.

\begin{lemma} {\bf (Generating function for the pairing)}\\
\label{lem.genP}
$\gpi^{jk} = Pe^G\in \Q[\eps,\hbar^{\pm 1}]\llb \alpha_j,\xi_j,\beta_k,\eta_k \rrb$, for $G={\frac{\alpha_j\beta_k}{\hbar}+\frac{\xi_j\eta_k}{\hbar}}$ and some $P = 1+\sum_{
\ell=1}^\infty P_\ell\eps^\ell$. Moreover $\gpi^{jk}\in \PG^-(\{j,k\},\emptyset)$.
\end{lemma}
\begin{proof}
With $G$ as above we will find the coefficients $P_\ell$ of the perturbation $P$ order by order such that the defining equation holds:
\[R_{ij}\pp \gpi^{jk} = e^{\hbar(\beta_k b_i+\eta_k y_i)}\]
At least when $\eps=0$ this equation is true because $R_{ij} = e^{b_ia_j\hbar+y_ix_j\hbar}+\OO(\eps)$ so
\[R_{ij}\pp e^{\frac{\alpha_j\beta_k}{\hbar}+\frac{\xi_j\eta_k}{\hbar}} = e^{\hbar(\beta_k b_i+\eta_k y_i)}+\mod \eps\]
Since $\wh(\hbar)=1$ we see $\wh(G)=0$ and its coefficients are in $\Q[\hbar^{\pm}]$ viewed as a quadratic in the Greek variables.

Next, assume that we computed $P_\ell$ for all $\ell<n$ such that for some $E_{ik}$ independent of $\eps$ we have
\[R_{ij}\pp (\sum_{\ell=0}^{n-1}P_\ell\eps^\ell)e^{\frac{\alpha_j\beta_k}{\hbar}+\frac{\xi_j\eta_k}{\hbar}} = e^{\hbar(\beta_k b_i+\eta_k y_i)}+\eps^n E_{ik}+\OO(\eps^{n+1})\]
To find $P_n$ we write the same equation, truncating the series for $\gpi$ at $\eps^n$: 
\[R_{ij}\pp (\sum_{\ell=0}^{n}P_\ell\eps^\ell)e^{\frac{\alpha_j\beta_k}{\hbar}+\frac{\xi_j\eta_k}{\hbar}} =e^{\hbar(\beta_k b_i+\eta_k y_i)}+
R_{ij}\pp \big(P_n\eps^ne^{\frac{\alpha_j\beta_k}{\hbar}+\frac{\xi_j\eta_k}{\hbar}}+ \eps^n E_{ik}\big)+\OO(\eps^{n+1})\]
Taking the coefficient of $\eps^n$ should give
\[0 =R[0]_{ij}\pp \big(P_ne^{\frac{\alpha_j\beta_k}{\hbar}+\frac{\xi_j\eta_k}{\hbar}}+ E_{ik}\big)\]
We can solve this equation for $P_n$ by pre-composing with $e^{\frac{\alpha_s\beta_i}{\hbar}+\frac{\xi_s\eta_i}{\hbar}}$ to remove the $R[0]_{ij}$ and find:
\[(P_n)_{sk}e^{\frac{\alpha_s\beta_k}{\hbar}+\frac{\xi_s\eta_k}{\hbar}}=- E_{ik}\pp e^{\frac{\alpha_s\beta_i}{\hbar}+\frac{\xi_s\eta_i}{\hbar}}\]
So finally $(P_n)_{sk} = -(E_{ik}\pp e^{\frac{\alpha_s\beta_i}{\hbar}+\frac{\xi_s\eta_i}{\hbar}})e^{-\frac{\alpha_s\beta_k}{\hbar}-\frac{\xi_s\eta_k}{\hbar}}$.
By induction on $n$ we see that $\gpi^{ij}\in \PG^\pm(\{i,j\},\emptyset)$. 
For this it is important to note that the error $E_{ij}$ has no denominators because the formula for $R_{ij}$ has none.
\end{proof}

Following the pattern of the Drinfeld double construction we set up generating functions for the Hopf algebra operations of $\A,\B$ next.

\begin{lemma} {\bf (Generating functions for co-product and the antipodes)}
\label{lem.DSgen}
\begin{enumerate}
\item $(\gD_\A)^i_{jk}, (\gD_\B)^i_{jk}\in \PG^+(\{i\},\{j,k\})$.
\item
We have $(\gS_\A)_i,(\gS_\B)_i,(\gSbar_\A)_i,(\gSbar_\B)_i\in \PG^+(\{i\},\{i\})$ and $\gep_\A^i = 1=\gep_\B^i = 1$.

\end{enumerate}
\end{lemma}
\begin{proof}
We will prove $(\gD_\A)^i_{jk}\in \PG^+(\{i\},\{j,k\})$ and leave the analogous case for $(\gD_\B)^i_{jk}$ to the reader.
$\bpi$ is a Hopf pairing so
\[(\gD_\A)^i_{jk} = R_{1k}R_{2j}\pp (\gm_\B)^{12}_3\pp \gpi_{i3}\]
By Lemma \ref{lem.compPG} (part 2) it suffices to check that the right-hand side at $\eps=0$
can be composed and the result is in $\PG^\pm$. We have
\[
R_{1k}|_{\eps=0}R_{2j}|_{\eps=0}\pp (\gm_\B)^{12}_3|_{\eps=0}\pp \gpi_{i3}|_{\eps=0} = \]\[
\la e^{\hbar(b_1a_k+y_1x_k+b_2a_j+y_2x_j)+(\beta_1+\beta_2)b_3+(\eta_1+\eta_2)y_3+\frac{\alpha_i\beta_3+\xi_i\eta_3}{\hbar}}\ra_{1,2,3}
=e^{\alpha_i(a_j+a_k)+\xi_i(x_j+x_k)}
\]
As this is in $PG^\pm$ so is $\gD_\A$. In fact $\gD_\A$ is in $\PG^+$ because computing using the definition in terms of generators shows that negative powers of $\hbar$
do not enter in $(\gD_\A)^i_{jk} = \Oo^{-1}\mathbf{\Delta}_\A(\Oo e^{\alpha_ia_i+\xi_ix_i})$. 

From the quasi-triangular axioms it follows that $(\gS_\B)^i_j = R^{-1}_{j1}\pp \gpi^{1i}$. By Lemma \ref{lem.compPG} (part 2) it follows that $\gS_\B$ is in $\PG^\pm$.
For this we should check that at $\eps=0$ the composition $R^{-1}_{j1}\pp \gpi^{1i}$ is well-defined. We get 
\[R^{-1}_{j1}|_{\eps=0}\pp \gpi^{1i}|_{\eps=0}=\la e^{-b_ja_1-B_j^{-1}y_jx_1+\frac{\alpha_1\beta_i}{\hbar}+\frac{\xi_1\eta_i}{\hbar}}\ra_1 = e^{-b_j\beta_i-B_j^{-1}y_j\eta_i}\]
 However $S_\B$ was already defined on the generators in the previous section and its definition does not involve negative powers of $\hbar$. Therefore the formula
$(\gS_\B)^i_j = S_\B(e^{\beta_ib_j+\eta_iy_j})$ shows that in fact $(\gS_\B)^i_j\in \PG^+(\{i\},\{j\})$. Precisely the same argument works for $\gS_\A$. 

The case of the co-units is clear from the definition. In the remaining cases of $\gSbar_\B,\gSbar_\A$ we apply the next lemma to invert the generating functions for the antipodes.
\end{proof}

\begin{lemma} {\bf (Inverting in $\PG^+$)}\\
\label{lem.PGinv}
Suppose $f,g\in \PG^+(J,K)$ are such that $f\pp g = g\pp f = \id \mod \epsilon$. 
Then there exists a unique $\bar{f}\in \PG^+(K,J)$ such that $\bar{f} = g\mod \epsilon$ and $\bar{f} \pp f = f\pp \bar{f} = \id$. 
\end{lemma}
\begin{proof}
Uniqueness is a general property of the compositional inverse so we focus on existence, proving it order by order as usual.
We are looking for $\bar{f} = (1+\sum_{k=0}^\infty P_k\eps^k)e^G$ such that $\bar{f}\pp g = \id$ and we have already found $G$ by assumption.
Suppose we found $P_k$ for all $k<n$ such that $(1+\sum_{k=0}^{n-1} P_k\eps^k)e^G\pp f = \id+\eps^n E$ for some error $E$. If for simplicity we assume that
$g$ does not depend on $\eps$ then we can find
$P_n$ by adding $\eps^n P_ne^G\pp f$ to both sides, composing with $g$ and taking the coefficient of $\eps^n$:  
\[\big(1+\sum_{k=0}^{n} P_k\eps^k\big)e^G\pp f \pp g = P_ne^G+ E\pp g\]
By definition of $P_n$ this expression should be $0$ so we find
$0 = P_ne^G+E\pp g$ and so $P_n = -e^{-G}(E\pp g)$. Finally, from $f,g\in \PG^+$ it follows that $\bar{f}\in \PG^+$ because the formula for $E$ transfers all the relevant properties to $P_n$.
\end{proof}

Turning to the double $\D$ recall $\D$ is a tensor product itself so as in Subsection \ref{sub.Double} we use the following notation for elements $\D^{\otimes J}$. If $j\in J$ and $f\in \D$ then
$f_j$ denotes the element in $\D^{J}$ that is $1$ except for in factor $j$. Since $\D = \B\otimes \A$ we use $f_{j_1}$ and $f_{j_2}$ for the $\B$ and the $\A$ parts $f_j$ so $f_j = f_{j_1}f_{j_2}$ by definition. We are now ready for the

\begin{proof} {\bf(of the main $\PG$ Theorem \ref{thm.MainPG})}\\
Equation \eqref{eq.doublemult} defines the multiplication in $\D$ and applying the functors $\OO,\G$ turns that equation into an expression for the generating function 
$(\gm_\D)^{i j}_k \in \mathcal{C}(\{i,j\},\{k\})$. We find
\[
(\gm_\D)^{i j}_k = (\gD_\B^{(2)})^{i_2}_{\bar{1}\bar{2}\bar{3}} (\gD_\A^{(2)})^{j_1}_{123}\pp(\gS_\A^{-1})_3\pp\gpi_{3,\bar{1}}\pp\gpi_{1,\bar{3}}\pp (\gm_\B)^{i_1, \bar{2}}_{k_1}\pp (\gm_\A)^{2, j_2}_{k_2}
\]
By Lemmas \ref{lem.genP} and \ref{lem.DSgen} all the building blocks of this formula are in $\PG^\pm$. According to Lemma \ref{lem.compPG} $\PG^\pm$ is closed under composition so $\gm_\D$ is also in $\PG^\pm$. Looking at the definition of $\gm_\D$ in terms of generators and relations it is clear that in fact it is in $\PG^+$.

By the same reasoning the generating functions for coproduct, co-unit and antipode in $\D$ automatically also become morphisms of $\PG^+$ because they can be expressed as a composition of morphisms in $\PG^+$. More precisely 
we can write their generating series respectively as
\[(\gD_\D)^i_{jk} = (\gD_\B)^{i_1}_{j_1k_1}(\gD_\A)^{i_2}_{j_2k_2}\pp(\gm_\D)^{j_1k_2}_k (\gm_\D)^{j_2k_1}_k\]
\[(\varepsilon_\D)_i = (\varepsilon_\B)_{i_1}(\varepsilon_\A)_{i_2}\] 
\[(\gS_\D)^i_{j} = (\gS_\B)^{i_1}_{j_1}(\gS^{-1}_\A)^{i_2}_{j_2}\pp(\gm_\D)^{j_2j_1}_j\]

Finally the spinner (group-like element) $\mathbf{C}_i = (\mathbf{AB})^{-\frac{1}{2}}$ clearly is mapped into $\PG$ by applying $\Oo^{-1}$. From this 
the ribbon element and its inverse are easily computed to be in $\PG$ using the formula
\begin{equation}
\label{eq.invribbon}
\mathbf{v}_i^{-1} = \mathbf{R}_{13}\mathbf{C}_2\pp \mathbf{m}^{123}_i
\end{equation}
The fact that we land in $\PG$ follows from the simple nature of the formula that is already in canonical order up to powers of $A$.
\end{proof}

\section{From algebra to tangle invariants 2}
\label{sec.AlgTang2}

\subsection{Rotational tangle diagrams}
\label{sub.rtangles}

In this section we introduce a variant of the Morse diagrams of tangles that are often used in discussing quantum invariants. Such diagrams serve to keep track of the rotation number of the strands
of the diagram. Normally this is done by marking the local minima and maxima (cups and caps). See for example the oriented sliced tangle diagrams in section 3.1 of \cite{Oh01} and the discussion in \cite{Tu16}. We prefer to stick to the rotation numbers themselves and propose the following definition of rotational tangle diagrams. Our rotational tangle diagrams are in some ways similar
to the rotational virtual knots of \cite{Ka15}.

\begin{definition}{\bf(Rotational tangle diagrams)}\\
The edges of a tangle diagram $D$ are the connected components of the strands after deleting a small open disk centered at each crossing.
The rotation number $\rot(e)\in \Z$ of edge $e$ is the rotation number of its tangent vector relative to the vertical\footnote{i.e. in the positive $y$ direction of the plane.} vector field.
A {\bf rotational tangle diagram} $D$ is a tangle diagram in the sense of Definition \ref{def.TangleDiagram} such that
the tangent vector at the endpoints of all the edges is vertical and it is transversal to the boundary at the end points.
\end{definition}

It should be clear that any tangle diagram can be turned into a rotational tangle diagram by applying local planar isotopies. More specifically one rotates the crossings and the end-points to make them point upwards as shown in Figure \ref{fig.XRots}.

\begin{figure}[htp!]
\begin{center}
\includegraphics[width = \linewidth]{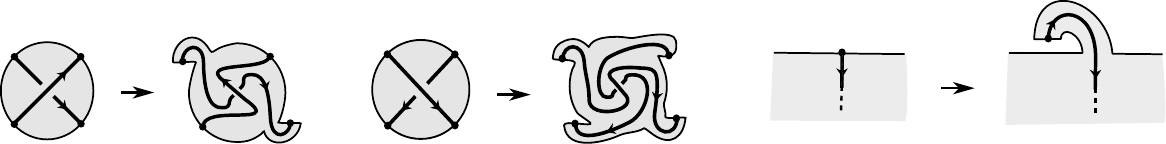}
\end{center}

\caption{Rotating crossings and endpoints to turn a tangle diagram into a rotational tangle diagram.}
\label{fig.XRots}
\end{figure}

A more elaborate example of a rotational tangle is the right hand side of Figure \ref{fig.Tangle817Rot}.

\begin{figure}[htp!]
\begin{center}
\includegraphics[width = \linewidth]{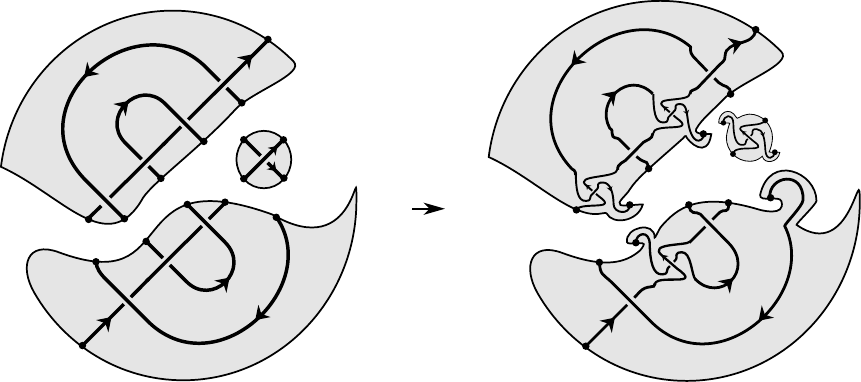}
\end{center}

\caption{Turning the $8_{17}$ tangle diagram from Figure \ref{fig.Tangle817} into a rotational tangle diagram.}
\label{fig.Tangle817Rot}
\end{figure}

As with tangle diagrams we record the simplest rotational tangle diagrams. Apart from the crossings $X_{ij},\bar{X}_{ij}$ we now also recognize two important diagrams that look like a $C$.
More precisely, $C_i$ represents a crossingless strand labelled $i$ that rotates counter-clockwise. Likewise $\bar{C}_{i}$ rotates clockwise as shown in Figure \ref{fig.XingsRot}.

\begin{figure}[htp!]
\begin{center}
\includegraphics[width = \linewidth]{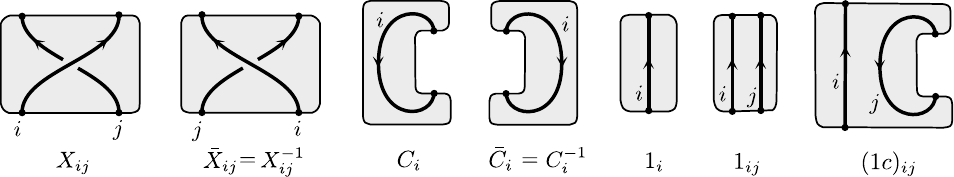}
\end{center}

\caption{The simplest rotational tangle diagrams. From left to right: $X_{ij},\bar{X}_{ij}$, $C_i$, $\bar{C}_{i}$, $1_i$, $1_{ij}$ and $(1c)_{ij}$.}
\label{fig.XingsRot}
\end{figure}

Using the above simple diagrams we can build more complicated ones by taking disjoint unions and merging strands. 
The operations disjoint union and merging mostly carry over to the rotational setting with some minor restrictions in the case of merging to take into account the rotation numbers.

\begin{definition}{\bf(Disjoint union and merging of rotational tangle diagrams)}\\
The disjoint union $DE$ of two rotational tangle diagrams is their disjoint union as tangle diagrams.

Merging $m^{ij}_k$ is defined as for tangle diagrams with the restriction that the rotation number of the arc $c$ that connects the end
of strand $i$ to the start of strand $j$ has rotation number $0$, see Figure \ref{fig.MergeRot}
\end{definition}

\begin{figure}[htp!]
\begin{center}
\includegraphics[width = 9cm]{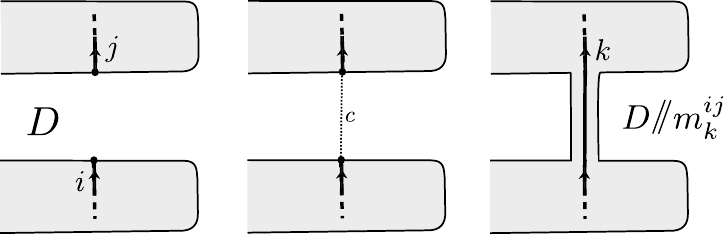}
\end{center}

\caption{Merging strands $i,j$ in rotational tangle diagram $D$.}
\label{fig.MergeRot}
\end{figure}

\begin{definition} {\bf(Equivalence of rotational tangle diagrams)}\\
Generate an equivalence relation $"="$ on rotational tangle diagrams by the following rules, where $D,E,F$ are rotational tangle diagrams:
\begin{enumerate}
\item $D=E$ if $D,E$ are planar isotopic respecting the orientation, labels on the strands and rotation numbers of the edges. 
\item If $D=E$ then $DF =EF$.
\item If $D=E$ then $D\pp m^{ij}_k = E\pp m^{ij}_k$, provided both make sense.
\item $D=E$ if $D$ and $E$ appear in one of the rotational Reidemeister equalities shown in Figure \ref{fig.RReidemeister}.
\end{enumerate}
\end{definition}

\begin{figure}[htp!]
\begin{center}
\includegraphics[width = \textwidth]{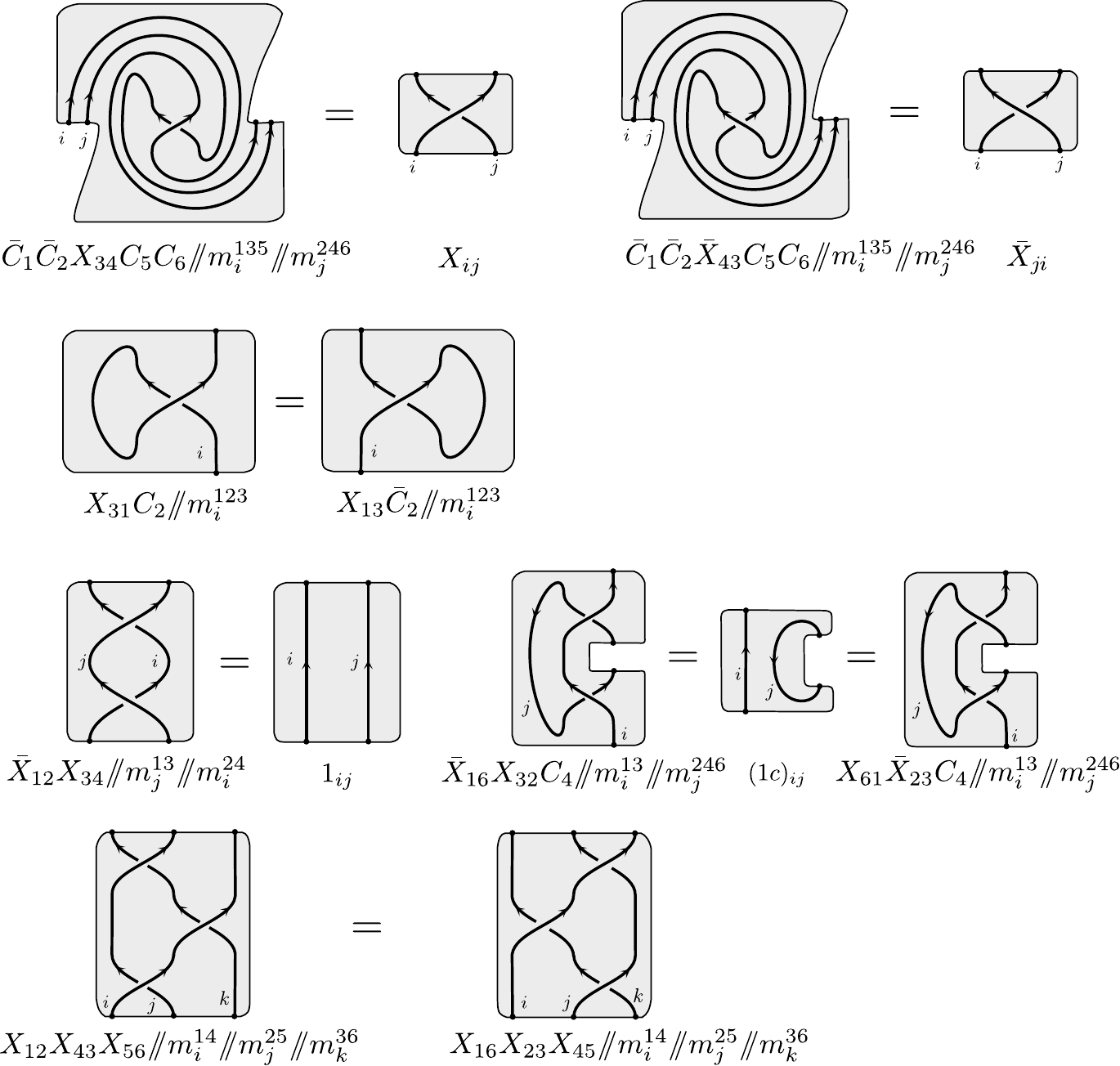}
\end{center}

\caption{The rotational Reidemeister moves, together with their algebraic description.}
\label{fig.RReidemeister}
\end{figure}

\begin{lemma}{\bf (Tangles inject into rotational tangles)}\\
For any tangle diagram $D$ there exists a rotational tangle diagram $D'$ is planar isotopic to it.

Moreover if two tangle diagrams $D,E$ are equivalent then any of the corresponding rotational tangle diagrams
$D',E'$ are equivalent (as rotational tangles) up to change of framing and rotation number at the end points.
\end{lemma}
\begin{proof}
Given tangle diagram $D$ a rotational tangle diagram planar isotopic to it is obtained by locally rotating the crossings and endpoints so they point upwards.

Any of the equivalences generating equivalence of tangle diagrams can be rotated similarly to yield a rotational equivalence. The exceptions
are the planar isotopies where one needs the swirls and the Reidemeister 1. The latter means we should ignore the framing. The rotation number at the endpoints is also arbitrary.
\end{proof}

An important reason for keeping track of the framing of our strands is that doubling a strand is a well-defined operation. 
Likewise an important reason for keeping track of the rotation numbers is to make sure strand reversal is well-defined and has good properties that
reflect the properties of Hopf algebras as we will see in the next subsection.

\begin{definition} {\bf(Strand doubling, reversal and deletion)}\\
\label{def.DiagHopfops}
Define the operations $\varepsilon,\Delta,S,\bar{S}$ on rotational tangle diagrams as follows. 
Suppose $D$ is a rotational tangle diagram that has a strand labelled $i$ and there are no strands labeled $r,\ell$.
\begin{enumerate}
\item Define $D\pp \varepsilon^i$ to be the diagram obtained from $D$ by deleting strand $i$.
\item Define $D\pp\Delta^i_{r,\ell}$ to be the diagram obtained from $D$ by choosing a small tubular neighborhood $N$ of strand $i$ 
and replacing strand $i$ by the two sides of $N$ that run parallel to $i$. The new strand to the left of $i$ is called $\ell$ and the other is called $r$.
The newly created crossings involving strands $r,\ell$ should all have the same sign as the corresponding crossing on strand $i$.
\item Define $D\pp S_i$ to be the diagram obtained from $D$ by rotating the endpoints of strand $i$ half a turn in the clock-wise direction and then 
reversing the orientation of strand $i$. To get a proper rotational diagram the crossings need to be rotated upwards as in Figure \ref{fig.XingsRot}. 
\item Define $D\pp \bar{S}_i$ to be the same as $D\pp S_i$ except that the ends are to be rotated half a turn in the counter-clockwise direction.
\end{enumerate}
\end{definition}

\begin{figure}[htp!]
\begin{center}
\includegraphics[width = \linewidth]{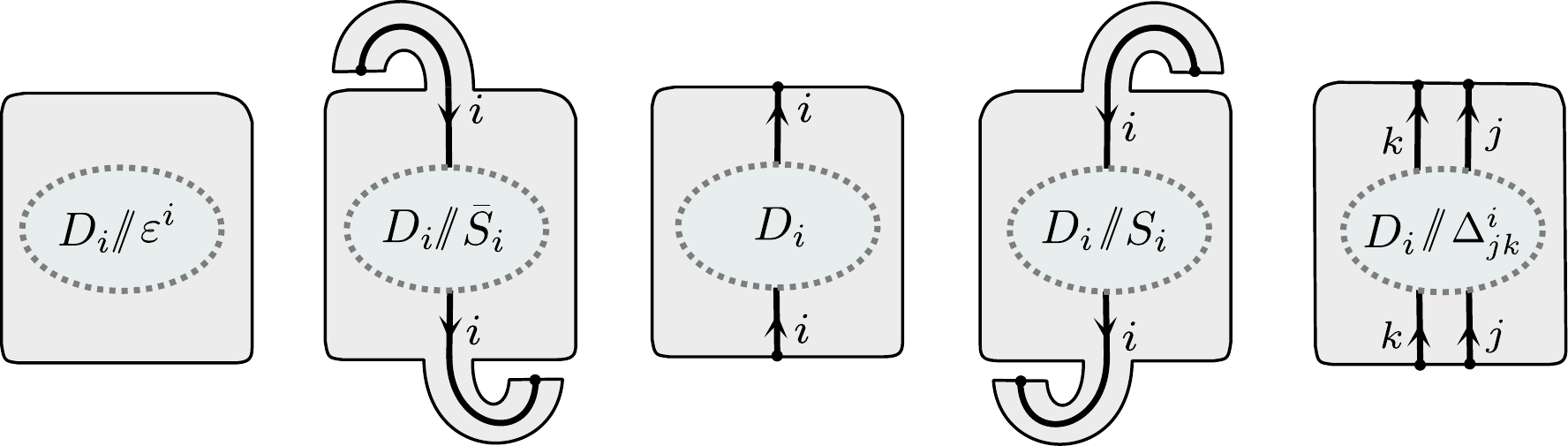}
\end{center}

\caption{A generic rotational tangle diagram $D$ (middle) together with the effect of the operations $\epsilon,\Delta,S,\bar{S}$ that delete, double and reverse strand $i$. Notice that $S$ looks like an $S$.}
\label{fig.HopfOps}
\end{figure}

Notice how the effect of the operation $S$ on a straight strand turns it into an 'S'. Less fortunate is that $\Delta^i_{r\ell}$ places the new strand $r$ to the right and the $\ell$ strand to the left relative to the framing and the orientation of strand $i$. 

As an illustration of the use of these properties we show how to build the (right-handed) Whitehead double of any long knot diagram.
Using the notation $v_i = \bar{X}_{13}C_2\pp m^{123}_i$ for the negative kink (Reidemeister $1$ curl) we set 
\begin{equation}
\label{eq.Wi}
W^i_0 = \Delta^i_{jk}\pp S_j X_{18}X_{62}v_3v_4\bar{C}_7C_5 \pp m^{1234j5678k}_0
\end{equation}

\begin{figure}[htp!]
\begin{center}
\includegraphics[width = 6cm]{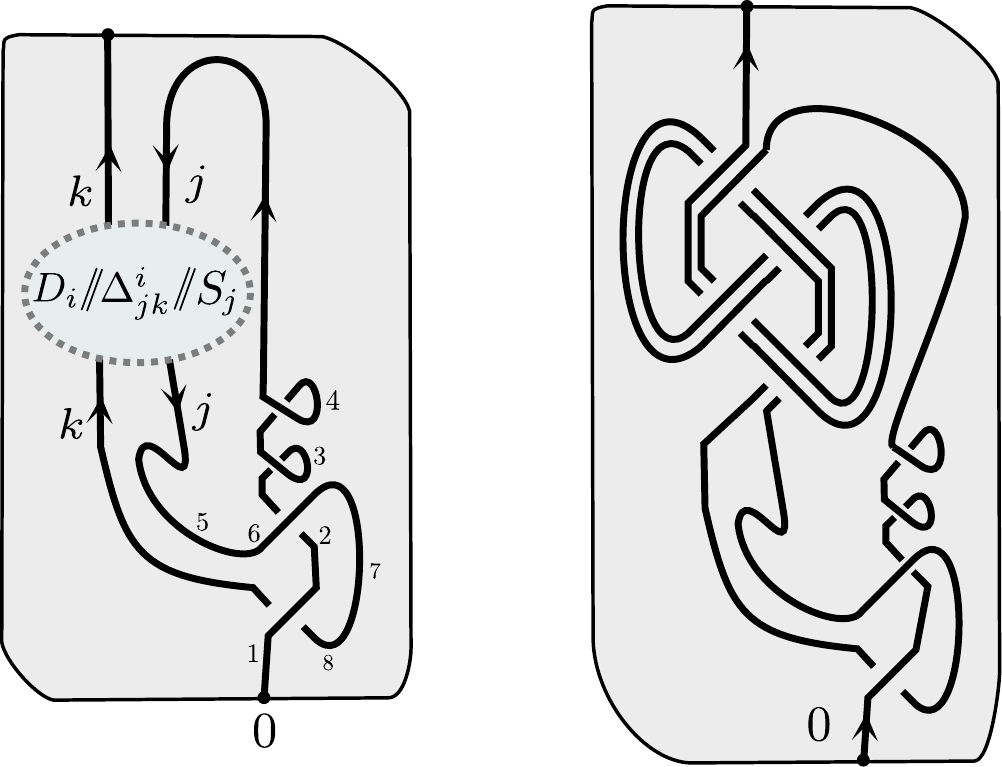}
\end{center}

\caption{Left: taking the (right-handed) Whitehead double of strand $i$ in diagram $D$. Right: the Whitehead double of the figure eight knot.}
\label{fig.WhiteheadDoubler}
\end{figure}

As illustrated in Figure \ref{fig.WhiteheadDoubler}, for any $0$-framed rotational long knot diagram $D_i$ with strand labelled $i$, a
$0$-framed rotational diagram of the Whitehead double is $D_i\pp W^i_0$.

As a more elaborate illustration we discuss Seifert surfaces in a way similar to \cite{Ha06}.

\begin{lemma} {\bf(Seifert surface criterion)}\\
\label{lem.SeifertCrit}
Define 
\begin{equation}
\label{eq.Bandersnatch}
\mathcal{B}^{ij}_k = C_3C_4\Delta^i_{r_1 \ell_1 }\Delta^j_{r_2\ell_2}\pp \bar{S}_{r_1}\pp S_{r_2}\pp m^{\ell_1 r_2 3 4 r_1 \ell_2}_{k}
\end{equation}
If rotational long knot diagram $K$ represents the boundary of a genus $g$ Seifert surface then there exists a $0$-framed rotational tangle diagram $L$ with $2g$ strands
named $1,\dots 2g$ such that
\begin{equation}
\label{eq.Seif}
K_1 = L\pp_{j=1}^{g}\mathcal{B}^{2j-1, 2j}_{j} \pp m^{12\dots g}_1
\end{equation}
\end{lemma}

\begin{proof}
Given a Seifert surface for a knot $K$ recall (e.g. \cite{BZ03} p.107) that can we bring it into band form where the surface looks like a disk with $2g$ bands attached as shown in Figure \ref{fig.Seifert1}.
The bands may be described by a blackboard framed tangle $L$ consisting of the cores of these bands.

\begin{figure}[htp!]
\begin{center}
\includegraphics[width = 10cm]{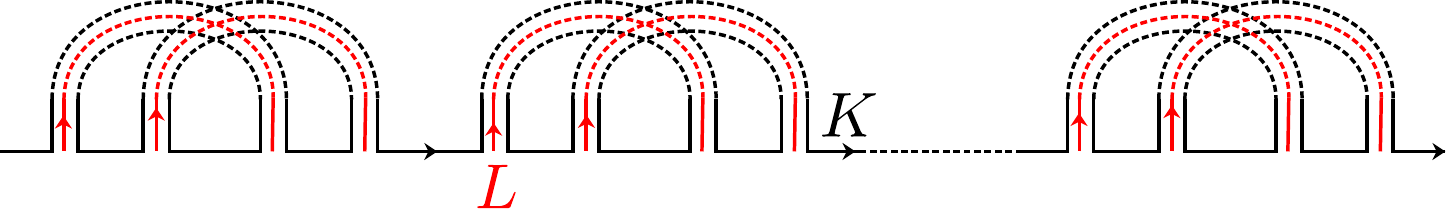}
\end{center}

\caption{A schematic picture for the Seifert surface of knot $K$ as a disk with $2g$ bands attached. In red we show the cores of the bands that make up the tangle $L$.}
\label{fig.Seifert1}
\end{figure}

To obtain the knot $K$ from the tangle we need to thicken each strand using the $\Delta$ operation and then fix the orientation using $S,\bar{S}$. We need to correct the half turns coming from the $S$ and one way to do this is include two copies of $C$ in the middle as shown in Figure \ref{fig.Seifert2} in the special case of the figure eight knot. Connecting the strands in the correct way
leads us to consider the operation $\mathcal{B}^{ij}_k$ defined above.
\end{proof}

\begin{figure}[htp!]
\begin{center}
\includegraphics[width = \textwidth/2]{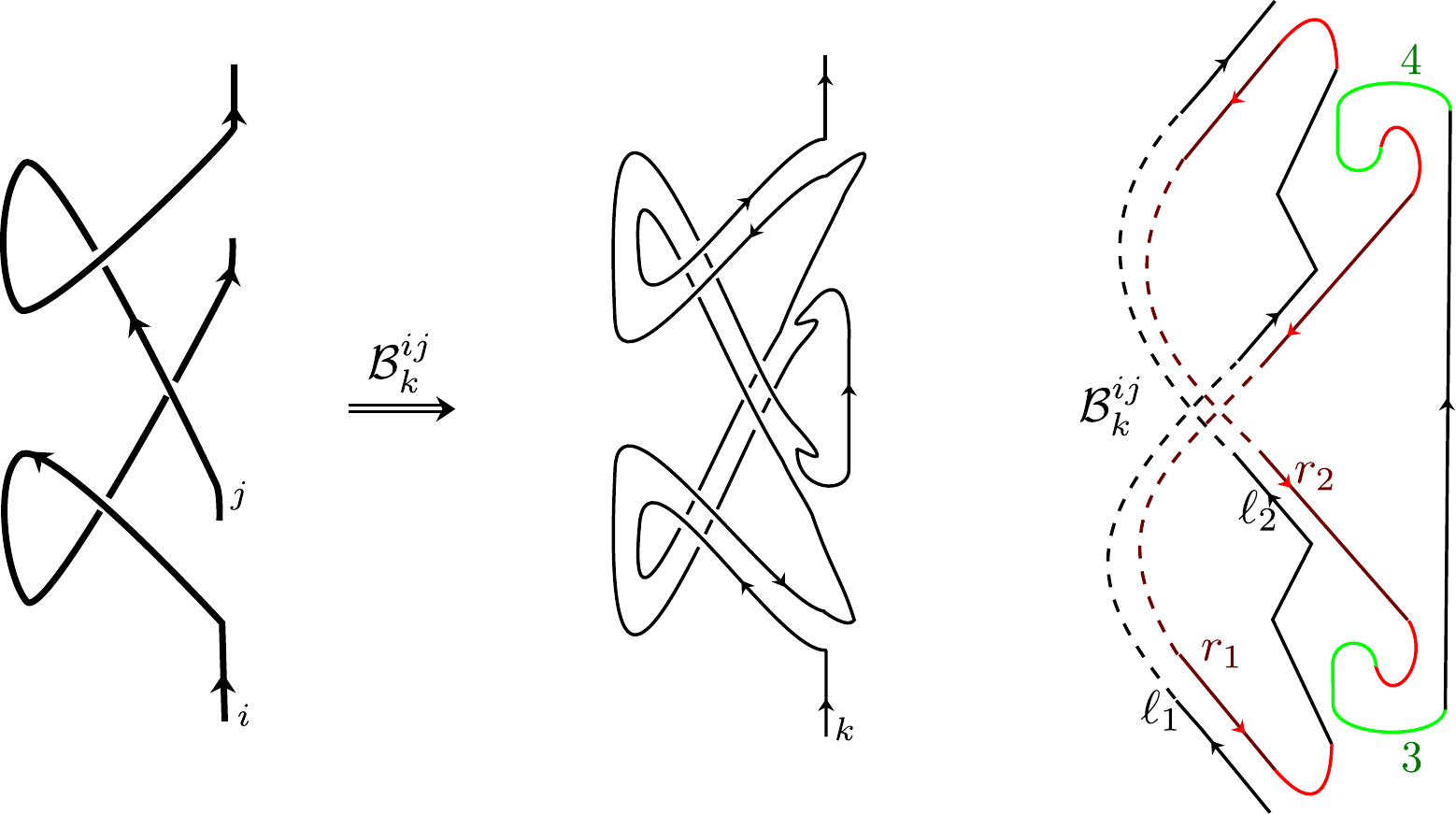}
\end{center}

\caption{Left: the $2g$-component tangle encoding the bands of a Seifert surface for the figure eight knot. Here the genus $g=1$ and thickening the components using $\mathcal{B}^{ij}_k$ reconstructs the boundary of the surface shown in the middle. Notice the cups and caps are organised so as to reflect the $S,\bar{S}$ and $C^2$ in the formula. Right: The formula of $\mathcal{B}^{ij}_k$ is illustrated more explicitly with the cups and caps belonging to the $S,\bar{S}$ shown in red and the two $C's$ in green.}
\label{fig.Seifert2}
\end{figure}

\subsection{Universal invariant of rotational tangles}

In this section we sketch how a ribbon Hopf algebra gives rise to a knot invariant and more generally an invariant of tangles known as the universal invariant (corresponding to that algebra). A general reference for what follows is \cite{Oh01}. Our purpose here is just to establish notation for the universal invariant. When we speak about knots we will always mean one-strand rotational tangles (long knots).

\begin{definition}
\label{def.uiZ}
Suppose $A$ is a ribbon Hopf algebra $A$ with unit $\bI$ and multiplication $\mathbf{m}^{ij}_k:A^{\otimes \{i,j\}}\to A^{\otimes \{k\}}$ with universal $R$-matrix $\mathbf{R}_{ij}$
and spinner $\bC$ as in subsection \ref{sub.qtriang}. For a rotational tangle diagram $D$ whose strands are labeled by set $L$ define $\bZ_A(D)\in A^{\otimes L}$ by the following rules.
\begin{enumerate}
\item If $D$ does not have crossings, $\bZ_{A}(D) = \prod_{\ell\in L} \bC_\ell^{\rot(\ell)}$. (crossingless diagrams).
\item $\bZ_{A}(X_{ij}^{\pm 1}) = \mathbf{R}^{\pm 1}_{ij}$ (value of the crossings).
\item If diagram $E$ is labelled by set $M$ then $\bZ_{A}(DE) = \bZ_A(D)\otimes \bZ_{A}(E)\in A^{\otimes L\sqcup M}$\\ (disjoint union is tensor product).
\item $\bZ_{A}(D\pp m^{ij}_k) = \bZ_{A}(D)\pp \mathbf{m}^{ij}_k$ (merging is multiplication).
\end{enumerate}
\end{definition}

For convenience one often restricts to finite dimensional Hopf algebras but the theorem is equally valid in the topological case, see for example \cite{Ha08}.
Associativity of the multiplication in $A$ assures us that $\bZ_A(D)$ is independent of the way we split diagram $D$ into elementary pieces.

For convenience we list the main features of the universal invariant phrased in our language. All these features are well known, see for example section 7 of \cite{Ha06}.

\begin{theorem}{\bf (Properties of the universal invariant)}
\label{thm.propofZ}
\begin{enumerate}
\item If $D$ and $D'$ are equivalent then $\bZ_A(D) = \bZ_A(D')$ (invariance).  
\item $\bZ_A(D\pp \Delta^i_{jk}) = \bZ_A(D)\pp \mathbf{\Delta}^i_{jk}$ (strand doubling).
\item $\bZ_A(D\pp S^{\pm 1}_i) = \bZ_A(D)\pp \bS_i^{\pm 1}$ (strand reversal).
\item $\bZ_A(D\pp \varepsilon^i)=\bZ_A(D)\pp \boldsymbol{\varepsilon}^i$ (strand removal).
\item If $D$ has a single strand then $\bZ_A(D)\in \mathcal{Z}(A)$ (centrality).
\end{enumerate}
\end{theorem}

In Section \ref{sec.structure} we will use these properties to connect the knot invariant $\bZ_\D$ to the knot genus.  Applying the universal invariant we see that if $K$ is a Whitehead double then
$\bZ_A(K) = \bZ_A(D)\pp \bZ_A(W_i)$. 

What we can say already is that since the universal invariant of a rotational tangle diagram is obtained by multiplying $R$-matrices the results of Section \ref{sec.GenD} tell us that
$Z_\D$ and everything that relates to it takes place in the category $\PG$.

\begin{corollary}
For any rotational tangle diagram $D$ whose strands are labelled by set $L$ we have $Z_\D(D)\in \PG(\emptyset,L)$.
\end{corollary}

To end this section we remark that all the axioms for a ribbon Hopf algebra $A$ have an intuitive interpretation in terms of rotational tangle diagrams and the universal invariant.
For example let us take a look at the formulas that the ribbon element $\bv$ should satisfy, see Definition \ref{def.ribbonHopfAlg}.
$\bv$ should be a central square root of $\bu \bS(\bu)$ so let us first see what $\bZ_A(\bu)$ and $\bZ_A(\bS(\bu))$ look like.
Using the definition of $\bu$ and $\bR_{12}\pp\bS_1\pp\bS_2=\bR_{12}$ and $\bR_{12}\pp\bS_1=\bar{\bR}_{12}$ we find
\[
\bZ(\bu_1)=\bZ(\bR_{12}\pp\bS_2\pp\bm^{21}_1) = \bar{\bR}_{12}\pp\bS^2_2\pp\bm^{21}_1 = \bZ_A(\bar{X}_{12}\pp S^2_2\pp m^{21}_1)
\] and likewise
\[
\bZ(\bS(\bu)_1) = \bR_{12}\pp\bS_2\pp\bm^{21}_1 \pp \bS_1 = \bar{\bR}_{12}\pp\bS_2^2\pp\bm^{12}_1   = \bZ_A(\bar{X}_{12}\pp S^2_2\pp m^{12}_1)
\]
This means that $\bZ(\bu)$ and $\bZ(\bS(\bu))$ are the universal invariants of the two tangles shown in Figure \ref{fig.uSu1}.

\begin{figure}[htp!]
\begin{center}
\includegraphics[width = 7cm]{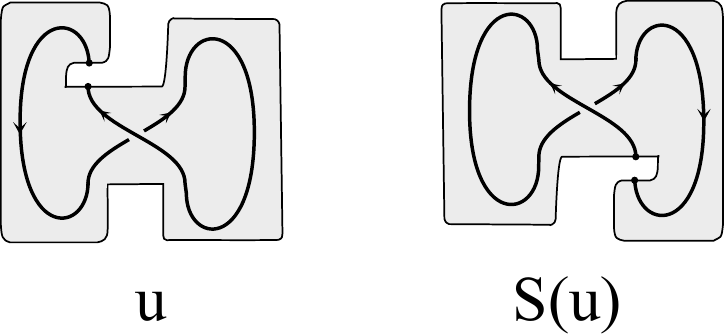}
\end{center}
\caption{Two rotational tangle diagrams whose universal invariants are $\bu$ and $\bS(\bu)$.}
\label{fig.uSu1}
\end{figure}

Merging the tangle diagrams corresponds to multiplying the corresponding universal invariants so the tangle for the product $\bS(\bu)\bu = \bv^2$ is shown in Figure \ref{fig.uSu2}.
Notice how the curls canceled out and two negative Reidemeister 1 kinks are left one with positive and one with negative rotation number.
According to the Reidemeister moves of our rotational tangles both positive kinks are equivalent, see the second line of Figure \ref{fig.RReidemeister}. It follows (from Reidemeister 2 and 3) that the two
negative kinks shown here are also equivalent. Either of them take the value $\bv$ when applying the universal invariant. Notice how this agrees with the equation $\bv = \bC^{-1}\bu$
where one takes the tangle for $\bu$ shown and prepends the inverse of $\bC$ to remove the curl and is left with the negative kink. 

\begin{figure}[htp!]
\begin{center}
\includegraphics[width = 7cm]{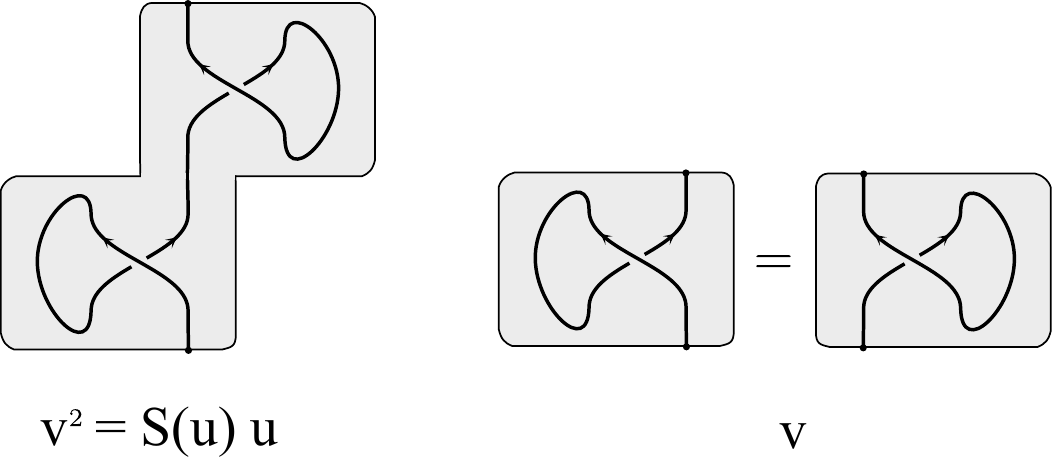}
\end{center}

\caption{Diagrams corresponding to $\bS(\bu)\bu = \bv^2$ and $\bv$.}
\label{fig.uSu2}
\end{figure}

As a further illustration of this graphical interpretation we consider the most complicated of the axioms of for the ribbon element:
\[
 \mathbf{\Delta}^{i}_{r,l}(\bv_i) = \bv_5\bv_6 \bR_{21}^{-1}\bR_{43}^{-1}\pp \bm^{145}_l\pp \bm^{236}_r\]
If we denote the diagram for the negative kink by  $v_i$ so that $\bZ(v_i) = \bv_i$ then we may interpret this equation
in terms of Figure \ref{fig.DeltaKink}. Indeed , the left hand side of the equation is the value of the left hand picture,
while the right hand side corresponds to the picture on the right. The labels $1,2,3,4,5,6$ are just for convenience,
the final two strands are called $l,r$. The crux here is that the two pictures shown are in fact isotopic or more precisely,
 equivalent under our Reidemeister moves for rotational tangle diagrams.

\begin{figure}[htp!]
\begin{center}
\includegraphics[width = 7cm]{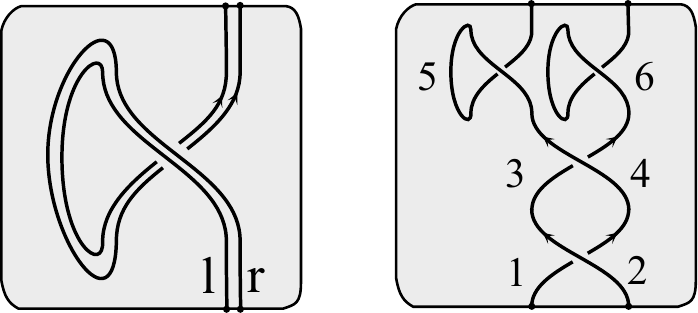}
\end{center}

\caption{Diagrams corresponding to the coproduct of the ribbon element.}
\label{fig.DeltaKink}
\end{figure}

The reader is invited to draw pictures for the remaining ribbon Hopf algebra axioms. In passing we remark that it is possible to extend the graphical calculus a bit further
to obtain a knot theoretical interpretation for the Drinfeld double construction and will report on that in a future work.

\section{Properties of the knot invariant $\bZ_\D$}
\label{sec.structure}

In this section we investigate the knot invariants $\bZ_\D$ both from a practical and from a theoretical point of view using the tools we developed in the previous sections.
We start with a discussion of the center of $\D$.

\subsection{The center of $\D$}

Restricted to knots, the invariant $\bZ_\D$ always takes values in the center $\mathcal{Z}(\D)$, see Theorem \ref{thm.propofZ} part 5). 
We therefore set out to determine the center of $\D$ before turning to a discussion of the properties of the invariant. 

\begin{theorem} {\bf (Center of $\D$)}
\label{thm.ZD}
\begin{enumerate}
\item $\mathbf{t} = \mathbf{b}-\epsilon \mathbf{a}$ is central. Also introduce $\mathbf{T} = e^{-\hbar \bt}$.
\item The element $\mathbf{w} = \mathbf{y}\mathbf{A}^{-1}\mathbf{x}+\frac{q\mathbf{A}^{-1}+\mathbf{AT}-\frac{1}{2}(\mathbf{1}+\mathbf{T})(q+1)}{\hbar(q-1)}$ is central and satisfies $\mathbf{w} = \mathbf{y}\mathbf{x}+g(\mathbf{a},\mathbf{t}) \mod \hbar$
for some power series $g$.
\item
Suppose $\tilde{\mathbf{w}}\in \D$ is any central element of the form $\tilde{\mathbf{w}}=\mathbf{y}\mathbf{x}+g(\mathbf{a},\mathbf{t}) \mod \hbar$,
for a power series $g$. The center $\mathcal{Z}(\D)$ of $\D$ is generated by $\bt,\tilde{\mathbf{w}}$ as an algebra over $\Q_\hbar[\eps]$.
\end{enumerate}
\end{theorem}
\begin{proof}
For part 1)
Centrality of $\bt$ (and hence $\mathbf{T}$) is readily checked using the commutation relations for $\D$ found in Lemma \ref{lem.D}.  
To check that $\mathbf{w}$ commutes with $\bx$ recall that $\bx\mathbf{A} = q \mathbf{A}\bx$ so 
\[
\bx\big( \by \mathbf{A}^{-1} \bx+\frac{q\mathbf{A}^{-1}+\mathbf{AT}}{\hbar (q-1)}\big) =\big( \by \mathbf{A}^{-1}\bx+\frac{1-\mathbf{AT}}{\hbar}+\frac{\mathbf{A}^{-1}+q\mathbf{AT}}{\hbar (q-1)}\big)\bx=\big( \by \mathbf{A}^{-1} \bx+\frac{q\mathbf{A}^{-1}+\mathbf{AT}}{\hbar (q-1)}\big) \bx
\]
Checking the commutation with $\by$ and $\ba$ is left to the reader. Expanding $\mathbf{A,T}$ as series in $\hbar$ we see that
$\mathbf{w} = \mathbf{y}\mathbf{x}+(\mathbf{a}+\frac{1}{2})\mathbf{t}+2\eps a(\mathbf{a}+1) \mod \hbar$.

For part 2) suppose for a contradiction there exists a $\mathbf{z}\in \mathcal{Z}(\D)$ which is not in generated by $\tilde{\mathbf{w}},\mathbf{t}$. 
Say $\mathbf{z} = \sum_{n}p_n(\mathbf{y},\mathbf{t},\mathbf{a},\mathbf{x})\hbar^n$, where $p_n$ are ordered polynomials in ybax order. 
There must be a least $n$ such that $p_n$ cannot be made $0$ by subtracting monomials in $\mathbf{t},\tilde{\mathbf{w}}$. Let us focus on the monomials with the highest power of $\mathbf{y}$ in $p_n$.
Commuting with $\mathbf{a}$ we see that such monomials must be of the form $c\mathbf{y}^j\mathbf{t}^s\mathbf{a}^u\mathbf{x}^j$ for some scalar $c$. Commuting with $\mathbf{x}$ further shows that $u$ must be $0$.
This means that we may assume $p_n$ does not depend on $\mathbf{x}$ or $\mathbf{y}$ at all because we may always subtract $c\mathbf{t}^s\tilde{\mathbf{w}}^j$ from the $p_n$. Finally commuting with $\mathbf{x}$ again we see that $p_n$ must be a polynomial in $\mathbf{t}$. This brings us to our desired contradiction as we may subtract this polynomial from $p_n$ to find that $n$ did not have the promised property.
\end{proof}

Apart from the $\mathbf{w}$ we introduced above there are many other central elements congruent to $\mathbf{y}\mathbf{x}+g(\mathbf{a},\mathbf{t}) \mod \hbar$.
For example perhaps a more canonical choice would be to take the logarithm of the inverse ribbon element: $\frac{\mathbf{1-T}}{\hbar^2 \mathbf{t}}\log \mathbf{v}^{-1}$.
In passing we remark that it is of the form $\mathbf{y}\mathbf{x}+(\mathbf{a}+\frac{1}{2})(1-\mathbf{T}) \mod\eps$ because of the formula 
\[e^{\Oo(\lambda y x+\mu a t)}=\Oo(e^{\frac{\hbar e^{-\mu t}}{1-T}(e^{\frac{\lambda(1-T) }{\hbar}}-1)yx+\mu at})\mod \eps\]
which can be proven noting that both sides satisfy the same differential equation in $\lambda$ just like the proof of Lemma \ref{lem.RHeis}.
However our choice $\mathbf{w}$ seems more practical in that it clearly involves only $\mathbf{T}$ and not $\bt$ so we will stick with it for now.

\subsection{Main results on $Z_\D$}

In this section we illustrate our Hopf algebra techniques by proving the following results about the knot invariant $\bZ_\D$. 

When it is convenient we will use the shorthand $Z=Z_\D=\Oo^{-1}\bZ_\D$. Here $\Oo$ always refers to the ordering map sending monomials in the variables $y,b,a,x$ to ordered monomials in $\mathbf{y,b,a,x}$.
Often $b$ will be replaced by $t = b-\eps a$. This replacement is harmless because $\mathbf{a,b}$ commute and are next to each other in the chosen ordering. Also recall $T = e^{-\hbar t}$. 

\begin{theorem}
\label{thm.compZD}
Set the width $\omega(K)$ to be the maximal number of strands in a sequence of tangles used to construct $K$ using merging and disjoint union operations.
Suppose $K$ is a $0$-framed knot $K$ with crossing number $n$ and width $\omega$. One can compute $\bZ_\D(K)$ up to order $\kappa\geq 1$ in $\epsilon$, 
in $\OO(\omega^{4\kappa}n^2\log n) \leq \OO(n^{2\kappa+2}\log n)$ integer operations. 
When $\kappa=0$ the required number of operations in $\Z$ is $\OO(\omega^{2}n^2\log n)\leq \OO(n^3\log n)$.
\end{theorem}

\begin{theorem}
\label{thm.structZ}
Denote the normalized Alexander polynomial of knot $K$ by $\Delta_K$. There exist polynomials $\rho_{k,j}(K)\in\mathbb{Z}[T,T^{-1}]$ such that for any
$0$-framed knot $K$: 
\[\bZ_\D(K) = \frac{1}{\Delta_K}\exp\Big(\sum_{k=1}^{\infty} \eps^k \sum_{j=0}^{2k} \rho_{k,j}(K)\frac{\mathbf{w}^j}{\Delta_K^{2k-j}}\Big)\]
\end{theorem}

Based on both the theory of the loop expansion of the Kontsevich invariant and experiments we expect the invariants $\rho_{k,0}$ to coincide with the
$(k+1)$-loop invariant of $sl_2$ as introduced by Rozansky \cite{Ro98}. A precise proof of this claim would take us too far afield.

\begin{theorem}
\label{thm.UQ}
$Z_\D$ determines the universal quantum $\mathfrak{sl}_2$ invariant and hence all colored Jones polynomials.
\end{theorem}
\begin{proof} It follows from Lemma \ref{lem.isoUh} that we may set $\epsilon=1$ for a long knot $K$ the universal $\mathfrak{sl}_2$ invariant is $\phi Z_\D(K)|_{\epsilon=1}$. Evaluating this invariant in the $N$-th dimensional irreducible representation in turn provides the $N$-colored Jones polynomial, \cite{Oh01}. \end{proof}

To first order in $\eps$ only the part $\rho_{1,0}$ is new in $Z_\D$. The rest is determined by the Alexander polynomial as shown in Theorem \ref{thm.propZ} below. 
It should be noted that the pair $\Delta_K,\rho_{1,0}(K)$ distinguishes all prime knots in the Rolfsen table up to ten crossings. Out of the $2978$ prime knots with $12$ crossings
our pair takes $2883$ distinct values. As a comparison the pair Khovanov Homology, HOMFLY polynomial takes only $2786$ distinct values on this set.
$\rho_{1,0}$ also appeared in \cite{BV19} but there we chose to set $\mathbf{w} = 0$. Some of the conjectures formulated in that paper are proven below.

Recall that the Alexander polynomial provides a lower bound on the knot genus as $\mathrm{deg}_T \Delta_K(T) \leq g_K$. The $\rho_{1,0}$ invariant provides a similar bound.
By the degree of a Laurent polynomial in $T$ we mean the exponent of the highest power of $T$ so for example $\deg(T+T^{-5}) = 1$.

\begin{theorem} {\bf(Properties of the knot invariant $\bZ$)}, See subsection \ref{sub.Seifert} for the proof.
\begin{enumerate}
\item $\rho_{1,1}(K)(T) = \frac{2\hbar T}{1-T}\frac{\mathrm{d}}{\mathrm{d}T}\Delta_K(T)$
\item $\rho_{1,2}(K) = 0$.
\end{enumerate}
\label{thm.propZ}
\end{theorem}

\begin{theorem}{\bf(Genus bound)}, see Subsection \ref{sub.Seifert} for the proof.\\
\label{thm.genus}
If $g(K)$ is the genus of $0$-framed knot $K$ then
\[\mathrm{deg}_T \rho_{1,0}(K)(T) \leq 2g_K\]
\end{theorem}

This new genus bound is sometimes sharper than the Alexander bound. For example for the 11-crossing prime knots there are
seven knots where the Alexander bound is not sharp. For five of those $\rho_{1,0}$ is an improvement.

\begin{table}[htp!]
\begin{center}
\begin{tabular}{c|ccc}
knot $K$& genus & $\frac{1}{2}\mathrm{deg}_T \rho_{1,0}(K)$ & $\mathrm{deg}_T \Delta_K$ \\
\hline 
11n34  & 3 & $3/2$ & 0\\ 
11n42  & 2 & $3/2$ & 0\\
11n45  & 3 & 2 & 2\\
11n67  & 2 & $3/2$ & 1\\
11n73  & 3 & $3/2$ & 2 \\
11n97  & 2 & $3/2$ & 1 \\
11n152 & 3 & 2 & 2
\end{tabular}
\end{center}
\end{table}

On the negative side we expect that $\bZ_\D$ will not detect the genus since it expected to be mutation invariant like the colored Jones polynomial.
As one can see in the table the mutant pair of Conway and Kinoshita-Terasaka (11n34 and 11n42 in the table) have different genus. 

As a further illustration of the use of $\bZ_\D$ we formulate the following criterion for knots being the Whitehead double of some knot.
Using normal surface theory Whitehead doubles can also be detected but this generally takes exponential time. 

\begin{theorem} {\bf(Whitehead double criterion)}\\
\label{thm.WDouble}
Suppose $K$ is a $0$-framed knot that is the Whitehead double of knot $L$ in the sense of Section \ref{sec.AlgTang2}.
We must have $\Delta_K(T)=1$ and
\[\rho_{1,0}(K) = -4v_2(L)\hbar^{-2}(T^{\frac{1}{2}}-T^{-\frac{1}{2}})^2\] 
where $v_2(L)$ is the Vassiliev invariant of order $2$ of knot $L$ \cite{PV99}.
\end{theorem}
\begin{proof}
Use the equation $Z_\D(K) = Z_\D(L_i)\pp Z_\D(W_i)$. Here $W_i$ is the Whitehead doubling operation. As $Z_\D(W_i)$ does not depend on $\tau$ and $Z_\D(L) = \frac{1}{\Delta_L(T)} \mod \eps$ we see that $\Delta_K = 1$. 

Moving on to the first order in $\eps$ we note that at $T=1$ we have $Z_\D(L)(1) = 1+2\eps v_2(L) y x +\OO(\eps^2)$.
This follows by asserting that the coefficient of $\eps$ in $Z_\D(L)(1)$ is in the ideal generated by $x,y,a$ as is clear from 
investigating $m^{ij}_k$ at $T=1$ up to order $\eps$. Next we know that $Z_\D(L)(T) = \frac{2T \Delta'(T) a+ 2T \Delta'(T)yx(1-T)^{-1}+g(T)}{\Delta^2(T)}$
with $g(1) = 0$. Since $\Delta_L(1)=1$ and $\Delta_L(T^{-1})=\Delta_L(T)$ we must have $\frac{d}{dT}\Delta_L(1) = 0$.
Therefore by 'l Hospital we get $Z_\D(L)(1) = -4T \Delta_L''(1)yx$. As $\Delta_L''(1) = 2v_2(L)$ we are done.
\end{proof}

We end this subsection with a few conjectures on how our results might extend to higher order.

\begin{conjecture} 
Imagine a $0$-framed knot $K$.
\begin{enumerate}
\item $\rho_{k,j}(K)$ with $j\geq 1$ is determined by $\{\rho_{i,r}(K)|i< k,\ r\leq 2k-2\}$. In fact
\item $\rho_{k,j}(K) = 0$ for any $j>k$.
\item $\rho_{k,j}(\bar{K}) = (-1)^{k+j}\rho_{k,j}(K)$, where $\bar{K}$ denotes the mirror image of $K$.
\item $\rho_{k,j}(K)(T^{-1}) = (-1)^{j}\rho_{k,j}(K)(T)$
\end{enumerate}
\end{conjecture}

We also expect a linear function in the genus to bound the degree of $\rho_{k,0}$ for higher $k$. For example the table in Appendix \ref{sec.KnotTable} suggests
that $\deg_T \rho_{2,0} \leq 4g$. Keep in mind that the $\rho^+_{2,0}$ listed there is the divided by a factor $(1-T)^2/T$.

\subsection{Computational complexity of the knot invariant $Z_\D$}
In this section we aim to prove Theorem \ref{thm.compZD}.

Recall the width $\omega(K)$ of a knot $K$ means the maximal number of strands in a sequence of tangles used to construct $K$ using merging and disjoint union operations.
This is related but not precisely the same as Gabai-width and tree width. Nevertheless it follows from \cite{LT79} that $\omega(K) =  \OO(\sqrt{n})$ for an $n$-crossing knot $K$.

\begin{lemma}
\label{lem.complexity}
If $K$ is a knot with crossing number $n$ and width $\omega$ then the computation of $Z_\D(K)$ up to order $\kappa\geq 1$ in $\epsilon$, takes at most $\OO(n\omega^{4\kappa})$ ring operations in $\Q(T_1,\dots, T_\omega)$. When $\kappa = 0$ the number of operations in this ring is at most $\OO(n\omega^{2})$.
\end{lemma}
\begin{proof}
To construct the knot we may disregard disjoint unions and $C$'s by incorporating them into the merges and the crossings. What remains is to carry out $\OO(n)$ strand merges.
Each merge involves at most $\omega$ strands by definition of width. Suppose we merge two strands $i,j$ and call the set of the remaining strand labels $Pass$. Composing with a merging operation $m^{ij}_k$ can be written as a multiplication followed by eight single variable contractions in $t_i,t_j,a_i,a_j,x_i,x_j,y_i,y_j$. By the Contraction Theorem each of these nine operations involves a bounded number of operations in a polynomial ring $\mathcal{R}_{Pass}$ over $\Q(T_{Pass})$ in the variables $x_{Pass},y_{Pass},a_{Pass}$. The weight $\wt$ of any monomial involved is at most $4\kappa$ because all the expressions we encounter are in $\PG$, see Theorem \ref{thm.MainPG}. In addition the Gaussian part satisfies $\wt=2$ which is important only when $\kappa=0$. A single ring operation in $\mathcal{R}_{Pass}$ takes $\OO(\omega^{4\kappa})$ ring operations in $\Q(T_{Pass})$. All in all we see that we need $\OO(n\omega^{4\kappa})$ ring operations in $\Q(T_{Pass})$ for the perturbation and $\OO(n\omega^{2})$ for the Gaussian part as claimed.
\end{proof}

Computations in rings of multivariable rational functions $\Q(T_1,T_2,\dots T_\omega)$ still take rather long. The number of integer operations will be exponential in $\omega$ so this is no good news for practical computations. Fortunately in the case of knots\footnote{Or more generally tangles with a bounded number of strands.} in the end only one variable $t$ remains and since $t$ is central we might as well drop the subscript of $t$ and specialize to a single $t$ from the very start of the computation. As long as we only apply algebra operations this gives precisely the same result. 

To make real progress we also need to bound the degree of the rational functions in $T$ that occur. By the degree of a rational function we mean the maximum of the
degree of the denominator and the numerator after dividing out common factors. For the estimation of the degree we use the full power of the Contraction Theorem as follows.

\begin{lemma}
\label{lem.denominator}
If $K$ is a tangle that can be constructed by merging $n$ crossings then the $T$-degree of the denominator and numerator of the coefficients of $Z_\D(K)$ computed to order $\eps^\kappa$ is $\OO(n)$.
\end{lemma}
\begin{proof}
We first contract the $a$-variables all at once using the contraction theorem and notice that the Gaussian is upper triangular so that no denominator occurs. Next contract the remaining $x,y$ variables. The Gaussian matrix has size $\leq 8n$ and its entries are never more than linear in $T$. Therefore the resulting determinant is of degree at most $8n$. The degree of the perturbation
is at most $4\kappa$ in $x,y,\xi,\eta$ so we find a denominator of at most $8n+24\kappa n$
\end{proof}

Collecting all these improvements on Lemma \ref{lem.complexity} we arrive at our final estimate of the complexity of the knot invariant $Z_\D$.

\begin{proof}(of Theorem \ref{thm.compZD})\\
By the above comments we can apply Lemma \ref{lem.complexity} with the ring $\Q(T)$ instead of its multivariate version. The Lemma \ref{lem.denominator} furthemore assures us that
the degree of any two rational functions involved has degree bounded by $\OO(n)$. Multiplying two such rational functions takes at most $\OO(n\log n)$ integer operations and so
the whole computation takes at most $\OO(n\omega^{4\kappa}n\log n)$ integer operations for the perturbation. For the Gaussian we likewise obtain $\OO(n\omega^{2}n\log n)$. We already remarked that $\omega(K) =  \OO(\sqrt{n})$ by \cite{LT79} thus finishing the proof. 
\end{proof}

\subsection{Computations using a Seifert surface}
\label{sub.Seifert}
In this subsection we will lay the foundation for the proof of Theorems \ref{thm.structZ}, \ref{thm.propZ} and \ref{thm.genus}. All these theorems will be proven using
the same technique. The fact that the knot $K$ must bound a Seifert surface can be expressed in Hopf algebra terms using Lemma \ref{lem.SeifertCrit} and the properties of the universal invariant.
We will start by setting up the argument in general and then pass to the simplest case where $\eps = 0$ to find the Alexander polynomial. Once that is done we
will recycle the argument to see what it tells us about $\bZ_\D$ to higher orders in $\eps$. To simplify the formulas we will omit $\hbar$. By the condition $\wh=0$ we can always
restore the necessary power of $\hbar$ when it is desired.

Recall from Lemma \ref{lem.SeifertCrit} that we can bring any Seifert surface for $K$ in band form. We assume the surface $\Sigma$ is of genus $g$ and is obtained by attaching $2g$ bands to a single disk in pairs of two as explained at the end of Subsection \ref{sub.rtangles}. 
The cores of the bands define a framed tangle $L$ with $2g$ strands. For example the figure eight knot and its Seifert surface in band form are shown in Figure \ref{fig.Seifert2}.

To reconstruct the boundary of the Seifert surface $K=\p \Sigma$ from the tangle $L$ we need to thicken the bands and merge them properly as
in Figure \ref{fig.Seifert2}. Lemma \ref{lem.SeifertCrit} makes this precise as: 
\[K = L\pp_{j=1}^{2g}\mathcal{B}^{2j-1,2j}_j\pp m^{1,2\dots g}_1\]
For example in the figure eight knot case the tangle $L$ has two components and if we call them $1$ and $2$ then we have $L = v_1\bar{X}_{2,3}\bar{v}_4\pp m^{13}_1\pp m^{24}_2$, where $v$ encodes the negative kink and $\bar{v}$ the positive one. 
Also the boundary of the Seifert surface is $K_1=\p \Sigma = L\pp \mathcal{B}^{12}_1$ is a diagram for the figure eight knot.

Applying the universal invariant $\bZ_\D$ to both sides we can use the properties of $\bZ$ from Theorem \ref{thm.propofZ} to get a formula for $\bZ_\D(K)$ in
terms of $L$:
\[
\bZ_\D(K) = \bZ_\D(L)\pp_{j=1}^{2g}\boldsymbol{\mathcal{B}}^{2j-1,2j}_j\pp \bm^{1,2\dots g}_1
\]
Here we followed Equation \eqref{eq.Bandersnatch} and the properties of $\bZ$ to write
\[\boldsymbol{\mathcal{B}}^{ij}_k = \bC_3\bC_4\mathbf{\Delta}^i_{r_1 \ell_1 }\mathbf{\Delta}^j_{r_2\ell_2}\pp \bar{\bS}_{r_1}\pp \bS_{r_2}\pp \bm^{\ell_1 r_2 3 4 r_1 \ell_2}_{k}
\]
Passing to generating functions with respect to the usual ordering we finally obtain a more practical version of the same equation.
\begin{equation}
\label{eq.SeifFor}
Z_\D(K) = Z_\D(L)\pp_{j=1}^{2g}\G(\mathcal{B})^{2j-1,2j}_j\pp \gm^{1,2\dots g}_1
\end{equation}

The Seifert matrix $V$ of $\Sigma$ is a square matrix of size $2g$ and can be obtained from $L$ as follows. 
Suppose $p_{ij}$ denotes the number of positive crossings where
strand $i$ passes over strand $j$ and $n_{i,j}$ the number of negative such. The $(i,j)$-th entry of the Seifert matrix is given by $V_{ij}=p_{ij}-n_{ij}$. 
In the figure eight knot example we have $V = \left(\begin{array}{cc} -1 & 0\\ -1 & 1 \end{array}\right)$.

The special shape of our Seifert surface $\Sigma$ means that the intersection form on $H_1(\Sigma)$ is especially simple. 
Using the cores of the bands as a basis the intersection form has matrix $F = \sum_{i=1}^g E_{2i-1,2i}-E_{2i,2i-1}$. 
Recall that in order for $V$ to be a Seifert matrix it should satisfy $V-V^t=F$. The notation $V^t$ means transpose of $V$.

In the remainder of this section we attempt to compute $\bZ_\D$ explicitly using formula \eqref{eq.SeifFor} and the Contraction Theorem.
We first do this to at $\eps=0$ to make contact with the Alexander polynomial and then extend the same approach partially to higher orders.

Recall that the (Conway normalized) Alexander polynomial can be expressed in terms of the Seifert matrix $V$,
see \cite{Li97} Chapter 8, as:
\begin{equation}
\label{eq.Alex}
\Delta_K(T) = \det(VT^{\frac{1}{2}}-V^tT^{-\frac{1}{2}})\end{equation}
The reader should check that applying this formula to our example gives $-T+3-T^{-1}$, the Alexander polynomial of the figure eight knot.

The key ingredient to the computation of $\bZ_\D$ is the generating function $\G(\mathcal{B})$. As it is a composition of the generating functions 
of the Hopf operations we can explicitly compute it at $\eps=0$ to find
$\G(\mathcal{B}^{ij}_k)|_{\eps=0} = T_ke^G$ where
\[
G=(T_k-1)\big(\Aa_i(\Aa_j-1)\xi_i\eta_i+\Aa_j(\Aa_i-1)\xi_j\eta_j+(\Aa_i+\Aa_j-\Aa_i\Aa_j)\xi_i\eta_j-\Aa_i\Aa_j\xi_j\eta_i\big)+\]
\begin{equation}
\label{eq.Bander0}
\big(\Aa_i(1-\Aa_j^{-1})\xi_i+(1-\Aa_i)\xi_j\big)x_k+\big((1-\Aa_j)\eta_i+\Aa_j(1-\Aa_i^{-1})\eta_j\big)y_k
\end{equation}
This computation can be done by hand but it may be more convenient and just as rigorous to do it by computer.

We are now ready to prove the first result on the Alexander polynomial at $\eps = 0$.
\begin{lemma}
\label{lem.Alex}
For a $0$-framed knot $K$ we have
\[Z_\D(K) = \frac{1}{\Delta_K(T)} \mod \eps\]
\end{lemma}
\begin{proof}
Using the set up as explained in this subsection we present $K$ as the boundary of a Seifert surface in band form with bands
encoded by the link $L$. In the ensuing Equation \ref{eq.SeifFor} for $Z_\D(K)$ we notice that $\G (\mathcal{B}^{ij}_k)|_{\eps=0}$ does \emph{not} depend on $\tau_i,\tau_j$. 
This can be used to simplify the expression for $Z_\D(L)$ a lot. 
Initially $L$ is assumed to be some product of $R$-matrices and
spinners that are merged to form $2g$ strands. If we are to compose with copies of $\G(\mathcal{B}^{ij})$
then the lack of factors $\tau$ means that after contraction all the $t$ are set to $0$ and hence all $T$ to $1$.
Since $t$ is central we might as well set $t=0,T=1$ already from the very start.
Since $C_i^{\pm 1}=T_i^{\pm \frac{1}{2}} \mod \eps$ we can ignore those in the expression for $Z_\D(L)$. Also it 
makes the $R$-matrices (at $\epsilon=0$) independent of $a$ so that the whole of $Z_\D(L)$ becomes independent of $a$.
In contracting $a,\alpha$ it is thus equivalent to setting $\alpha=0$ in both the multiplication tensors 
and in the formula for $\mathcal{B}$ itself. For the multiplication this turns it into commutative multiplication:
because the non-trivial terms involving $\Aa$ and $(1-T)\xi \eta$ all vanish.

Our conclusion is that at $\epsilon=0$, we get
\[
Z_\D(K)|_{\eps=0} = \tilde{L}\pp_{j=1}^{g}\tilde{\mathcal{B}}^{2j-1, 2j}_{j} \pp_{i=1}^g \mathrm{id}^i_1
\]
where $\mathrm{id}^i_1 = e^{\tau_it_1+\alpha_ia_1+\eta_iy_1+\xi_ix_1}$ just places everything in the first tensor factor.
Also $\tilde{L} = \prod_{ij}e^{y_ix_j (p_{ij}-n_{ij})}$ and $\tilde{\mathcal{B}}^{ij}_k = T_ke^{(T_k-1)(\xi_i\eta_j-\xi_j\eta_i)}=
\G(\mathcal{B}^{ij}_k)|_{\eps=0,\mathcal{A}=1}$.

The only contractions left to do are those in $x,y$ and we can use the contraction theorem (Theorem \ref{thm.zip}) to carry these out.
In the final step only one strand remains so we might as well set the $T_j = T$ as $t$ is central anyway.
We have 
\[Z_\D(K)|_{\eps=0} = T^g\la e^{(y,\xi)W{\eta \choose x}} \ra = T^g\det(1-W)^{-1}\]
where contraction is on all the pairs $x_i,\xi_i$ and $\eta_i,y_i$ and
the matrix $W$ for the quadratic form is a $4g$ by $4g$ matrix of the form 
\begin{equation}
\label{eq.W}
W=\left(\begin{array}{cc} 0 & V \\ (T-1)F & 0 \end{array}\right)
\end{equation}
As before $V$ is the $2g$ by $2g$ Seifert matrix of the surface coming from the crossings of $L$ and the other block
comes from the $\tilde{\mathcal{B}}$ and it is equal to $(T_k-1)F$. Here $F = \sum_{i=1}^g E_{2i-1,2i}-E_{2i,2i-1}$ represents the intersection form
and $V$ should satisfy $V-V^t=F$.

Carrying out the formula of the Contraction Theorem and the determinant of a block matrix we find
\[
Z_\D(K)|_{\eps=0} = T^g\det(1-W)^{-1} = T^g\det(I-V(T-1)F)^{-1} = T^g\det(F+V(T-1))^{-1}\]\[ = T^g\det(V-V^t+VT-V)^{-1} = T^g\det(VT-V^t)^{-1} = \det(V^tT^{\frac{1}{2}}-T^{-\frac{1}{2}}V)^{-1} = \Delta_K(T)^{-1}\]
Here we used $F^2=-I,\ \det (F) = 1$ and $V-V^t=F$ and the fact that the final determinant is homogeneous of degree $2g$.
\end{proof}

We now turn to the general structure of the $Z_\D(K)$ knot invariant of a $0$-framed knot completing the
\begin{proof}(Of Theorem \ref{thm.structZ}):\\
We know $Z_\D(K)$ is central and has no Gaussian term so we can write it as
\begin{equation}
\label{eq.ZM}
Z_\D(K) = \sum_{k=0}^\infty  \eps^k \sum_{j=0}  M_{k,j}(K)\mathbf{w}^j\end{equation}
for some coefficients $M_{k,j}\in \Q[T]$. Since $\mathbf{w} = \mathbf{yx}+(1-T)(\mathbf{a}+\frac{1}{2})\mod \eps$ a non-zero term $M_{k,j}$ with $k>2k$ would imply a non-zero monomial of weight $\wt >0$ contradicting the fact that $Z_\D \in \mathcal{PG}$. Indeed the power of $\mathbf{yx}$ would be too high.

Next we analyse the denominator of $M_{k,j}$. Fix $k$ and work modulo $\eps^{k+1}$. We will consider the coefficient of $\mathbf{y}^j\mathbf{x}^j\eps^k$ in $Z_\D(K)$. Only the $M_{k,r}$ with $r\geq j$ contribute to this coefficient and by induction on $2k-j$ we will show that the denominator of these $M_{k,r}$ is $\Delta_K^{2k-r+1}$. To this end we remark that the elements $\G\mathcal{B}^{ij}_k$ are in $\PG^+$ so the only place where denominators can arise is in contracting $x,y$ with these elements. The multiplications that happen after that are between finite expressions (the Gaussian part is $0$) so there cannot appear a new factor in the denominator there.

Contracting the $\G\mathcal{B}^{ij}_k$ with the invariant of the tangle $L$ in the variables $t$ and $a$ goes as described above. No denominators appear. In the $x,y$ contractions
we get a denominator which is precisely the determinant of the matrix $W$ encoding the Gaussian that we computed in Equation \eqref{eq.W}. It appears once as the determinant and again for each factor
$x_i$ or $\eta_i$ that appears in the perturbation. To have a factor $y^jx^j$ remain we can at most have $2k-j$ appearances of $\eta,x$ because they need to contract with $2k-j$ variables
leaving a weight $2j$ as the maximal weight $\wt$ is $4 k$ to balance $\wt(\eps^k) = -4k$. It follows that $M_{k,j}$ the number of these appearances is at most $2k-j$ because the restrictions on the weight.

So far we proved the following formula for $Z$:
\[\bZ_\D(K) = \frac{1}{\Delta_K}\sum_{k=0}^{\infty} \eps^k \sum_{j=0}^{2k} \tilde{\rho}_{k,j}(K)\frac{\mathbf{w}^j}{\Delta_K^{2k-j}}\]
Since $\mathbf{w}$ is central and $\tilde{\rho}_{0,0} = 1$ we can take the logarithm in the sense of commmutative power series to finish the proof of the theorem.
\end{proof}

\subsection{A closer look at the first order of $Z_\D$}
\label{sub.firstorder}

In this subsection we look more closely at $\bZ_\D$ in the first order in $\eps$. The Seifert arguments from the previous section can 
be specialized to prove both Theorems \ref{thm.propZ} and \ref{thm.genus}. As in the previous subsection we choose not to write the factors $\hbar$ explicitly.

Looking more closely at the first order in $\epsilon$ what is important is to know two specializations of the generating function $\G\mathcal{B}^{ij}_k$:
First $\G\mathcal{B}^{ij}_k|_{\eps = 0}$ given in \eqref{eq.Bander0} for contracting with surviving terms $\eps a$. Second
the restriction to $\G\mathcal{B}^{ij}_k|_{\alpha = 0}$ for contracting with the terms independent of $a$. 
By direct computation we find
$\G\mathcal{B}^{ij}_k|_{\alpha=0} = T_kPe^{(T_k-1)(\xi_i\eta_j-\xi_j\eta_i)}$ where
\[
P = y_kx_k(\xi_i\eta_j-\xi_j\eta_i)-2a_k(1+T_k(\xi_i\eta_j-\xi_j\eta_i))+x_k(\xi_i^2\eta_j+\xi_j(\xi_j\eta_i-2)+2\xi_i(1-2\xi_j\eta_j))+\]
\[
y_k(\xi_j\eta_i^2+\eta_j(2+\xi_i\eta_j)-2\eta_i(1+\xi_i\eta_j))+(T_k-1)(-2\xi_i\eta_j+2\xi_i\eta_i+2\xi_j\eta_j+\xi_i^2\eta_i\eta_j+\xi_i\xi_j\eta_j^2)+
\]
\begin{equation}
\label{eq.Bandera0}
(3-4T_k+T_k^2)(\xi_i\xi_j\eta_i\eta_j+\frac{\xi_i^2\eta_j^2-\xi_j^2\eta_i^2}{4})
\end{equation}

\begin{proof} (of Theorem \ref{thm.propZ}):\\
Our starting point is again Equation \eqref{eq.SeifFor} but now we look at the first order in $\eps$.
First let us look more closely at $M_{1,2}$, referring to the coefficients in Equation \eqref{eq.ZM}. It must be $0$ because there is no way for $\eps y_k^2x_k^2$ to appear after contracting with $\G\mathcal{B}$.
Factors of $x_k,y_k$ cannot come the multiplication afterwards either because in the $\eps=0$ part $x_k,y_k$ do not appear at all.

To find $M_{1,1}$ it is most convenient to look at the coefficient of $\eps a_k$ in $Z_\D(K)$. Since we showed $M_{1,2}=0$ and the $M_{1,0}$ term will not contribute this coefficient is
$M_{1,1}(1-T)^{-1}$ because $w = yx+(1-T)(a+\frac{1}{2})\mod \eps$. After contracting the $t$ and the $a$-variables in turn as done in the previous Seifert arguments, the coefficient of $\eps a_k$ must be the result of contracting 
\[
\la-2T(gT^{g-1}+T^g\xi F\eta)e^{yVx+(T-1)\xi F \eta}\ra = -2T\p_T \la T^g e^{yVx+(T-1)\xi F \eta}\ra =\]\[ -2T\p_T \frac{1}{\Delta_K(T)} = \frac{2T\Delta_K'(T)}{\Delta_K(T)^2}
\]
where we set $\xi F \eta = \sum_{i=1}\xi_{2i-1}\eta_{2i}-\xi_{2i}\eta_{2i-1}$. So this implies $M_{1,1} = \frac{2T\Delta_K'(T)}{(1-T)\Delta_K(T)^2}$.
\end{proof}

As a final instance of the Seifert arguments we now study the most interesting case $M_{1,0}$ in Equation \eqref{eq.ZM} and relate it to the knot genus.

\begin{proof} (of Theorem \ref{thm.genus}):\\
Focusing on the coefficient of $\eps$ we first contract $t$ and $a$ in Equation \eqref{eq.SeifFor} getting no denominator. We are left with contracting perturbation that is at most degree $4$ in $x,y,\xi,\eta$ times the above found Gaussian $T^g e^{(y,\xi)W{\eta \choose x}}$ with $W$ given in Equation \eqref{eq.W}.
For any perturbation $P$ we thus get 
\[\la P((y,\xi),{\eta \choose x})T^g e^{(y,\xi)W{\eta \choose x}}\ra_{x,y} = T^g(\det\tilde{W}) \la P((y,\xi),\tilde{W}{\eta \choose x}) \ra\]
with $\tilde{W} = (1-W)^{-1}$ which can be computed using the formula (valid for any square matrices $B,C$):
\[\left(\begin{array}{cc} 1 & B \\ C & 1 \end{array}\right)^{-1} = \left(\begin{array}{cc} (1-BC)^{-1} & -(1-BC)^{-1}B \\ (1-CB)^{-1} &-(1-CB)^{-1}C \end{array}\right)\]
In our case we take $B=-V$ and $C = (1-T)F$ and notice that $\det(1-BC) = \det(1-CB) =T^{g}\Delta_K(T)$ using the proof of the proof of Lemma \ref{lem.Alex}.
In terms of the adjugate matrix we can thus write $\tilde{W} = \frac{1}{\Delta_K(T)}Y$ where the matrix $Y$ is given by
\[
Y = T^{-g}\left(\begin{array}{cc} \mathrm{adj}(1-BC) & -\mathrm{adj}(1-BC)B \\ \mathrm{adj}(1-CB)  & -\mathrm{adj}(1-CB)C \end{array}\right)
\]
We can estimate the $T$-degree of the entries of $Y$ because $B$ is constant in $T$ and $C$ is linear and the adjugate matrix is homogeneous of degree $2g-1$. It follows that
the top $2g$ rows of $Y$ have degree at most $g-1$ in $T$ and the bottom rows have degree at most $g$. 

In applying the contraction theorem the relevant terms in the perturbation $P$ can come from two sources. First we can have a quartic involving
only Greek letters if it comes from the $\eps$ part of $\G\mathcal{B}$ or $a$-part of the tangle $L$. In that case the $T$-coefficient of $P$ has degree at most $2$ and so
we find that replacing $\eta$ by a multiple of $x$ times a polynomial of degree $\leq g-1$ twice yields a maximal degree of $2g$ in $T$.

Second if the perturbation monomial comes from the $x,y$ part of $L$ then it is replaced by a polynomial of degree $\leq g$ but has no coefficient of its own
so doing this twice also yields $T^{2g}$ at most. 

We conclude that the degree of $\tilde{\rho}_{1,0} = \frac{M_{1,0}}{\Delta_K(T)^3}$ is at most $2g$ as claimed. 
\end{proof}

\subsection{Computer practicum 2}

In this final subsection we illustrate some of the results using the Mathematica implementation listed in Appendix \ref{sec.Implementation}. Computer input is written in bold and the output is directly below. Recall that in Mathematica a command ending in a semicolon \texttt{;} is not printed.
The tests can in principle be ran at any order in $\eps$ by setting \texttt{\$k} at the top of the program. In the printed output we chose \texttt{\$k=1}.
Most of the tests proceed by deciding the equality between two morphisms $Pe^G$ in $\PG$. In Mathematica we use the symbol $\equiv$ for this and it just checks that the $G$ are equal and the $P$ are equal too after elementary simplifications. The output is usually \texttt{True} and to improve readibility we bunched up the tests into lists of similar items resulting in a list of several copies of \texttt{True}.

We start by testing the Hopf algebra properties of the algebra $\D$.\\
\includegraphics[width=10cm]{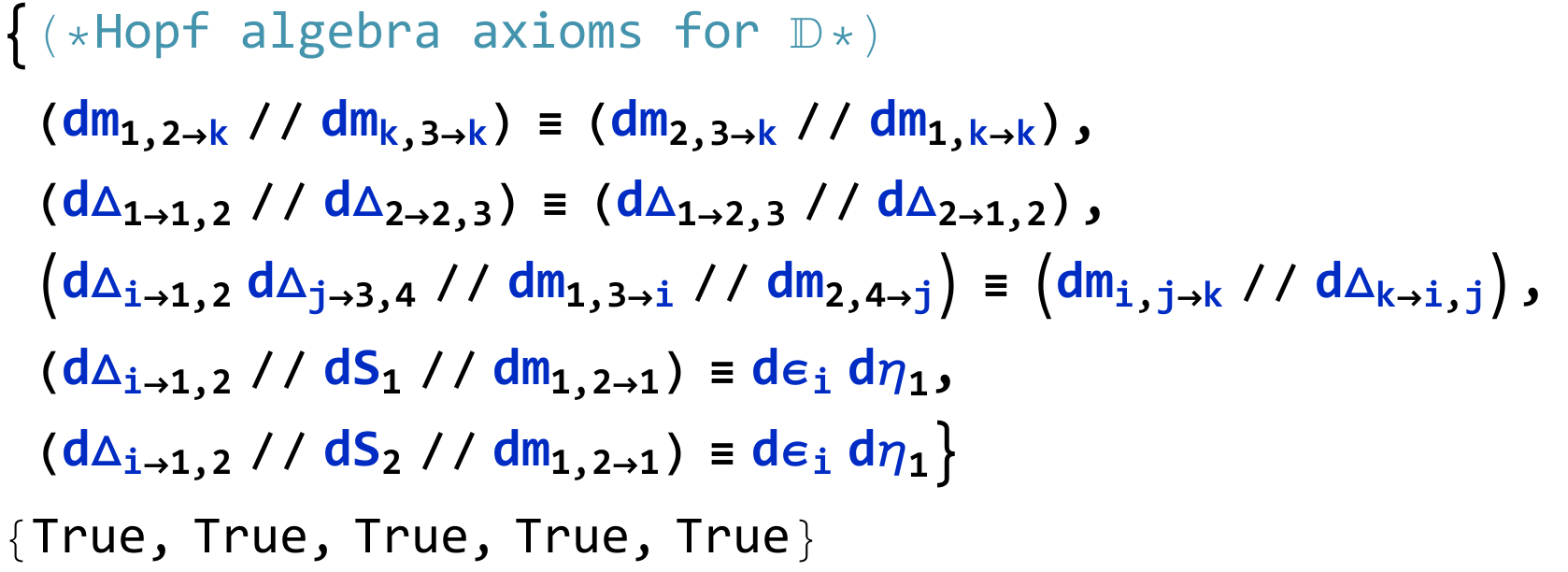}\\
Next we check the ribbon Hopf algebra structure is as claimed:\\
\includegraphics[width=12cm]{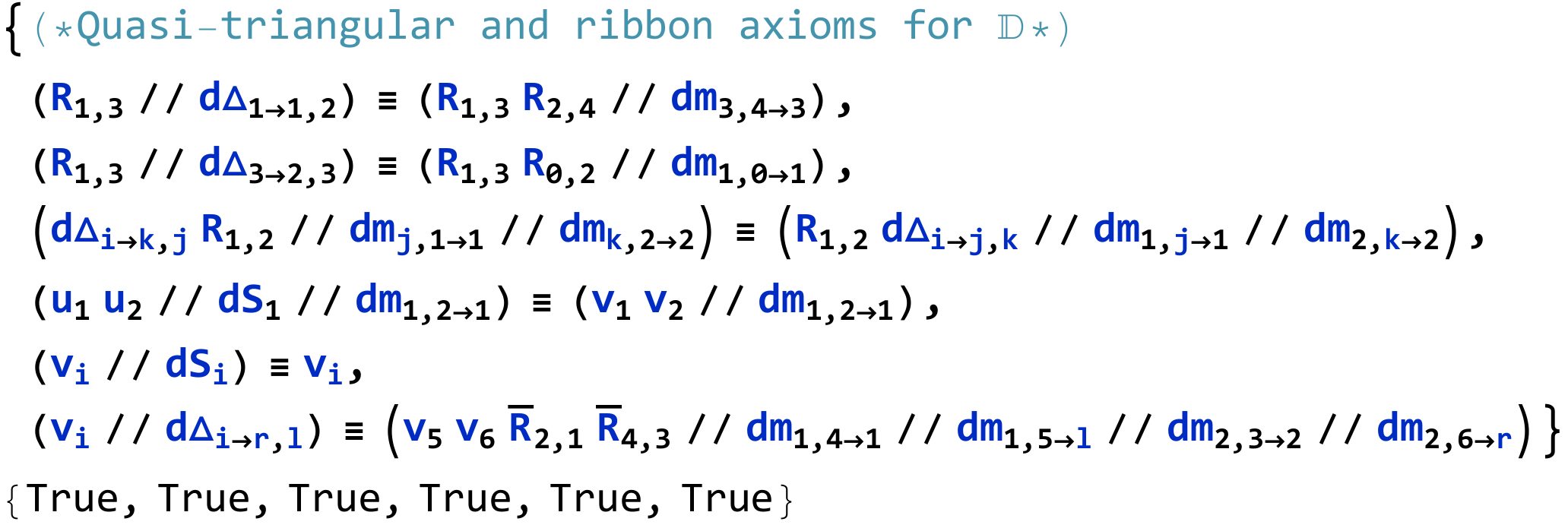}\\
We also check that the central element $\mathbf{w}$ and $\bC$ work as advertised.\\ 
\includegraphics[width=6cm]{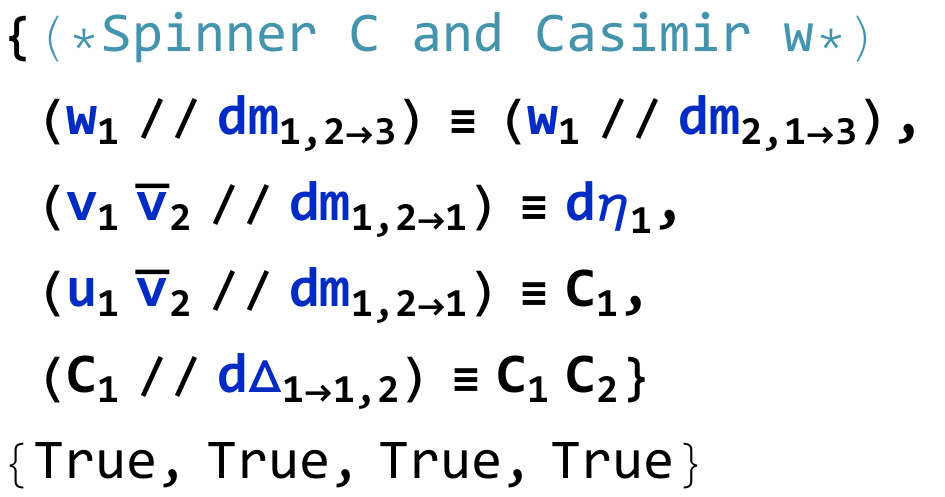}\\
Next invariance of $\bZ_\D$ under the Reidemeister moves is tested:\\ 
\includegraphics[width=14cm]{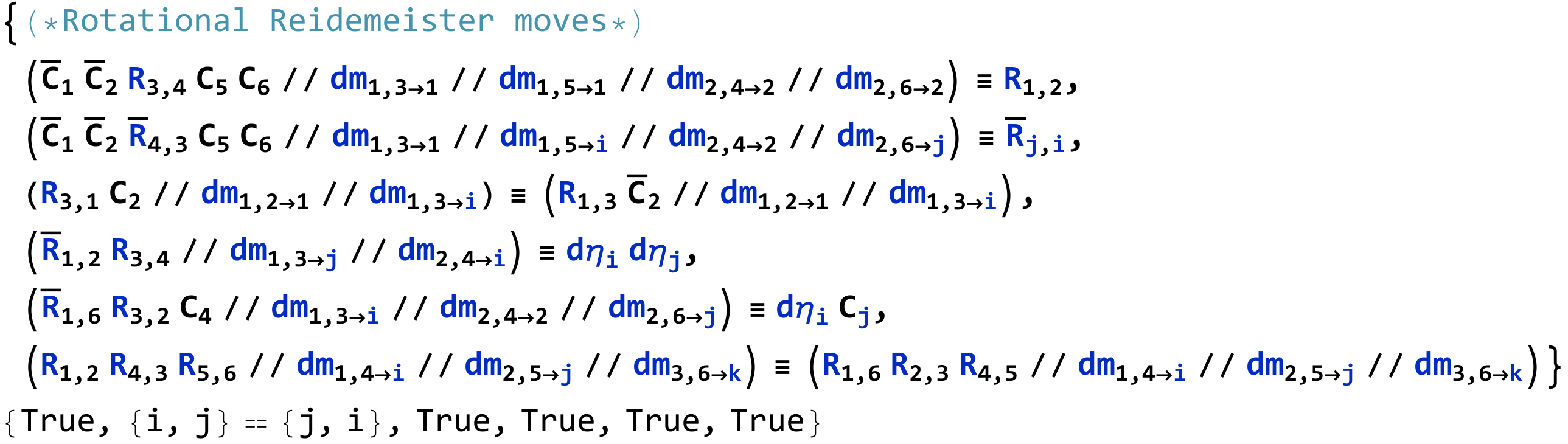}\\
One of the outputs is in fact \texttt{\{i,j\}==\{j,i\}} for some silly technical reason and it should be interpreted as \texttt{True}.

After checking the algebra is correct we investigate the invariant of the (zero-framed) right-handed trefoil from several points of view.
First we compute $\bZ$ directly by merging three positive crossings and three negative kinks and print the output.
Next we check that the same result can be obtained by thickening the two components of tangle $L$ using the $\mathcal{B}^{ij}_k$
band thickener. Then we find the invariant of the Whitehead double of the trefoil by applying $W_i$, illustrating Theorem \ref{thm.WDouble}. Finally we express $\bZ_\D$ in terms
of $\mathbf{w}$ as in Theorem \ref{thm.structZ}, bringing out $\rho_{1,0}$ as the coefficient of $\eps$ in the output.
\\
\includegraphics[width=\linewidth]{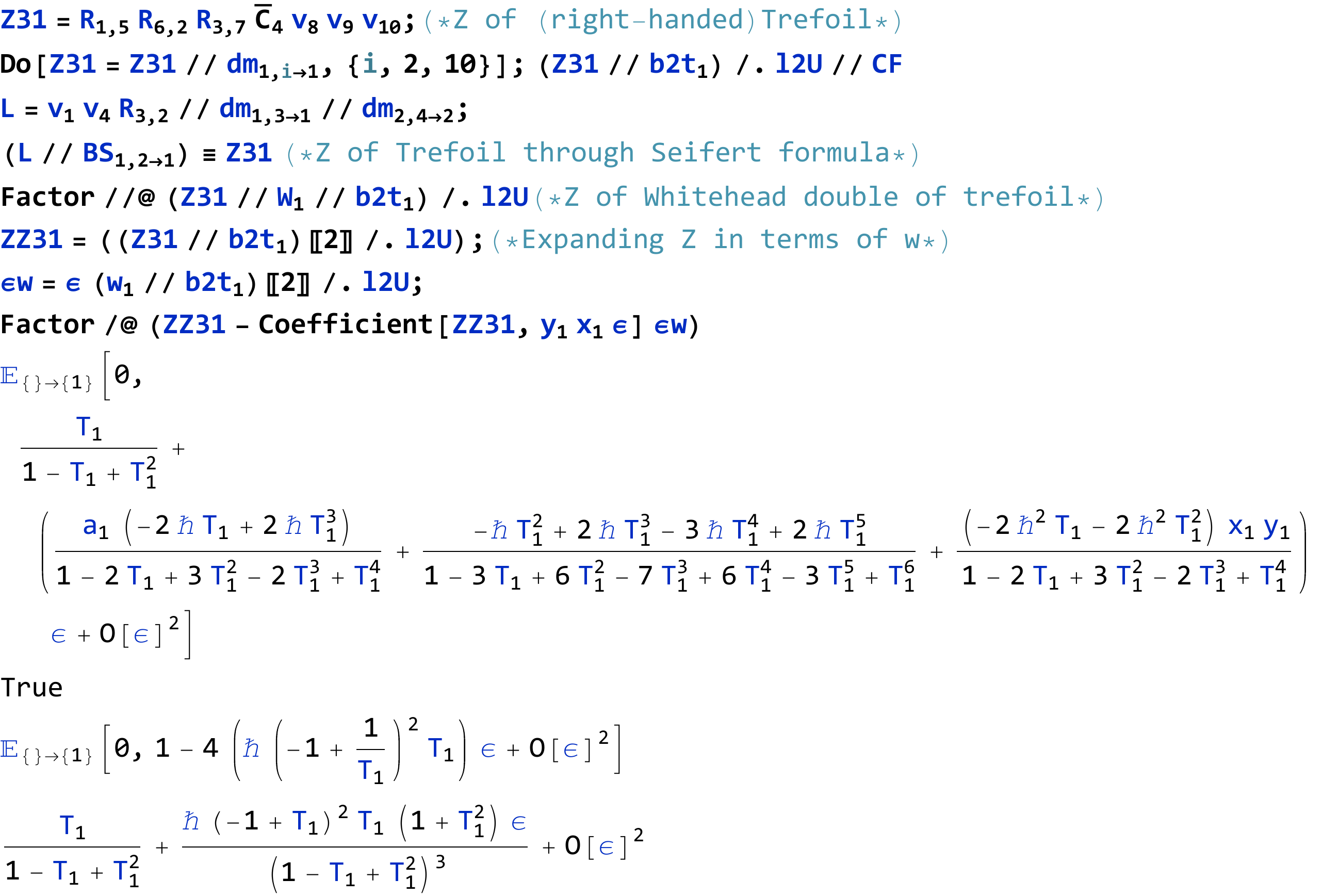}\\
With some more effort we can also compute the invariant of the trefoil knot to second order. We do not show the raw output but instead illustrate how to find the coefficients
$\rho_{k,j}$ from Theorem \ref{thm.structZ} step by step by subtracting powers of $\mathbf{w}$ and $\eps$. The reader is warned that this computation takes some time as the program shown here is optimized for simplicity, not speed. A link to more efficient implementation is found in Appendix \ref{sec.Implementation}\\
\includegraphics[width=\linewidth]{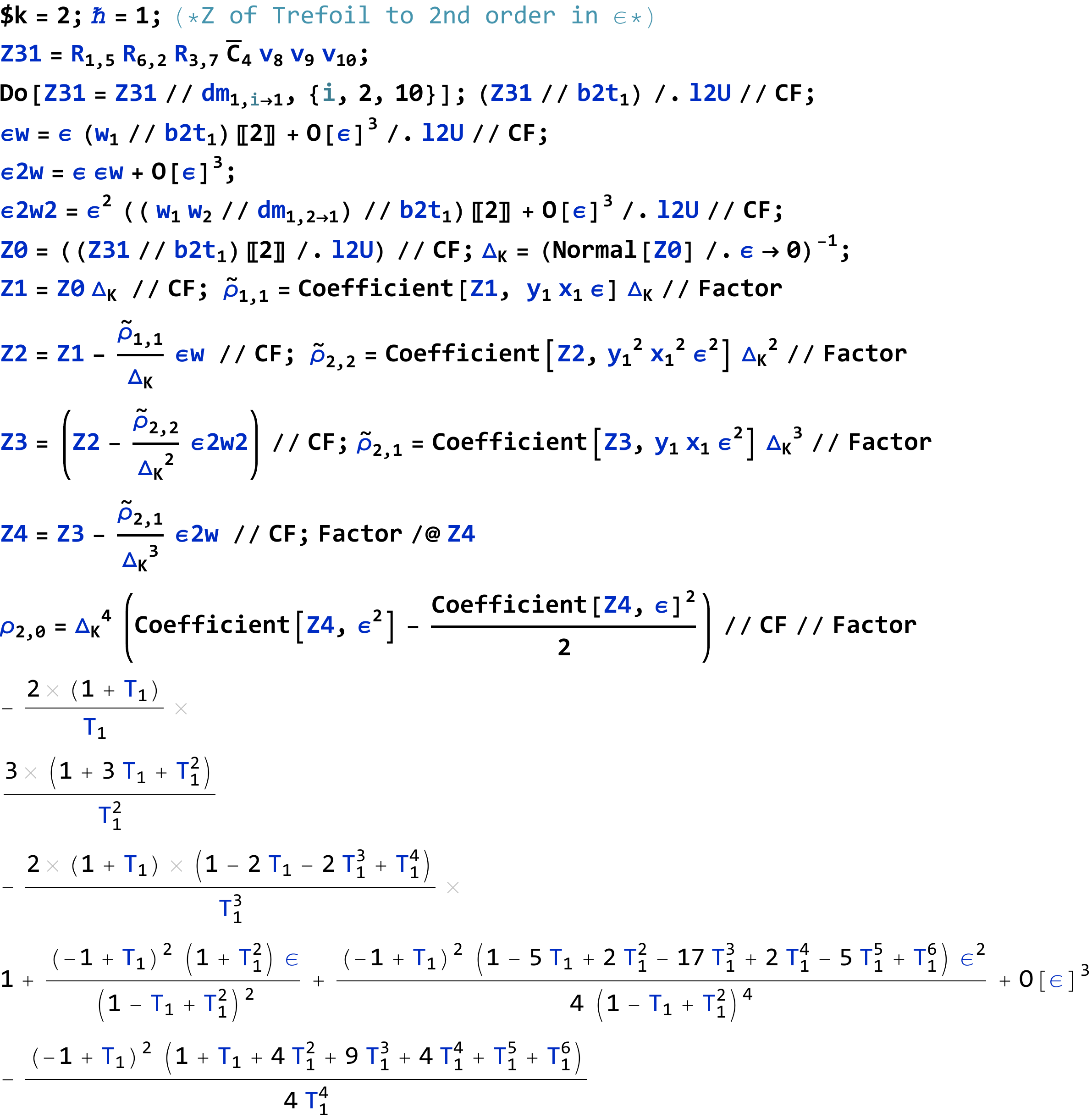}\\
In the final line of output we observe the $\rho_{2,0}$ of the trefoil. The factor $\frac{(-1+T)^2}{-4T}$ seems to be universal and so it is omitted in the knot table in Appendix \ref{sec.KnotTable}. There we listed the value of $\bZ_\D$ for all prime knots up to ten crossings to second order in $\eps$. The $\rho_{2,0}$ polynomials are printed in blue, and for the trefoil we recognize the $T^3+T^2+4T+9$.

\appendix
\section{Proof of $u = AB\ S(u)$}
\label{sec.uABSu}

In this appendix we give a self-contained proof of the fundamental equality 
$u = AB\ S(u)$ following Othsuki's proof for the $U_{\hbar}(\mathfrak{sl}_2)$ case, see \cite{Oh01} appendix A.
We will not be using generating function techniques here but instead work directly in the algebra $\D$. In the main text we would write
elements and operations of $\D$ in boldface but here we will use ordinary script for simplicity.

Recall that $u_1 = R_{12} \pp S_2 \pp m^{21}_1$ while $S(u)_1 = R_{12} \pp S_2 \pp m^{12}_1$ and the $R$-matrix is
$R_{12} = \sum_{m,n=0}^\infty \frac{\hbar^{m+n}y^m_1b^n_1a^n_2x_2^m}{[m]!n!}$.

It will be more convenient here to use the symmetric quantum integers $(n)_q = \frac{q^{\frac{n}{2}}-q^{-\frac{n}{2}}}{q^{\frac{1}{2}}-q^{-\frac{1}{2}}}$
and the corresponding $q$-factorial and binomial coefficients. This is not a big difference since $[k] = q^{\frac{k-1}{2}}(k)_q$.

In what follows we will introduce $K = AT^{\frac{1}{2}}$ and change to more convenient generator $X = K^{-1}x$. This has
the advantage of turning the $q$-commutator relation of $x,y$ into $[y,X] = (H)_q$ where we define $(H)_q = \frac{K-K^{-1}}{\hbar}$.
Also set $(H+k)_q = \frac{q^{\frac{k}{2}}K-q^{-\frac{k}{2}}K^{-1}}{\hbar}$ and the $q$-binomial coefficient ${H\choose k} = (H)_q(H-1)_q\dots (H-k+1)_q/(k)_q!$
 The following identity holds:
\[
(a)_q(H+c+b)_q+(b)_q(H+c-a)_q = (a+b)_q(H+c)_q
\]
It also follows by induction that (see also Lemma A.7 of \cite{Oh01}):
\begin{equation}
\label{eq.yX}
y^rX^r = \sum_{n=0}^r {r \choose n}_q^2(r-n)_q!^2{H\choose r-n}_qX^ny^n
\end{equation}
Next we establish an explicit formula for both $u$ and $S(u)$. We will use $S(x) = -q^{-1}A^{-1}x = -q^{-1}T^{\frac{1}{2}}X$ and $B =TA$ and $AB = K^2$:
\[
u= \sum_{m,n}\frac{\hbar^{m+n}}{[m]!n!}S(x)^mS(a)^ny^mb^n = \sum_{m,n}\frac{\hbar^{m+n}}{[m]!n!}(-q^{-1}T^{\frac{1}{2}}X)^my^m(-a+m)^nb^n =  
\]
\[
e^{-a b \hbar}\sum_{m}\frac{\hbar^m}{[m]!}(-q)^{-m} X^m y^m K^{-m} 
\]
Next $S(u) = $
\[
 \sum_{m,n}\frac{\hbar^{m+n}}{[m]!n!}y^mb^n S(x)^mS(a)^n =  \sum_{m,n}\frac{\hbar^{m+n}}{[m]!n!}y^m(-q^{-1}T^{\frac{1}{2}}X)^m(b+\epsilon m)^n(-a)^n = 
\]
\[
e^{-a b \hbar} \sum_{m}\frac{\hbar^m}{[m]!}(-q)^{-m}y^m X^m K^m 
\]

To prove that $S(u) = u K^{-2}$ we introduce 
$S(u)^{<r} = e^{-a b \hbar} \sum_{m=0}^{r-1}\frac{\hbar^m}{[m]!}(-q)^{-m}y^m X^m K^m $
$u^{<r}K^{-2} = e^{-a b \hbar}\sum_{m}\frac{\hbar^m}{[m]!}(-q)^{-m} X^m y^m K^{-m-2}$ and show they satisfy
$S(u)^{<r}-u^{<r}K^{-2}$ is divisible by $\hbar^r$.
We do so by proving the following more precise result by induction on $r$:

$S(u)^{<r}-u^{<r}K^{-2} = $
\[
\hbar^r e^{-a b \hbar} (-1)^{r-1}q^{-\frac{r(r-1)}{4}}\sum_{n=0}^{r-1}\frac{q^{-\frac{n}{2}}}{(n)_q!}X^ny^n{H \choose r-n}_q(r-n)_q!\sum_{i=0}^n q^{\frac{i(r-n-1)}{2}}{r-i-1 \choose r-n-1}_qK^{r-2i-2}
\]
\begin{proof}
Induction on $r$ will prove the slightly simplified version: $L(r) = R(r)$ where
\[L(r) = (S(u)^{<r}-u^{<r}K^{-2})\hbar^{-r} e^{a b \hbar} (-1)^{r-1}q^{\frac{r(r-1)}{4}}K^{-r+2}\]
and
\[
R(r) = \sum_{n=0}^{r-1}\frac{q^{-\frac{n}{2}}}{(n)_q!}X^ny^n{H \choose r-n}_q(r-n)_q!\sum_{i=0}^n q^{\frac{i(r-n-1)}{2}}{r-i-1 \choose r-n-1}_qK^{-2i}
\]
In the induction basis we compute $L(1)=\frac{K-K^{-1}}{\hbar} = (H)_q$ coinciding with the $R(1)$.
For the induction step we assume $L(r) = R(r)$ and examine \[L(r+1) = (S(u)^{<r+1}-u^{<r+1}K^{-2})\hbar^{-r-1} e^{a b \hbar} (-1)^{r}q^{\frac{(r+1)r}{4}}K^{-r+1} =\]
\[
\hbar^{-r-1} (-1)^{r}q^{\frac{(r+1)r}{4}}K^{-r+1}\sum_{m=0}^{r}\frac{\hbar^m}{(m)_q!q^{\frac{m(m-1)}{4}}}(-q)^{-m}(y^m X^m K^m - X^m y^m K^{-m}) =  
\]
\[
-\hbar^{-1}q^{\frac{r}{2}}K^{-1}L(r)+
\hbar^{-1} \frac{q^{-\frac{r}{2}}}{(r)_q!}(y^r X^r K - X^r y^r K^{-2r+1}) 
\]
We aim to show that the coefficient of $X^ny^n$ in this expression equals that of $R(r+1)$ for any $n\leq r$. To do this we need to use the commutation relation \eqref{eq.yX}.
When $n = r$ the term $L(r)$ does not contribute and the coefficient of $X^ry^r$ in $L(r+1)$ is: $\text{Coeff}_{X^ry^r}(L(r+1)) = $
\[
\hbar^{-1} \frac{q^{-\frac{r}{2}}}{(r)_q!}(K - K^{-2r+1}) = \frac{q^{-\frac{r}{2}}}{(r)_q!}(H)_q\sum_{i=0}^{2r}K^{-2i}
\]
which equals the coefficient of $X^ry^r$ in $R(r+1)$, that is $\text{Coeff}_{X^ry^r}(R(r+1))$.

Next let us apply the induction hypothesis $L(r) = R(r)$ and investigate the coefficient of $X^ny^n$ in $L(r+1)$, where $n<r$. It is $\text{Coeff}_{X^ny^n}(L(r+1)) = $
\[
-\hbar^{-1}q^{\frac{r}{2}}K^{-1}\text{Coeff}_{X^ny^n}(R(r))+
\hbar^{-1} \frac{q^{-\frac{r}{2}}}{(r)_q!}\text{Coeff}_{X^ny^n}(y^r X^r) K =
\]
\[
-\hbar^{-1}q^{\frac{r}{2}}K^{-1}\frac{q^{-\frac{n}{2}}}{(n)_q!}{H \choose r-n}_q(r-n)_q!\sum_{i=0}^n q^{\frac{i(r-n-1)}{2}}{r-i-1 \choose r-n-1}_qK^{-2i}\]
\[+
\hbar^{-1} \frac{q^{-\frac{r}{2}}}{(r)_q!}{r \choose n}_q^2(r-n)_q!^2{H\choose r-n}_q K =
\]
\[
\hbar^{-1}{H\choose r-n}_q\frac{(r-n)_q!}{(n)_q!}Kq^{-\frac{r}{2}}\Big({r \choose n}_q
-q^{\frac{2r-n}{2}}\sum_{i=0}^n q^{\frac{i(r-n-1)}{2}}{r-i-1 \choose r-n-1}_qK^{-2i-2}\Big)\]
Using the identity (see Lemma A.8 of \cite{Oh01}) 
\[
{r \choose n}_q
-q^{\frac{2r-n}{2}}\sum_{i=0}^n q^{\frac{i(r-n-1)}{2}}{r-i-1 \choose r-n-1}_qK^{-2i-2} = (1-q^{r-n}K^{-2})\sum_{i=0}^n q^{\frac{i(r-n)}{2}}{r-i \choose r-n}_qK^{-2i}
\]
we find  $\text{Coeff}_{X^ny^n}(L(r+1)) = $
\[
{H\choose r-n}_q\frac{(r-n)_q!}{(n)_q!}Kq^{-\frac{r}{2}}(H-(r-n))_q\sum_{i=0}^n q^{\frac{i(r-n)}{2}}{r-i \choose r-n}_qK^{-2i}
\]
which equals $\text{Coeff}_{X^ny^n}(R(r+1))$ as promised.
\end{proof}

\section{Computer implementation}
\label{sec.Implementation}
In this appendix we list the full Mathematica implementation for the computation of the universal invariants $\tilde{Z}_{\mathcal{U}(\mathfrak{h})}$ and $Z_\D$
corresponding to the Heisenberg algebra and our main Drinfeld double example. The parameter $\$k$ at the top of the code decides how many orders of $\eps$ we take into account.
The heart of the program is the function Contract, which implements the Contraction Theorem \ref{thm.zip}. For completeness we decided to list all the utilities that make the program
run smoothly but we will not discuss them in detail here. The version of the program presented here was not optimized for speed but rather for (relative) simplicity. 
A faster implementation can be found at the authors' website \texttt{www.rolandvdv.nl/PG}

The first page of the program sets up the machinery necessary for dealing with perturbed Gaussians and their contraction calculus. The second page goes on to define
the generating functions for some of the most basic operations such as multiplication in $\B$ and $\A$ and the $R$-matrix. Once those are defined we express the operations
in the double $\D$ by precisely the formulas given in Section \ref{sec.GenD}. For computing in the Heisenberg algebra one only needs the top two lines.
The notation for the objects follow the notation in the main text quite closely. For example we write $\texttt{am}_{i,j\to k}$ for $(\gm_\A)^{i,j}_k$. Likewise 
operations in $\D$ generally start with \texttt{d} and those of the Heisenberg algebra start with \texttt{h}. \\

\includegraphics[width=15cm]{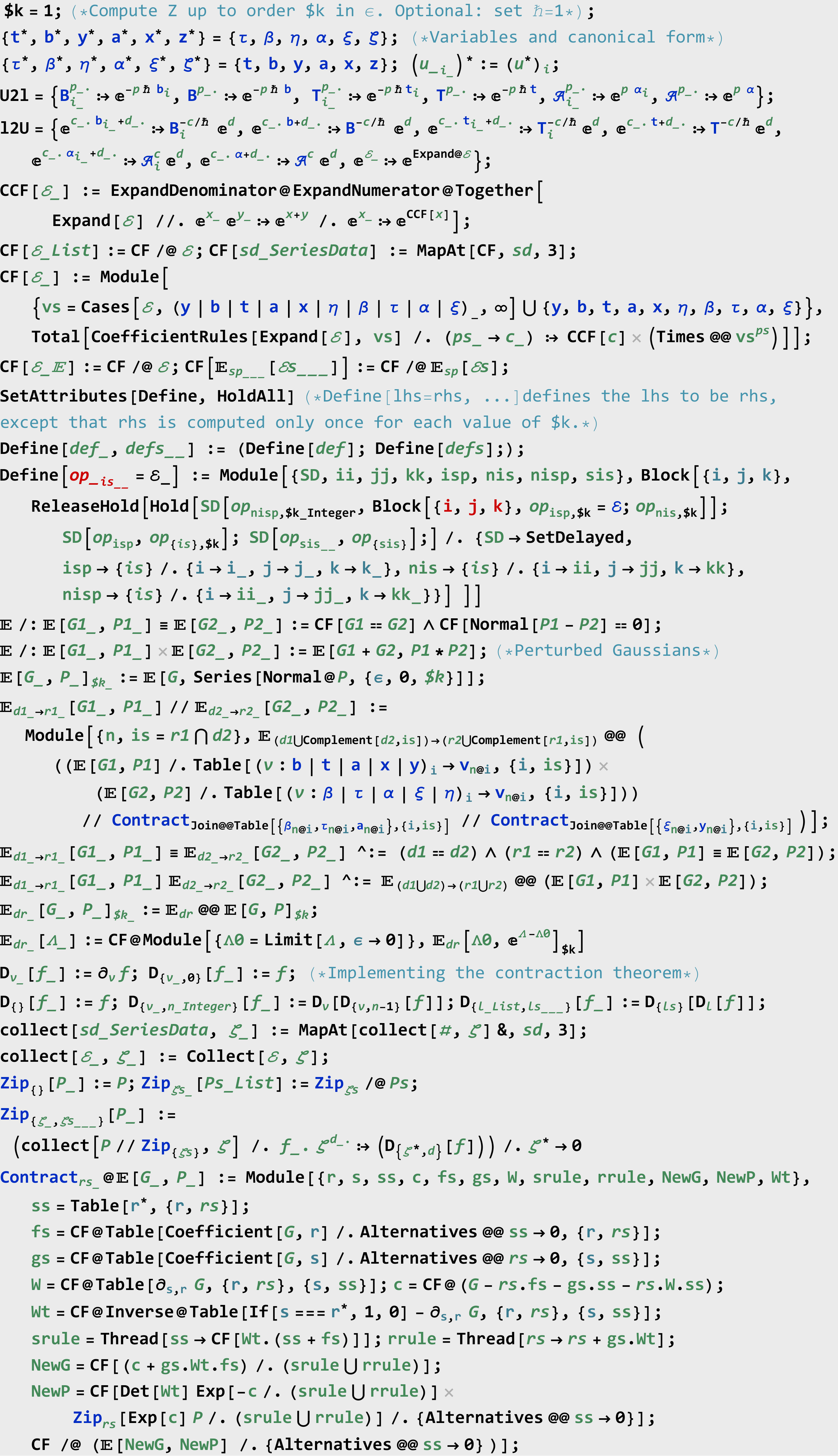}\\

\includegraphics[width=15cm]{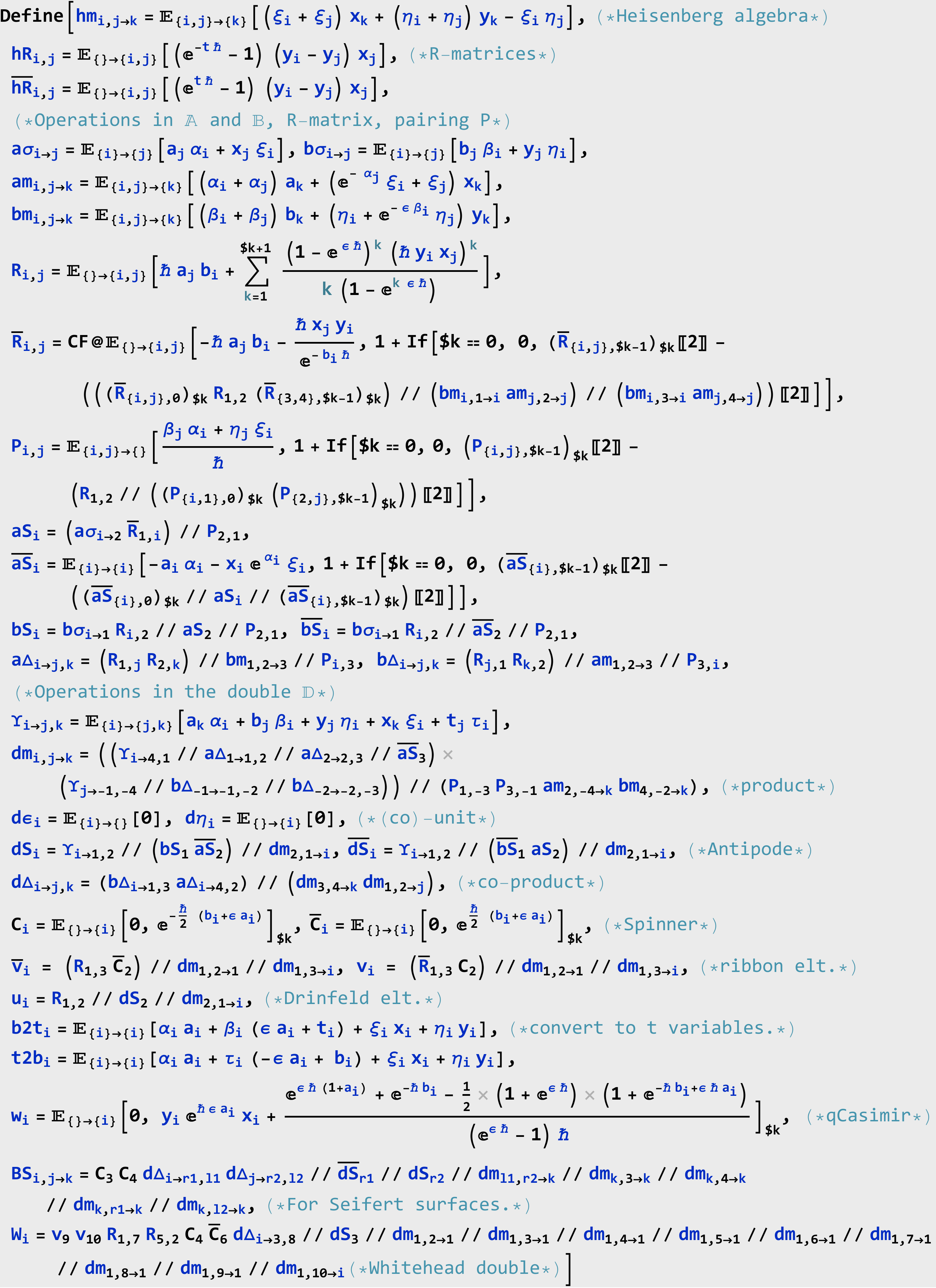}

\section{Glossary of notation}
\label{sec.Notation}

To facilitate non-linear reading of the paper we list all commonly used non-standard notation. 

First we have the convention that {\bf boldface} symbols
are in the original algebra while italic refers to the commutative description in the chosen ordered basis, see Section \ref{sec.a2la}. Sometimes the corresponding generating function
is abbreviated by pre-pending a superscript $\G$, for example we use $\bm,m,\gm$ for the algebra multiplication, its commutative description and the generating function of m.

A few other conventions we use are a bar on top of a symbol to denote its inverse, e.g. $\bar{R} = R^{-1}$. Also $[n],[n]!$ denote quantum integers see Section \ref{sec.TwoStep}.
Finally $A^{\otimes S}$ denotes the tensor power of $A$ indexed by set $S$ in the sense of Definition \ref{def.labten}.

The rest of the notations are listed in alphabetic order. The Greek symbols are listed next to their Roman equivalents, e.g. $\xi$ is found next to $x$ and $\eta$ next to $y$.

\begin{small}

\begin{multicols}{2}
\noindent
$\pp$ composition done right, Sec. \ref{sec.GenFunc}\\
$\mathcal{A} = e^{\alpha}$, Def. \ref{def.PG}\\
$\alpha$, Sec. \ref{sec.GenD}\\
$\ba$ generator of $\D$, Def. \ref{def.D}\\
$A = e^{-\eps a \hbar}$, Lem. \ref{lem.HopfA}\\
$\A$ the Hopf dual of $\B$, Lem. \ref{lem.AlgA}\\
$\bb$ generator of $\D$, Def. \ref{def.D}\\
$\B$ the Borel algebra, Def. \ref{def.B}\\
$\mathcal{B}^{ij}_k$ band thickening, Lem. \ref{lem.SeifertCrit}\\
$\beta$ Sec. \ref{sec.GenD}\\
$\bC$ the spinner, Lem. \ref{lem.spinner}\\
$\tilde{\mathcal{C}},\mathcal{C}$ generating function category, Sec. \ref{sec.GenFunc}, \ref{sec.GenD}\\
$\D$ the double algebra, Def. \ref{def.D}\\
$\Delta$ co-product and strand doubling, Def. \ref{def.DiagHopfops}\\
$\Delta_K$ the Alexander polynomial, Eqn. \eqref{eq.Alex}\\
$\eps$ the expansion variable, Sec. \ref{sec.TwoStep} \\
$\varepsilon$ the co-unit and strand removal, Def. \ref{def.DiagHopfops}\\
$F$ Seifert intersection form, Sec. \ref{sub.Seifert}\\
$\G$ Generating function functor, Def. \ref{def.G}\\
$G$ quadratic for the Gaussian, Def. \ref{def.PG}\\
$\tilde{\mathcal{H}}, \mathcal{H}$ category of maps between tensor powers of an algebra, Sec. \ref{sec.GenFunc}, \ref{sec.GenD}  \\
$\hbar$ deformation parameter, Sec. \ref{sec.TwoStep}\\
$\K = \Q[\eps]\llb \hbar \rrb$ base ring, Sec. \ref{sec.TwoStep}\\
$m^{ij}_k$ multiplication and merging, Section \ref{sec.GenFunc}, Def. \ref{def.dum} \\
$M_{k,j}$ coefficient in expansion of $\bZ_\D$, Eqn. \eqref{eq.ZM}\\
$\Oo$ ordering isomorphism, Sec. \ref{sec.GenFunc}\\
$\OO$ ordering functor, Sec. \ref{sec.GenFunc}\\
$P$ perturbation for the Gaussian, Def. \ref{def.PG}\\
$\mathcal{P}$ Category of maps between polynomial algebras, \ref{sec.GenFunc}\\
$\PG, \PG^\pm,\PG^+$ category of perturbed Gaussians, Def. \ref{def.PG}, \ref{def.PGm}\\
$\pi$ Hopf pairing, Lem. \ref{lem.AlgA}\\
$PBW$-type algebra, \ref{def.hadic}\\
$q=e^{\hbar \eps}$ deformation parameter, \ref{def.B}\\
$\Q_{\hbar}[z]=\Q[z]\llb\hbar \rrb$, Def. \ref{def.hadic}\\
$R$ R-matrix, Eqn. \eqref{eq.RHeis}, Thm. \ref{thm.Drh}, Sec. \ref{sub.qtriang} \\
$\rho_{i,j},\tilde{\rho}_{i,j}$ coefficients of $Z_\D$, \ref{thm.structZ}\\
$S$ antipode and strand reversal,  Def. \ref{def.DiagHopfops}\\
$\mathbf{t} = \bb-\eps \ba$, Thm. \ref{thm.ZD}\\
$\mathbf{T} = e^{-\hbar \mathbf{t}}$, Thm. \ref{thm.ZD}\\
$\bu$ Drinfeld element, Sec. \ref{sub.qtriang}\\
$V$ Seifert matrix, Sec. \ref{sub.Seifert}\\
$\bv$ ribbon element, Sec. \ref{sub.qtriang}, Thm. ZD\\
$v_2$ Vassiliev invariant of order 2, Thm. \ref{thm.WDouble}\\
$W_i$ Whitehead doubling operation, Eqn. \eqref{eq.Wi} \\
$\mathbf{w}$ quantum Casimir, Thm. \ref{thm.ZD}\\
$\bx$ generator of $\D$, Def. \ref{def.D}\\
$X_{i,j}$ crossing, Def. \ref{def.TangleDiagram}\\
$\by$ generator of $\D$, Def. \ref{def.D}\\
$\bZ$ universal invariant, Def. \ref{def.uiZ}\\
$\tilde{\bZ}$ baby universal invariant, Def. \ref{def.uiZt} \\
$\mathcal{Z}$ center of algebra, Thm. \ref{thm.ZD}\\
\end{multicols}
\end{small}

\section{Table of knots}
\label{sec.KnotTable}

Below we list the value of $\bZ_\D$ up to order $2$ in $\eps$ on all prime knots with up to ten crossings. According to Theorem \ref{thm.structZ} we only need to list the (Conway normalized) Alexander polynomial and the coefficients  $\rho_{1,0}$ and $\rho_{2,0}$. What we actually list are $\rho_{1,0}\frac{T}{(T-1)^2}$ and $\rho_{2,0}\frac{T}{-4(T-1)^2}$. These happen to be palindromic polynomials with integer coefficients just like Alexander (i.e. invariant under $T\mapsto T^{-1}$) so 
it suffices to only list the monomials with non-negative exponent. These are called  $\Delta^+, \rho^+_{1}$ and $\rho^+_{2}$. We also list the amphichirality, ribbonness, genus and unknotting number of each knot.
The code that produced this table is an optimized version of the program presented in this paper and is available from the authors website.
\includepdf[pages={1-},scale=1]{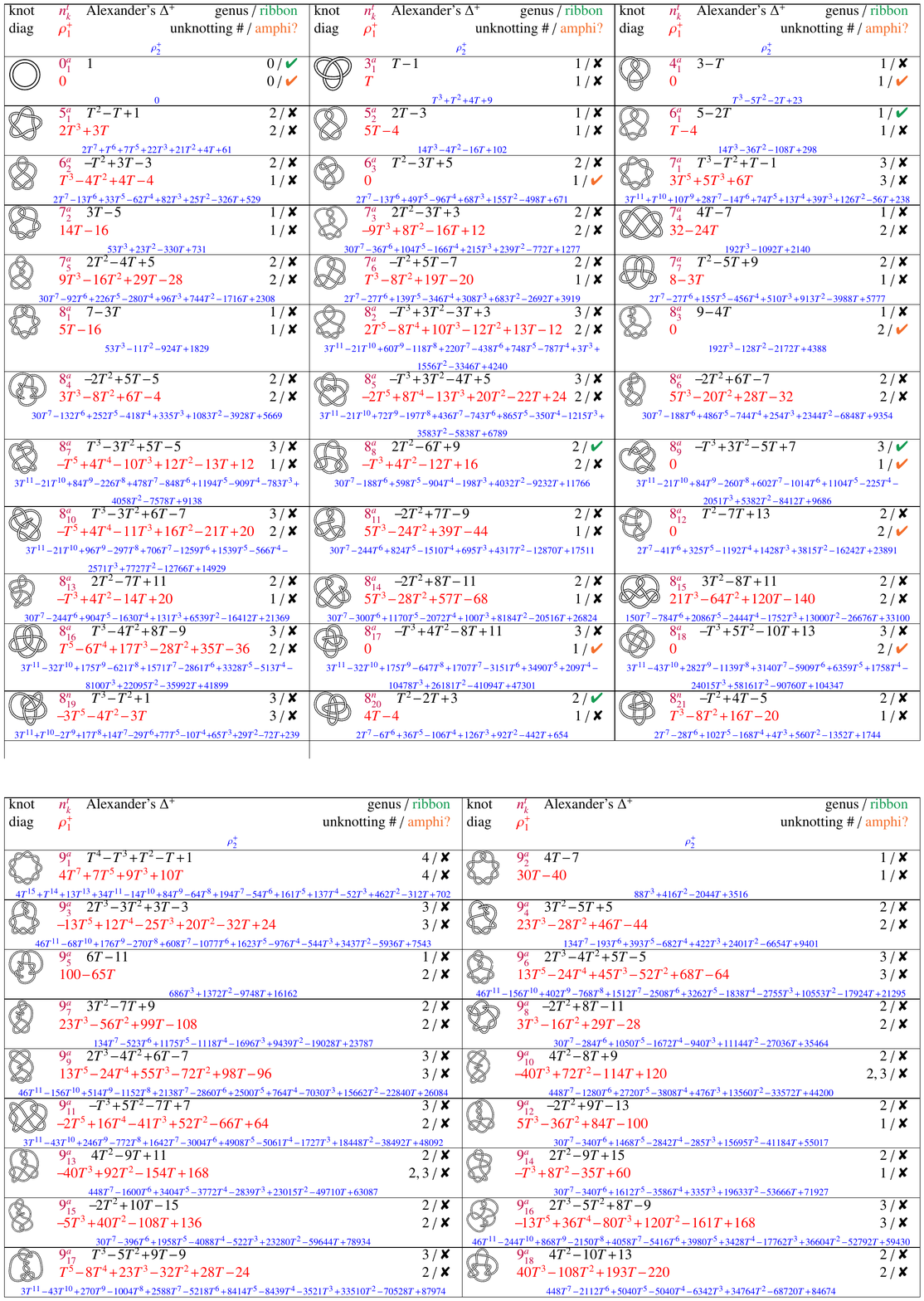}

\cleardoublepage

\phantomsection

\addcontentsline{toc}{section}{References}

\bibliographystyle{plain}
\bibliography{biblio}

\noindent{\sc Department of Mathematics, University of Toronto, Toronto Ontario M5S 2E4, Canada}\\
\emph{E-mail address:} \texttt{drorbn@math.toronto.edu}\\
\emph{URL:} \texttt{http://www.math.toronto.edu/drorbn}\\

\noindent{ \sc  Bernoulli Institute Mathematics, Groningen University, P.O. Box 407, 9700 AK Groningen, The Netherlands}\\
\emph{E-mail address:} \texttt{r.i.van.der.veen@rug.nl}\\
\emph{URL:} \texttt{http://www.rolandvdv.nl}\\

\end{document}